\documentclass[10pt, a4paper]{amsart}

\usepackage{enumerate, mathtools}
\usepackage{hyperref}
\hypersetup{
    colorlinks=true,
    linkcolor=blue,
    filecolor=blue,      
    urlcolor=blue,
    citecolor=blue,
}
\usepackage{amsmath}
\usepackage{amsthm}
\usepackage{amssymb}
\usepackage[matha,mathx]{mathabx} 
\usepackage{xcolor}
\usepackage{tikz}
\usepackage{tikz-cd}
\usepackage{subcaption}
\usepackage{url}
\usepackage{stmaryrd}
\usepackage{bbm}
\usepackage{appendix}
\usepackage{bm}
\usepackage{longtable}
\usepackage{tikz-qtree}

\usetikzlibrary{arrows}

\numberwithin{equation}{section}
\numberwithin{figure}{section}

\newtheorem{thmabc}{Theorem}
\newtheorem{conabc}[thmabc]{Conjecture}
\newtheorem{thm}{Theorem}[section]
\newtheorem*{thm*}{Theorem}
\newtheorem*{con*}{Conjecture}
\newtheorem{lem}[thm]{Lemma}
\newtheorem{prop}[thm]{Proposition}
\newtheorem{cor}[thm]{Corollary}
\newtheorem{lemma}[thm]{Lemma}

\theoremstyle{definition}
\newtheorem{defn}[thm]{Definition}
\newtheorem{remark}[thm]{Remark}
\newtheorem{ex}[thm]{Example}
\newtheorem{quest}[thm]{Question}

\newtheorem*{acknowledgements}{Acknowledgements}

\DeclareMathOperator{\Trel}{T}
\DeclareMathOperator{\Hilb}{Hilb}
\DeclareMathOperator{\sd}{sd}
\DeclareMathOperator{\des}{des}
\DeclareMathOperator{\rk}{rk}

\DeclareMathOperator{\Aut}{Aut}

\DeclareMathOperator{\rank}{rk}

\DeclareMathOperator{\codim}{rk}

\newcommand{\nofield}{K}
\newcommand{\field}{\mathbb{K}}
\newcommand{\fieldSR}{\mathbb{F}}
\newcommand{\bfr}{\bm{r}}
\newcommand{\bfs}{\bm{s}}
\newcommand{\bff}{\mathbf{f}}
\newcommand{\bfT}{\bm{T}}
\newcommand{\bfX}{\bm{X}}
\newcommand{\bfZ}{\bm{Z}}
\newcommand{\N}{\mathbb{N}}

\newcommand{\Q}{\mathbb{Q}}

\newcommand{\F}{\mathbb{F}}
\newcommand{\Z}{\mathbb{Z}}

\renewcommand{\phi}{\varphi}

\newcommand{\pbinom}[3][]{
    \genfrac{[}{]}{0pt}{}{#2}{#3}_{\ifthenelse{\isempty{#1}}{p}{#1}}
}
\renewcommand{\leq}{\leqslant}
\renewcommand{\geq}{\geqslant}
\renewcommand{\epsilon}{\varepsilon}
\newcommand{\gp}[1]{\frac{#1}{1-#1}}
\newcommand{\pb}[2][]{(\hspace{-0.9mm}( #2 )\hspace{-0.9mm})_{#1}}
\newcommand{\m}{\mathrm{Cgp}}
\newcommand{\1}{\mathbbm{1}}

\newcommand{\comp}{\mathrm{c}}
\newcommand{\norm}[1]{\left\| #1 \right\|}
\newcommand{\lri}{\mathfrak{o}}
\newcommand{\qvar}{Q}

\def \mcL {\mathcal{L}}
\def \Gri {\mathcal{O}}
\def \Ftwo {\mathbb{F}_2}

\def \Fq {\mathbb{F}_q}
\def \Z {\mathbb{Z}}

\def \Q {\mathbb{Q}}
\def \N {\mathbb{N}}
\def \mfp {\mathfrak{p}}

\newcommand{\mcP}{\mathcal{P}}
\newcommand{\mcN}{\mathcal{N}}
\newcommand{\fHP}{\mathsf{fHP}}
\newcommand{\cfHP}{\mathsf{cfHP}}
\newcommand{\mcF}{\mathcal{F}}
\newcommand{\mcU}{\mathcal{U}}
\newcommand{\A}{\mathcal{A}}
\newcommand{\intpos}{\mathcal{L}}
\newcommand{\toplat}{\widetilde{\intpos}}
\newcommand{\proplat}{\overline{\intpos}}
\newcommand{\GP}[1]{\mathsf{gp}\!\left(#1\right)}
\newcommand{\GPZ}[1]{\mathsf{gp}_0\hspace{-0.85mm}\left(#1\right)}

\newcommand{\atom}{\mathrm{at}}

\newcommand{\Graft}[2]{\mathcal{G}_{#1,#2}}
\newcommand{\TP}[1]{\mathrm{TP}_{#1}}
\newcommand{\TPnew}[2]{\mathrm{TP}_{\mathsf{#1},#2}}
\newcommand{\PT}[2]{\mathrm{PT}_{#1,#2}}

\newcommand{\LPTzero}[3]{\mathrm{LPT}^0_{\mathsf{#3},#1,#2}}
\newcommand{\LPT}[3]{\mathrm{LPT}_{\mathsf{#3},#1,#2}}
\newcommand{\LRTtwo}[2]{\mathrm{LRT}_{\mathsf{#1},#2}}
\newcommand{\LRT}[3]{\mathrm{LRT}_{\mathsf{#3},#1,#2}}
\newcommand{\Ome}[3]{\Omega_{\mathsf{#3},#1,#2}}
\newcommand{\Lam}[3]{\Lambda_{\mathsf{#3},#1,#2}}
\newcommand{\Ph}[3]{\Phi_{\mathsf{#3},#1,#2}}
\newcommand{\pitwo}[2]{\pi_{\mathsf{#1},#2}}
\newcommand{\Pitwo}[2]{\Pi_{\mathsf{#1},#2}}
\newcommand{\rhorep}[2]{\rho_{\mathsf{#1},#2}}
\newcommand{\UTP}[1]{\mathrm{UTP}_{#1}}
\newcommand{\picirc}{\pi^{\circ}}
\newcommand{\pibar}{\overline{\pi}^{\circ}}

\newcommand{\msfA}{\mathsf{A}} \newcommand{\msfX}{\mathsf{X}}
\newcommand{\msfI}{\mathsf{I}} \newcommand{\kind}{\mathsf{c}}
\newcommand{\nb}{\mathsf{dl}} \newcommand{\kb}{\mathsf{u}}

\newcommand{\treeCgp}{\mathsf{Cgp}}

\allowdisplaybreaks

\title[Flag Hilbert--Poincar\'e series of
  hyperplane arrangements]{Flag Hilbert--Poincar\'e series and Igusa
  zeta functions of hyperplane arrangements}

\author{Joshua Maglione} \address{Fakult\"at f\"ur Mathematik,
  Universit\"at Bielefeld, D-33501 Bielefeld, Germany}
  \curraddr{Fakult\"at f\"ur Mathematik,
  Otto von Guericke Universit\"at Magdebug, 39106 Magdeburg, Germany}
\email{joshua.maglione@ovgu.de}

\author{Christopher Voll} \address{Fakult\"at f\"ur Mathematik,
 Universit\"at Bielefeld, D-33501 Bielefeld, Germany}
\email{C.Voll.98@cantab.net}

\thanks{Supported by DFG-grants FR~1639/4-1 and VO~1248/4-1
  (project number~373111162).}


\begin{document}

\begin{abstract}
  We introduce and study a class of multivariate rational functions
  associated with hyperplane arrangements, called flag
  Hilbert--Poincar\'e series.  These series are intimately connected
  with Igusa local zeta functions of products of linear polynomials,
  and their motivic and topological relatives. Our main results
  include a self-reciprocity result for central arrangements defined
  over fields of characteristic zero. We also prove combinatorial
  formulae for a specialization of the flag Hilbert--Poincar\'e series
  for irreducible Coxeter arrangements of types $\mathsf{A}$,
  $\mathsf{B}$, and $\mathsf{D}$ in terms of total partitions of the
  respective types. We show that a different specialization of the
  flag Hilbert--Poincar\'e series, which we call the coarse flag
  Hilbert--Poincar\'e series, exhibits intriguing nonnegativity
  features and---in the case of Coxeter arrangements---connections
  with Eulerian polynomials. For numerous classes and examples of
  hyperplane arrangements, we determine their (coarse) flag
  Hilbert--Poincar\'e series. Some computations were aided by a
  SageMath package we developed.
\end{abstract}

\date{\today} \maketitle

\thispagestyle{empty}

\tableofcontents{}

\section{Introduction}
\label{sec:intro}

A hyperplane arrangement over a field $\field$ is a finite set $\A$ of
affine hyperplanes in $\field^d$ for some integer~$d=:\dim(\A)$. In
this paper we introduce and study a multivariate rational function
$\fHP_{\A}(Y,\bfT)\in\Q(\bfT)[Y]$, called the \emph{flag
  Hilbert--Poincar\'e series} of $\A$, encompassing much of the
topology and combinatorics of~$\A$.

In order to define $\fHP_{\A}$, we introduce some further
notation. Let $\intpos(\A)$ be the intersection poset of $\A$, ordered
by reverse-inclusion.  Two hyperplane arrangements are
\emph{equivalent} if their intersection posets are isomorphic. We
denote by $\hat{0}$ (resp.\ $\hat{1}$) the bottom (resp.\ top) element
of a poset (provided $\hat{1}$ exists). Observe that $\hat{1}\in
\intpos(\A)$ if and only if $\A$ is central,
i.e.~$\bigcap_{H\in\A}H\neq\varnothing$. Define $\toplat(\A) =
\intpos(\A)\setminus\{\hat{0}\}$ and $\proplat(\A) =
\intpos(\A)\setminus \{\hat{0},\hat{1}\}$. For $x\in\intpos(\A)$, we
write $\rank(x):=\rank_{\intpos(\A)}(x)$ for the \emph{rank} of~$x$,
viz.\ the supremum over the lengths of all chains from $\hat{0}$ to
$x$.  For a poset~$P$, the \emph{order complex} $\Delta(P)$ associated
with $P$ is the simplicial complex with vertex set $P$, whose
simplices are the flags of~$P$. For $x\in\intpos(\A)$, define
hyperplane arrangements
\begin{alignat*}{2}
    \A_x &= \left\{H\in\A ~\middle|~ x\subseteq H\right\},
    &&\textup{(subarrangement)}\\ \A^x &= \left\{x\cap H ~\middle|~
    H\in\A\setminus \A_x, \; x\cap H \neq \emptyset \right\} \quad
    &&\textup{(restriction)}.
\end{alignat*}
Set $\A_{\emptyset} := \A$. Interlacing these constructions we obtain, for
$x,y\in \intpos(\A)$, the arrangement $\A^x_{y} := (\A^x)_{y} =
(\A_{y})^x$. Recall further the \emph{Poincar\'e polynomial}
\begin{equation} \label{eqn:Poincare-poly}
    \pi_{\A}(Y) =\sum_{x\in \intpos(\A)} \mu(x)(-Y)^{\rk(x)} \in \Z[Y]
\end{equation}
associated with~$\A$, where $\mu$ is the M\"obius function
on~$\intpos(\A)$; cf.~\cite[Def.~2.48]{OrlikTerao/92}.  The Poincar\'e
polynomial is closely related to the characteristic polynomial
$\chi_{\A}(Y)$ of $\A$ via the identity (see
\cite[Def.~2.52]{OrlikTerao/92})
  \begin{equation}\label{equ:chi=cha}
    \chi_{\A}(Y) = Y^d\pi_{\A}(-Y^{-1}).
  \end{equation} We require the following
flag generalization: for $F = (x_1 < \cdots < x_{\ell})
\in\Delta(\intpos(\A))$ (possibly empty), set $x_0=\hat{0}$ and
$x_{\ell+1} = \emptyset$, and define
\begin{align}\label{eqn:Poincare-face}
   \pi_F (Y) = \prod_{k=0}^{\ell} \pi_{\A_{x_{k+1}}^{x_k}}(Y) \in
   \Z[Y].
\end{align} 

The following function is the main protagonist of the current paper.

\begin{defn}\label{def:algebraic}
    Let $\bfT := (T_x)_{x\in \toplat{(\A)}}$ be indeterminates. The \emph{flag
    Hilbert--Poincar\'e series} associated with $\A$ is
    \begin{align*} 
        \fHP_{\A}(Y,\bfT) &= \sum_{F\in\Delta(\toplat(\A))} \pi_F(Y)
        \prod_{x\in F}\dfrac{T_x}{1 - T_x} \in\Q(\bfT)[Y].
    \end{align*}
\end{defn} 
If $\A$ is central, then
$$ \fHP_{\A}(Y,\bfT) = \frac{1}{1-T_{\hat{1}}}
\sum_{F\in\Delta(\proplat(\A))}\pi_F(Y) \prod_{x\in F}\dfrac{T_x}{1 -
  T_x}.
$$
We remark that $\fHP_{\A}(0,\bfT)$ is the
(fine) Hilbert series of the Stanley--Reisner ring of the order
complex of $\toplat(\A)$; see Proposition~\ref{prop:coarse.Y=0}. 

The following self-reciprocity result for central arrangements over
fields of characteristic zero is our first main theorem. The
\emph{rank} of $\A$, denoted by $\rank(\A)$, is the rank of a maximal
element of $\intpos(\A)$.

\begin{thmabc}[Self-reciprocity]\label{thm:reciprocity.new}
  Let $\A$ be a central hyperplane arrangement over a field of
  characteristic zero. Then
    \begin{align} \label{eqn:reciprocity.new}
        \fHP_{\A}\left(Y^{-1}, (T_x^{-1})_{x\in\toplat(\A)}\right) &=
        (-Y)^{-\rank(\A)}T_{\hat{1}}\cdot \fHP_{\A}(Y,\bfT).
    \end{align}
\end{thmabc}

\begin{remark}
The restriction to fields of characteristic zero reflects our method
of proof (see Sections~\ref{sec:ilzf} and~\ref{sec:reci}) rather than any known
counterexamples in positive characteristic. Indeed, the Fano
arrangement, comprising the seven points in the projective plane
over~$\mathbb{F}_2$, has no equivalent arrangement over characteristic zero,
yet satisfies~\eqref{eqn:reciprocity.new}; see
Section~\ref{subsec:Fano}.

Simple examples show that the kind of \emph{self}-reciprocity expressed in
Theorem~\ref{thm:reciprocity.new} is not to be expected for noncentral
arrangements; see Section~\ref{subsec:m-para}. It remains of interest to
investigate the possibilities of reciprocity results linking
$\fHP_{\A}(Y^{-1},\bfT^{-1})$ to the flag Hilbert--Poincar\'e series of other
(``reciprocal'') hyperplane arrangements. One may also want to explore
self-reciprocity phenomena for (central) arrangements over fields of positive
characteristic and potential connections to reciprocity phenomena for other
rational generating functions such as the ones introduced
in~\cite[Sec.~5]{BMS/21}.
\end{remark}

Various substitutions of the variables of the flag Hilbert--Poincar\'e series
yield connections to seemingly different enumeration problems: first, we
explain in Section~\ref{subsec:Igusa.intro} that flag Hilbert--Poincar\'e
series encode the same information as certain $\mfp$-adic integrals associated
with hyperplane arrangements (see Theorem~\ref{thm:equivalence}). In
Section~\ref{subsec:Igusa-topological} we explicate the specific connections
to the well-studied class of (both uni- and multivariate) \emph{Igusa local
  zeta functions} associated with products of linear polynomials, and their
cousins, the \emph{topological zeta functions} (see
Corollary~\ref{cor:topological}).

Second, we discuss in Section~\ref{subsec:atom.braid} an alternative
combinatorial formula (see Theorem~\ref{thm:atom.braid}) for specific
multivariate substitutions, viz.\ \emph{atom zeta functions}, associated with
\emph{classical} Coxeter arrangements---viz.\ irreducible Coxeter arrangements
of types $\mathsf{A}$, $\mathsf{B}$, or~$\mathsf{D}$---in terms of total
partitions and rooted trees.

Third, we focus in Section~\ref{subsec:coarse.intro} on \emph{coarse flag
  Hilbert--Poincar\'e series}, viz.\ the bivariate ``coarsening'' of the flag
Hilbert--Poincar\'e series $\fHP_{\A}(Y,\bfT)$ obtained by setting $T_x=T$ for
all~$x$. Our Theorem~\ref{thm:coarse.Y=1} presents coarse flag
Hilbert--Poincar\'e series associated with Coxeter arrangements as
``$Y$-analogs'' of Hilbert series of the Stanley--Reisner rings of the first
barycentric subdivisions of standard simplices. On the level of rational
generating functions, this is reflected by an intriguing connection with
Eulerian polynomials.

\begin{remark}
  Another bivariate substitution relates the flag Hilbert--Poincar\'e series
  of a hyperplane arrangement $\A$ with the motivic zeta function
  $Z_{M(\A)}(Y,T)$ introduced in~\cite[Def.~1.1]{JKU/21} associated with the
  (representable) matroid $M(\A)$ determined by~$\A$. Indeed, we have
  \begin{align*} 
    \mathsf{fHP}_\A\left( -Y^{-1}, \left(Y^{-\mathrm{rk}(x)}T^{|\A_x|}\right)_{x\in\widetilde{\mathcal{L}}(\A)} \right) &= Z_{M(\A)}(Y, T).
\end{align*}
Our Theorem~\ref{thm:reciprocity.new} implies \cite[Thm.~1.6]{JKU/21} in this
case. See also Remark~\ref{rem:jku.cor1.8}.
\end{remark}

\subsection{Flag Hilbert--Poincar\'e series and $\mathfrak{p}$-adic integrals} 
\label{subsec:Igusa.intro}

For general arrangements $\A$ over fields of characteristic zero, the functions 
$\fHP_{\A}$ are universal objects from
which various $\mfp$-adic integrals associated with $\A$ may be
obtained via specializations. To discuss this connection, we first
recall some representability properties of hyperplane arrangements.

If $\nofield$ is a field and $\A_{\nofield}$ is a hyperplane
arrangement defined over $\nofield$ such that
$\intpos(\A)\cong\intpos(\A_{\nofield})$ as posets, then we say that
$\A$ is \emph{$\nofield$-representable} and call $\A_{\nofield}$ a
\emph{$\nofield$-representation} of $\A$.  If, as we now assume, $\A$
is a hyperplane arrangement defined over a field $\field$ of
characteristic zero, there exists a finite extension $\nofield$ of
$\mathbb{Q}$ such that $\A$ is $\nofield$-representable;
cf.~\cite[Prop.~6.8.11]{Oxley}. Having fixed such a representation
of~$\A$, we may further assume, without loss of generality, that each
$H\in \A$ is of the form $H=V(L)$, where $L(\bfX) =
c_L+\sum_{j=1}^d\alpha_{L,j}X_j\in\Gri_{\nofield}[\bfX]$ is an affine
linear polynomial over $\Gri_{\nofield}$, the ring of integers of the
number field~$\nofield$.  These choices allow us, in fact, to identify
the arrangement $\A$ with the collection of polynomials $L$ arising in
this way. We will use this freedom frequently.

In the sequel we denote by $\lri$ a compact discrete valuation ring (cDVR)
with an $\Gri_{\nofield}$-module structure. This could be a finite extension
of the completion $\Gri_{\nofield,\mfp}$ of $\Gri_{\nofield}$ at a nonzero
prime ideal $\mfp$ (in characteristic zero) or a power series ring of the form
$\Fq\llbracket X \rrbracket$, where $\Fq$ is the residue field of such a ring
(in positive characteristic).

Denoting by $\mfp$ the unique maximal ideal of $\lri$, we write
$\A(\lri/\mfp)$ for the reduction of $\A$ modulo~$\mfp$. If
$\mcL(\A)\cong \mcL(\A(\lri/\mfp))$, then $\A$ is said to have \emph{good
  reduction over~$\Fq$}, provided $\lri/\mfp$ has cardinality~$q$. It is
well-known that $\A$ has good reduction over $\Fq$ for all such $q$ not
divisible by finitely many (``bad'') primes; cf.\
\cite[Chap.~5.1]{Stanley/07}.

We now explain the connection between the flag Hilbert--Poincar\'e
series $\fHP_{\A}$ and various (multi- and univariate) $\mfp$-adic
integrals associated with the hyperplane arrangement~$\A$.

\begin{defn}\label{def:analytic.zf} 
    The \emph{analytic zeta function} of $\A$ over $\lri$ is
    \begin{align*}
        \zeta_{\A(\lri)}(\bfs) &= \int_{\lri^{\dim(\A)}}
        \prod_{x\in\toplat(\A)} \norm{\A_x}^{s_x} \, |\textup{d}\bfX|,
    \end{align*} 
    where $s_x$ is a complex variable for each $x\in\toplat(\A)$,
    further $\norm{\mathcal{X}} := \max\{|f| ~|~ f\in \mathcal{X}\}$
    for a finite set $\mathcal{X}\subset\lri$, and $|\textup{d}\bfX|$
    is the additive Haar measure on $\lri^{\dim(\A)}$, normalized so
    that~$\lri^{\dim(\A)}$ has measure~$1$.
\end{defn}

Our next main result establishes that the functions
$\zeta_{\A(\lri)}(\bfs)$ and $\fHP_{\A}(Y,\bfT)$ determine each other,
in the following precise sense.

\begin{thmabc}\label{thm:equivalence} 
    Let $\A$ be a hyperplane arrangement over a number
    field~$\nofield$. For indeterminates
    $\bfs:=(s_x)_{x\in\toplat(\A)}$ and
    $\bm{r}:=(r_x)_{x\in\toplat(\A)}$ and $x\in\toplat(\A)$, let
    \begin{align}\label{def:gxs} 
        g_x(\bfs) &= \codim(x) + \sum_{y\in \toplat(\A_x)}s_y,
       \\ h_x(\bm{r}) &= \sum_{y\in \toplat(\A_x)} (r_y -
        \codim(y))\mu(y, x).\label{def:hxr}
    \end{align} 
    If $\lri$ is a cDVR and an $\Gri_{\nofield}$-module with residue
    field cardinality $q$ such that $\A$ has good reduction
    over~$\Fq$, then
    \begin{align} 
        \zeta_{\A(\lri)}\left(\bfs\right) &=
        \fHP_{\A}\left(-q^{-1},\left(q^{-g_x(\bfs)}\right)_{x\in
          \toplat(\A)}\right), \label{eqn:equiv1}\\ \fHP_{\A}\left(-q^{-1},
        (q^{-r_x})_{x\in\toplat(\A)}\right) &=
        \zeta_{\A(\lri)}\left(\left(h_x(\bm{r})\right)_{x\in
          \toplat(\A)}\right).\label{eqn:equiv2}
    \end{align}
\end{thmabc}

A consequence of Theorem~\ref{thm:equivalence} is that
$\fHP_\A(Y,\bfT)$ provides an \emph{explicit}, combinatorial formula
for the multivariate rational function
$\zeta_{\A(\lri)}(\bfs)$. General formulae for $\mfp$-adic integrals
associated with polynomial mappings---not necessarily defined by
linear forms---are typically obtained via resolutions of
singularities (or log-principalizations) of the varieties these
mappings define; see, for instance, \cite{VeysZunigaGalindo/08} and
compare Section~\ref{subsec:funeq}. Our combinatorial approach
obliterates the need for these, in general quite unwieldy,
algebro-geometric tools and the choices they require.

We will prove Theorem~\ref{thm:equivalence} in
Section~\ref{sec:ilzf}. The interpretation of $\fHP_{\A}$ in terms of
the $\mfp$-adic integrals $\zeta_{\A(\lri)}$ expressed by
Theorem~\ref{thm:equivalence} is key to our proof of
Theorem~\ref{thm:reciprocity.new} in Section~\ref{sec:reci}.

We record numerous examples of the rational functions $\fHP_{\A}$ in
Section~\ref{sec:exa}. The case of Boolean arrangements $\A =
\msfA_1^n$ is of particular interest: specific substitutions of the
functions $\fHP_{\msfA_1^n}$ arise in the study \cite{RV:CICO} of the
average sizes of kernels of generic matrices with support constraints
over finite quotients of cDVRs; see Section~\ref{subsec:ask}.

\subsection{Igusa and topological zeta functions}
\label{subsec:Igusa-topological}

Assume now, as in Section~\ref{subsec:Igusa.intro}, that $\A$ is a
$\nofield$-representation and $\lri$ is a cDVR and an $\Gri_{\nofield}$-module. An important specialization of the multivariate function
$\zeta_{\A(\lri)}\left(\bfs\right)$ yields the (univariate)
\emph{Igusa local zeta function} (over~$\lri$) associated with the
product $f_{\A}(\bfX) := \prod_{L \in \A}L(\bfX)$ of linear polynomials
$L\in\Gri_{\nofield}[\bfX]$ (see \cite{Denef:Report}):
\begin{equation}\label{eqn:Igusa}
    \mathsf{Z}_{f_{\A},\lri}(s) := \zeta_{\A(\lri)}\left(
    \left(s\cdot\delta_{|\A_x|=1}\right)_{x\in\toplat(\A)}\right) =
    \int_{\lri^{\dim(\A)}}|f_{\A}|^s |\textup{d}\bfX|;
\end{equation}
here $s$ is a complex variable. Motivic zeta functions related with
such integrals have been studied and can be used to understand the
topological zeta function $\mathsf{Z}_{f_{\A}}^{\mathrm{top}}(s)$
associated with~$f_{\A}(\bfX)$. This is executed, for example,
in~\cite{budur2020contact} for generic central arrangements, among
others; cf.\ Section~\ref{subsec:gen.central.arr}.

As a consequence of Theorem~\ref{thm:equivalence}, we derive the
(multivariate and univariate) topological zeta function of a
hyperplane arrangement over a field of characteristic zero. From
Theorem~\ref{thm:equivalence}, $\fHP_{\A}(Y, \bfT)$ defines a system
of local zeta functions of Denef type in the sense of
Rossmann~\cite[Sec.~5]{Rossmann:top-zf}. Associated with $\A$ is thus
a unique rational function $\zeta^{\mathrm{top}}_{\A}(\bfs)\in
\Q(\bfs)$, interpretable as the limit ``$q\rightarrow 1$'' and called
the \emph{multivariate topological zeta function} of~$\A$. Analogous
to~\eqref{eqn:Igusa}, the topological zeta function
$\mathsf{Z}_{f_{\A}}^{\mathrm{top}}(s)$ is a specialization of
$\zeta^{\mathrm{top}}_{\A}(\bfs)$. For further details, see
Section~\ref{subsec:topological}.

For $F\in\Delta(\toplat(\A))$, set
\begin{align}\label{eqn:pi-circ}
    \picirc_{F}(Y) &:= \dfrac{\pi_F(Y)}{(1 +
      Y)^{|F|}}, &
    \pibar_{F}(Y) &:= \dfrac{\pi_F(Y)}{(1 +
    Y)^{|F|+1}}.
\end{align}
In fact, $\picirc_F(Y)$ is a polynomial in $Y$, and if $\A$ is central
and $F\in\Delta(\proplat(\A))$, then $\pibar_F(Y)$ is also a
polynomial in~$Y$; see Lemma~\ref{lem:pi-circ}.

\begin{cor}\label{cor:topological}
  Let $\A$ be a hyperplane arrangement over a field of characteristic zero,
  and let $\bfs = \left((s_x)_{x\in\toplat(\A)}\right)$ be indeterminates.
  Then the multivariate topological zeta function of $\A$ is
    \begin{align}\label{eqn:topo.multi} 
      \zeta_{\A}^{\mathrm{top}}(\bfs)
      &= \sum_{F\in\Delta(\widetilde{\mathcal{L}}(\mathcal{A}))}
        \picirc_{F}(-1) \prod_{x\in F}
        \dfrac{1}{\codim(x)+\sum_{y\in\toplat(\A_x)}s_y}.
    \end{align}
    If $\A$ is also central, then
    \begin{equation}\label{eqn:topo.multi.central}
      \zeta^{\mathrm{top}}_{\A}(\bfs) = \dfrac{1}{\rk(\A) + \sum_{y\in\toplat(\A)}s_y} \sum_{F\in\Delta(\proplat(\A))} \pibar_F(-1) \prod_{x\in F} \dfrac{1}{\rk(x) + \sum_{y\in\toplat(\A_x)}s_y}.
    \end{equation}
\end{cor}

We see in particular that if
$f\in\mathbb{C}[\bfX]$ is the product of linear polynomials, then the
topological zeta function $\mathsf{Z}_{f}^{\mathrm{top}}(s)$ depends
only on the combinatorics of the associated (multi-) arrangement; see
also \cite[Prop.~2.2]{BSY:zeta-b-functions}. This is also along the
lines and consistent with work of van der Veer, who proves that
topological zeta functions for general matroids are independent of the
underlying building set; see~\cite{vanderVeer}.

The zeta functions $\mathsf{Z}_f(s)$ and
$\mathsf{Z}_f^{\mathrm{top}}(s)$ lie at the heart of the Monodromy and
Topological Monodromy Conjectures~\cite{Igusa:Monodromy}, connecting
the singularities of $\{f=0\}$ with the local monodromy action on the
Milnor fibers of $f$ and relating the roots of the Bernstein--Sato
polynomial $b_f(s)$ with the (real parts of the) poles of
$\mathsf{Z}_f(s)$; see~\cite{Denef:Report} for a general
introduction. Budur, Musta\c{t}\u{a}, and Teitler~\cite{BMT:Monodromy}
proved that arbitrary complex hyperplane arrangements satisfy part of
the Topological Monodromy Conjecture and further reduced the remaining
part to the so-called Strong Monodromy
Conjecture~\cite[Conjecture~1.2]{BMT:Monodromy}, which is concerned
with the existence of a specific, combinatorially defined root
of~$b_f(s)$. The Strong Monodromy Conjecture is known to hold for some
classes of hyperplane
arrangements~\cite[Theorem~1.4]{BSY:zeta-b-functions}---notably for
Weyl arrangements~\cite[Theorem~1.1]{BW:Weyl-monodromy}---but it is,
in general, open. Wu recently proved a multivariate version of the
Topological Monodromy Conjecture for hyperplane arrangements;
cf.\ \cite[Thm.~1.7]{Wu/20}.

\subsection{Atom zeta functions}\label{subsec:atom.braid}

The specialization $\mathsf{Z}_{f_{\A},\lri}(s)$ defined
in~\eqref{eqn:Igusa} loses sight of all variables not corresponding to
atoms (i.e.\ minimal elements in $\widetilde{\mcL}(\A)$) and cannot
distinguish atoms. Just as the multivariate topological zeta
function~\eqref{eqn:topo.multi} refines the univariate topological
zeta function, we consider the following,
slightly more distinguishing $\mfp$-adic specialization
of~$\zeta_{\A(\lri)}(\bfs)$.

\begin{defn}\label{def:atom} 
    The \emph{atom zeta function} of $\A$ is
    \begin{equation*}
        \zeta_{\A(\lri)}^{\atom}\left((s_L)_{L\in\A}\right) =
        \zeta_{\A(\lri)}\left((s_x\cdot
        \delta_{|\A_x|=1})_{x\in\toplat(\A)}\right) =
        \int_{\lri^{\dim(\A)}} \prod_{L\in\A} |L|^{s_L}\,
        |\textup{d}\bfX|.
      \end{equation*}
\end{defn}
Here we identified atoms with elements $L\in\A$. We note that the
independent variables $(s_L)_{L\in\A}$ allow for the treatment of
multi-arrangements in the sense of~\cite{budur2020contact}.

We remark that the atom zeta function is the finest coarsening of the
multivariate zeta function $\zeta_{\A(\lri)}(\bfs)$ that is, in
general, multiplicative with respect to direct products of hyperplane
arrangements. Namely, if $\A$ and $\A'$ are arrangements of
hyperplanes in (disjoint vector spaces) $K^d$ and $K^{d'}$ and $\lri$
is as above, then, by Fubini's theorem,
$$\zeta^{\atom}_{(\A\times\A')(\lri)}((\bfs,\bfs')) =
\zeta^{\atom}_{\A(\lri)}(\bfs)\zeta^{\atom}_{\A'(\lri)}(\bfs').$$

Our next result paraphrases an explicit combinatorial formula for atom zeta
functions associated with classical Coxeter arrangements; see
Section~\ref{subsec:braid.total} for definitions and the precise version
(Theorem~\ref{thm:total-partition}). There we define, in particular, for
$n\in\N$, the sets $\TPnew{\mathsf{X}}{n}$ of total partitions of
type~$\mathsf{X}_n$; for type $\mathsf{X}=\mathsf{A}$, these are also defined
in~\cite[Example~5.2.5]{Stanley:Vol2} and related to Schr\"oder's fourth
problem.

\begin{thmabc}\label{thm:atom.braid}
  Let $\mathsf{X} \in \{\mathsf{A},\mathsf{B},\mathsf{D}\}$ and
  $n\in\N$, with $n\geq 2$ if~$\mathsf{X}=\mathsf{D}$. Then there
  exist, for all $\tau \in \TPnew{\mathsf{X}}{n}$, explicitly
  determined polynomials $\pitwo{X}{\tau}(Y)\in\Z[Y]$ and products of
  geometric progressions $\treeCgp_{\mathsf{X},\tau}\left( Z,
  (T_L)_{L\in\mathsf{X}_n}\right)$ such that the following holds: for
  all cDVR $\lri$ with residue field cardinality~$q$, assumed to be
  odd unless $\mathsf{X} = \mathsf{A}$,
    \begin{align*}
        \zeta_{\mathsf{X}_n(\lri)}^{\atom}\left((s_L)_{L\in\mathsf{X}_n}\right)
        &= \frac{1}{1-q^{-n-\sum_{L\in\mathsf{X}_n}s_L}}
        \sum_{\tau\in\TPnew{\mathsf{X}}{n}}\pitwo{X}{\tau}(-q^{-1})
        \treeCgp_{\mathsf{X},\tau}\left(q^{-1},(q^{-s_L})_{L\in
          \mathsf{X}_n}\right).
    \end{align*}
\end{thmabc}
  
One may use these formulae to obtain alternative combinatorial formulae for
the (multi- and univariate) topological zeta functions given in
Corollary~\ref{cor:topological}. We leave the details to the reader and refer to Example~\ref{ex:A3}.

As a consequence of Theorem~\ref{thm:atom.braid}, we obtain in
Corollary~\ref{cor:A_n-unlab} an explicit formula for Igusa's local
zeta function $\mathsf{Z}_{f_{\mathsf{A}_n},\lri}(s)$ in terms of
unlabeled rooted trees with $n+1$ leaves. In
Corollary~\ref{cor:B.typeA} we express the atom zeta
functions~$\zeta^{\textup{at}}_{\mathsf{B}_n(\lri)}(\bfs)$
and~$\zeta^{\textup{at}}_{\mathsf{D}_n(\lri)}(\bfs)$ as sums over
$\TPnew{A}{n}$, a set notably smaller than~$\TPnew{B}{n}$
and~$\TPnew{D}{n}$.

\subsection{Coarse flag Hilbert--Poincar\'e series}\label{subsec:coarse.intro}

Consider now the bivariate specialization of the flag
Hilbert--Poincar\'e series~$\fHP_{\A}$ obtained by setting $T_x = T$
for each $x\in \toplat(\A)$:
\begin{defn}\label{def:coarse.new}
    The \emph{coarse flag Hilbert--Poincar\'e series} of $\A$ is
    $$\cfHP_{\A}(Y,T) = 
    \sum_{F\in\Delta(\toplat(\A))} \pi_F(Y)
    \left(\dfrac{T}{1 - T}\right)^{|F|}\in\Q(T)[Y].$$
\end{defn}

We define the polynomial $\mathcal{N}_{\A}(Y, T)\in\Q[Y,T]$ by the
formula
\begin{equation}\label{def:numer.coarse}
    \cfHP_{\A}(Y,T) =
    \frac{\mathcal{N}_{\A}(Y,T)}{(1-T)^{\rank(\A)}}.
\end{equation}
In Section~\ref{sec:coarse} we explore a number of remarkable properties
of these rational functions, including nonnegativity features of
$\mcN_{\A}(Y,T)$ and---in the case of Coxeter
arrangements---connections with Eulerian and Stirling numbers.

In Proposition~\ref{prop:coarse.Y=0}, we observe that $\mcN_{\A}(0,T)$
has nonnegative coefficients. Its proof is based on the fact that
$\cfHP_{\A}(0,T)$ is the coarse Hilbert series of the Stanley--Reisner
ring of the order complex of $\toplat(\A)$. The Cohen--Macaulayness of
this complex implies the nonnegativity of the associated $h$-vector,
i.e.~the coefficients of~$\mathcal{N}_{\A}(0, T)$.

Recall that the $n$th \emph{Eulerian polynomial} $E_n(T)$ is defined via
\begin{equation}\label{def:euler.poly}
  E_n(T) = \sum_{w\in S_n}T^{\mathrm{des}(w)} \in \Z[T],
\end{equation}
where $\des(w) := |\{i\in [n-1] ~|~ w(i) > w(i+1)\}|$.  Let $S(n,k)$
be the Stirling number of the second kind; see
\cite[Sec.~1.9]{Stanley:Vol1}.  It is well-known that
\begin{equation}\label{def:hilbert.series}
  \frac{E_n(T)}{(1-T)^n} = \sum_{k=1}^{n}k!\;S(n,k)
  \left(\frac{T}{1-T}\right)^{k-1}
\end{equation}
is the (coarse) Hilbert series of the Stanley--Reisner ring
$\F[\sd(\partial \Delta_{n-1})]$ associated with the first barycentric
subdivision of the boundary of the $(n-1)$-dimensional simplex
$\Delta_{n-1}$ over a field~$\F$;
cf.\ \cite[Thm.~9.1]{Petersen/15}. 

A real hyperplane arrangement $\A$ is a \emph{Coxeter arrangement} if
the set of reflections across its hyperplanes fixes $\A$ and forms a
finite Coxeter group under composition. We call $\A$
\emph{irreducible} if it is not a direct product of two nontrivial
arrangements. Finite Coxeter arrangements may be decomposed as direct
products of {irreducible} Coxeter arrangements. The latter come in two
classes: \emph{classical} Coxeter arrangements of types $\mathsf{A}$,
$\mathsf{B}$, or $\mathsf{D}$ and \emph{exceptional} Coxeter
arrangements of types $\mathsf{E}_6$, $\mathsf{E}_7$, $\mathsf{E}_8$,
$\mathsf{F}_4$, $\mathsf{G}_2$, $\mathsf{H}_2$, $\mathsf{H}_3$,
$\mathsf{H}_4$, or $\mathsf{I}_{2}(m)$ for $m\geq 7$.

The following result shows that the coarse flag Hilbert--Poincar\'e series of
(most) Coxeter arrangements may be viewed as ``$Y$-analogs'' of the Hilbert
series~\eqref{def:hilbert.series}.

\begin{thmabc}\label{thm:coarse.Y=1}
    Let $\A$ be a Coxeter arrangement with no irreducible factor
    equivalent to~$\mathsf{E}_8$ and $\F$ be a field. Then
    \begin{equation}\label{eqn:Eulerian.coarse}
        \dfrac{\cfHP_{\A}(1, T)}{\pi_{\A}(1)} =
        \frac{E_{\rk(\A)}(T)}{(1-T)^{\rk(\A)}} = \Hilb(\F[\sd(\partial
          \Delta_{\rk(\A)-1})], T).\end{equation}
    In other words,
    $$\mathcal{N}_{\A}(1,T) = \pi_{\A}(1) E_{\rk(\A)}(T)$$ and,
    equivalently, for $1 \leq k \leq \rk(\A)$,
        \begin{align*} 
      \sum_{\substack{F\in\Delta(\proplat(\A)) \\ |F| = k-1}}
      \dfrac{\pi_F(1)}{\pi_{\A}(1)} &= k! \; S(\rk(\A),k).
    \end{align*} 
\end{thmabc}

The Stirling numbers of the second kind enter our proof of
Theorem~\ref{thm:coarse.Y=1} via a simple formula, essentially due to Cayley,
for the numbers of plane trees of a given length and number of leaves; see
Lemma~\ref{lem:plane-trees}.

The simple examples in Section~\ref{subsec:generic-cen.coarse} show
that the conclusion of Theorem~\ref{thm:coarse.Y=1} does not hold for
general, non-Coxeter hyperplane arrangements, even when they are
central: the coefficients of $\mathcal{N}_{\A}(1,T)$ are typically not
multiples of~$\pi_{\A}(1)$. However,
equation~\eqref{eqn:Eulerian.coarse} of Theorem~\ref{thm:coarse.Y=1}
holds for small-rank non-Coxeter restrictions of type-$\mathsf{D}$
arrangements; cf.~Appendix~\ref{sec:app.coarse.res-D}.

\begin{quest}
    Which further hyperplane arrangements
    satisfy~\eqref{eqn:Eulerian.coarse}? Which property of hyperplane
    arrangements does this equation reflect?
\end{quest}

To prove Theorem~\ref{thm:coarse.Y=1}, we first reduce to the irreducible
case by showing that coarse flag Hilbert--Poincar\'e series are,
essentially, Hadamard multiplicative; see
Proposition~\ref{prop:Hadamard}. For $\mathsf{X}\in\{\mathsf{A},
\mathsf{B}, \mathsf{D}\}$, the result is proven in
Section~\ref{subsec:setup.coarse}; for the types $\mathsf{I}_2(m)$ it
follows from Proposition~\ref{pro:Im}. We computed the coarse flag
Hilbert--Poincar\'e series of the other irreducible Coxeter
arrangement with the help of \textsf{HypIgu}~\cite{hypigu}, a
SageMath~\cite{sagemath} package developed by the first author to
compute (coarse) flag Hilbert--Poincar\'e series and other rational
functions associated with hyperplane arrangements. The results of
these computations, along with many other examples, are recorded in
Appendix~\ref{sec:app.exa.coarse}; in each case, the validity of
Theorem~\ref{thm:coarse.Y=1} follows by inspection. The type
$\mathsf{E}_8$ is excluded from Theorem~\ref{thm:coarse.Y=1} only
because we do not supply a proof nor an explicit computation.

All our computations support the following general nonnegativity conjecture.
\begin{conabc}\label{conj}
  For all hyperplane arrangements $\A$, the
  polynomial $\mathcal{N}_{\A}(Y,T)\in\Z[Y,T]$ has nonnegative
  coefficients.
\end{conabc}

Indeed, the polynomial $\mathcal{N}_{\A}(Y, T)$ has nonnegative coefficients
for all of the arrangements in our appendix. We furthermore view
Conjecture~\ref{conj} as an extension of the following observation, which uses
deep results from algebraic combinatorics. We note that $\pi_{\A}(Y)$ is the
Poincar\'e polynomial of a quotient of an exterior algebra, known as the
Orlik--Solomon algebra~\cite[Thm.~3.68]{OrlikTerao/92}.
\begin{prop} \label{prop:coarse.Y=0} For all hyperplane arrangements $\A$ we
  have $$\mathcal{N}_{\A}(Y, 0) = \pi_{\A}(Y)$$ and, for all fields
  $\mathbb{F}$,
  \begin{align*} 
    \dfrac{\cfHP_{\A}(0, T)}{1 - T} &=
    \Hilb(\mathbb{F}[\Delta(\intpos(\A))], T).
  \end{align*}
  In particular, the coefficients of both $\mathcal{N}_{\A}(Y, 0)$ and $\mathcal{N}_{\A}(0, T)$ are nonnegative.
\end{prop}

\subsection{Notation}

We let $\N$ be the set of positive integers. For $I\subseteq \N$, we
denote $I\cup\{0\}$ by $I_0$. For $n\in\N$, set $[n] := \{1,\dots,
n\}$. Let $\delta_P$ be $1$ when property $P$ is true and $0$
otherwise. For a set $I$, denote by $\mathcal{P}(I)$ the power set of
$I$ and by $\mathcal{P}(I; k)$ the set of subsets of $I$ of
cardinality~$k$. We set $\mathcal{P}(n) := \mathcal{P}([n])$ and
$\mathcal{P}(n; k):=\mathcal{P}([n]; k)$. We set $\GP{X}=X/(1-X)$ and
$\GPZ{X} = 1/(1-X)$.

Throughout, $\field$ denotes a field of characteristic zero and $\nofield$ a
number field with ring of integers $\Gri_{\nofield}$. We write $\lri$ for a
compact discrete valuation ring (cDVR) of arbitrary characteristic, often
assumed to be an $\Gri_\nofield$-module. Its residue field has cardinality $q$
and characteristic~$p$. The following table records some further frequently used
notation.

\begingroup \renewcommand\arraystretch{1.18}
\begin{longtable}{l|l|l}
    Symbol & Description & Reference \\ \hline \hline $\A$, $\A(\lri)$
    & hyperplane arrangement in $\field^d$ resp.\ in $\lri^d$&
    \S~\ref{sec:intro},~\ref{subsec:Igusa.intro} \\ $\intpos(\A)$ &
    intersection poset of $\A$& \S~\ref{sec:intro} \\ $\Delta(P)$ &
    order complex of a poset $P$& \S~\ref{sec:intro} \\ $\hat{0},
    \hat{1}$ & bottom resp.\ top elements of a poset&
    \S~\ref{sec:intro} \\ $\widetilde{P}$, $\overline{P}$ & posets
    $P\setminus\{\hat{0}\}$ resp.\ $P\setminus\{\hat{0},\hat{1}\}$&
    \S~\ref{sec:intro} \\ $\mathsf{X}_n$ & Coxeter arrangement of type
    $\mathsf{X}$ and rank $n$& Eq.~\eqref{eqn:classical-Coxeter}
    \\ $\pi_{\A}(Y)$ & Poincar\'e polynomial of $\A$ & Eq.~\eqref{eqn:Poincare-poly} \\ $\pi_{F}(Y)$ & flag Poincar\'e polynomial of $F$ & Eq.~\eqref{eqn:Poincare-face} \\ $\picirc_{F}(Y)$, $\pibar_{F}(Y)$
    & normalized Poincar\'e polynomials of $F$&
    Eq.~\eqref{eqn:pi-circ} \\ $\fHP_{\A}(Y, \bfT)$ & flag
    Hilbert--Poincar\'e series of $\A$&
    Def.~\ref{def:algebraic} \\ $\cfHP_{\A}(Y, T)$ & coarse flag
    Hilbert--Poincar\'e series of $\A$&
    Def.~\ref{def:coarse.new} \\ $\mathcal{N}_{\A}(Y, T)$ & numerator of
    $\cfHP_{\A}(Y, T)$ &
    Eq.~\eqref{def:numer.coarse} \\ $\zeta_{\A(\lri)}(\bfs)$ & analytic
    zeta function of $\A(\lri)$& Def.~\ref{def:analytic.zf}
    \\ $\zeta_{\A}^{\mathrm{top}}(\bfs)$ & multivariate topological zeta function
    of $\A$& Eq.~\eqref{eqn:multi-top-zf}
    \\ $\zeta_{\A(\lri)}^{\atom}(\bfs)$ & atom zeta function
    of $\A(\lri)$& Def.~\ref{def:atom} \\ $\mathsf{Z}_f(s)$ &
    Igusa zeta function of $f\in\Gri_{\nofield}[\bfX]$&
    Eq.~\eqref{eqn:Igusa}\\ $\mathsf{Z}_f^{\mathrm{top}}(s)$ &
    topological zeta function of $f\in\Gri_{\nofield}[\bfX]$ &
    \S~\ref{subsec:Igusa-topological} \\ $\mcP_\msfX(I)$,
    $\mcP_\msfX(I;2)$ & Coxeter versions of $\mcP(I)$
    resp.\ $\mcP(I;2)$& \S~\ref{subsec:braid.total} \\ $\Pitwo{X}{n}$
    & set partitions of type $\mathsf{X}_n$
    &\S~\ref{subsec:braid.total} \\ $\TPnew{\mathsf{X}}{n}$,
    $\TP{n+1}$ & total partitions of type $\mathsf{X}_n$
    resp.\ $\mathsf{A}_n$ & \S~\ref{subsec:TP} \\ $\UTP{n+1}$ &
    unlabeled analog of $\TP{n+1}$ &
    \S~\ref{subsubsec:A.B}\\ $\pitwo{X}{\tau}(Y)$ & Poincar\'e
    polynomial of $\mathsf{X}$-labeled tree $\tau$ &
    Eq.~\eqref{eqn:pi.X.tau} \\ $\treeCgp_{\mathsf{X}, \tau}(Z, \bfT)$
    & rational function of $\mathsf{X}$-labeled tree $\tau$ &
    Eq.~\eqref{eqn:Cgp}\\ $\PT{n}{k}$ & plane trees with $n$
    leaves and $k$ generations &
    Def.~\ref{def:PT.LPT.LRT}\\ $\LPT{n}{k}{X}$, $\LRT{n}{k}{X}$ &
    labeled plane resp.\ rooted trees of type $\mathsf{X}_n$ &
    Def.~\ref{def:PT.LPT.LRT}\\ $\LPTzero{n}{k}{X}$ & ``$0$-labeled''
    plane trees of type $\mathsf{X}_n$ &
    Def.~\ref{def:LPTzero}
\end{longtable}
\endgroup

\section{Analytic zeta functions of hyperplane arrangements}\label{sec:ilzf}

Recall the assumptions made on the hyperplane arrangement~$\A$ at the beginning
of Section~\ref{subsec:Igusa.intro}. In particular, we assume that $\A$ is
defined over a number field $\nofield$ and all cDVRs $\lri$ considered are
$\Gri_\nofield$-modules.  Recall further Definition~\ref{def:analytic.zf} of
the analytic zeta function $\zeta_{\A(\lri)}(\bfs)$ of $\A$ over~$\lri$, and
that we write $\Fq$ for the residue field of~$\lri$. Theorem~\ref{thm:equivalence} will follow from the next theorem.

\begin{thm}\label{thm:explicit-formulae.new}
    If $\A$ has good reduction over~$\Fq$, then
    \begin{align}\label{eqn:general-sum.new}
        \zeta_{\A(\lri)}(\bfs) &= \sum_{F\in\Delta(\toplat(\A))}
        \pi_{F}(-q^{-1}) \prod_{x\in F}
        \dfrac{q^{-\codim(x)-\sum_{y\in \toplat(\A_x)}s_y}}{1 -
          q^{-\codim(x)-\sum_{y\in \toplat(\A_x)}s_y}}.
    \end{align}
    If $\A$ is also central, then
    \begin{multline}\label{eqn:general-sum-central.new}
        \zeta_{\A(\lri)}(\bfs) =\\ \dfrac{1}{1 -
          q^{-\rank(\A)-\sum_{y\in
              \toplat(\A)}s_y}}\sum_{F\in\Delta(\proplat(\A))}
        \pi_{F}(-q^{-1}) \prod_{x\in F}
        \dfrac{q^{-\codim(x)-\sum_{y\in \toplat(\A_x)}s_y}}{1 -
          q^{-\codim(x)-\sum_{y\in \toplat(\A_x)}s_y}}.
    \end{multline}
\end{thm}

Theorem~\ref{thm:equivalence} is now an immediate consequence: its first
equation~\eqref{eqn:equiv1} follows from Theorem~\ref{thm:explicit-formulae.new}
and Definition~\ref{def:algebraic} of $\fHP_{\A}(Y,\bfT)$. The second
equation~\eqref{eqn:equiv2} is deduced from equation~\eqref{eqn:equiv1} by an
application of M\"obius inversion: setting $s_x = \sum_{y\in\toplat(\A_x)} (r_y
- \codim(y))\mu(y, x)$, we have $r_x = \codim(x) + \sum_{y\in\toplat(\A_x)}
s_y$.

In the remainder of this section, we prove
Theorem~\ref{thm:explicit-formulae.new}. For the proof we will need
the following lemma.

\begin{lem}\label{lem:gen-strat.new}
  If $\A$ has good reduction over~$\Fq$, then
  \begin{align}\label{eqn:rec}
    \zeta_{\A(\lri)}(\bfs) &= \sum_{x\in\intpos(\A)}
    q^{-\codim(x)-\sum_{y\in\toplat(\A_x)} s_y}
    \pi_{\A^x}(-q^{-1})
    \zeta_{\A_x(\lri)}\left((s_y)_{y\in\toplat(\A_x)}\right).
  \end{align}
  If $\A$ is also central, then
  \begin{multline}\label{eqn:central-recursive}
    \zeta_{\A(\lri)}(\bfs) = \dfrac{1}{1 -
      q^{-\rank(\A)-\sum_{x\in\toplat(\A)} s_x}}
    \sum_{x\in\intpos(\A)\setminus\{\hat{1}\}}
    q^{-\codim(x)-\sum_{y\in\toplat(\A_x)} s_y}
    \\ \quad\cdot \pi_{\A^x}(-q^{-1})
    \zeta_{\A_x(\lri)}\left((s_y)_{y\in\toplat(\A_x)}\right).
  \end{multline}
\end{lem}

\begin{proof}
    Suppose $\lri$ has maximal ideal $\mathfrak{p}$, and let
    $d=\dim(\A)$. For $x\in\toplat(\A)$, set 
    \begin{align*}
        U_x &= \left\{z\in\lri^d ~\middle|~ L(z) \in
        \mathfrak{p} \iff L \in \A_x \right\}
    \end{align*}
    and $U_{\hat{0}} = \lri^d\setminus\bigcup_{y\in\toplat(\A)} U_y$.
    Fix $x\in\intpos(\A)$. The number of vectors $v\in\mathbb{F}_q^d$
    such that $v+\mathfrak{p}^d\subseteq U_x$ is $\chi_{\A^x}(q)$,
    where $\chi_{\A^x}(Y)$ is the characteristic polynomial of $\A^x$;
    see \cite[Thm.~5.15]{Stanley/07}. Recall from~\eqref{equ:chi=cha}
    that $\chi_{\A^x}(Y) = Y^{d-\rk(x)}\pi_{\A^x}(-Y^{-1})$. Since
    $\A$ has good reduction over~$\Fq$,
    \begin{align*}
        \zeta_{\A(\lri)}(\bfs) &= \sum_{x\in\intpos(\A)} \int_{U_x} \prod_{y\in \toplat(\A)} \norm{\A_y}^{s_y} \, |\textup{d}\bfX| \\
        &= \sum_{x\in\intpos(\A)} q^{-d-\sum_{y\in\toplat(\A_x)}s_y} \chi_{\A^x}(q) \zeta_{\A_x(\lri)}\left((s_y)_{y\in\toplat(\A_x)}\right) \\
        &= \sum_{x\in\intpos(\A)} q^{-\codim(x)-\sum_{y\in\toplat(\A_x)}s_y} \pi_{\A^x}(-q^{-1}) \zeta_{\A_x(\lri)}\left((s_y)_{y\in\toplat(\A_x)}\right) . 
    \end{align*}
    We give some justification for the second equality: for each
    $x\in\intpos(\A)$, we choose $v_x\in\F_q^d$ such that
    $v_x+\mfp^d\subseteq U_x$ and apply a change of variables
    $X_i\mapsto v_{x,i} + \pi \widetilde{X}_i$, where $\pi$ is a
    uniformizer of~$\lri$.
    
    If $\A$ is also central, then $\hat{1}\in\intpos(\A)$, so
    \begin{align*} 
        \int_{U_{\hat{1}}} \prod_{y\in \toplat(\A)} \norm{\A_y}^{s_y}
        \, |\textup{d}\bfX| &= q^{-\rank(\A) - \sum_{y\in\toplat(\A)}
          s_y}\zeta_{\A(\lri)}(\bfs). 
    \end{align*}
    Hence
    \begin{align*}
      \zeta_{\A(\lri)}(\bfs) &= q^{-\rank(\A) - \sum_{y\in\toplat(\A)}
      s_y}\zeta_{\A(\lri)}(\bfs) + \\ &\quad \sum_{x\in\intpos(\A) \setminus \{\hat{1}\}} q^{-\codim(x)-\sum_{y\in\toplat(\A_x)}s_y} \pi_{\A^x}(-q^{-1}) \zeta_{\A_x(\lri)}\left((s_y)_{y\in\toplat(\A_x)}\right).
    \end{align*}
    Solving for $\zeta_{\A(\lri)}(\bfs)$ yields~\eqref{eqn:central-recursive}.
\end{proof}

In Section~\ref{subsec:proof.thm.braid.detail}, we will specify
Lemma~\ref{lem:gen-strat.new} combinatorially in case of the classical
Coxeter arrangements.

\subsection{Proof of Theorem~\ref{thm:equivalence}}

As explained above, it suffices to prove
Theorem~\ref{thm:explicit-formulae.new}. We start by
proving~\eqref{eqn:general-sum-central.new} and thus assume that $\A$
is central.

Observe that if $x\in\intpos(\A)$ and $y\in\intpos(\A_x)$, then $y\leq x$.
Thus, applying Lemma~\ref{lem:gen-strat.new} recursively yields a sum indexed
by a subset of $\Delta(\intpos(\A)\setminus\{\hat{1}\})$. With the only
exception of~$x=\hat{0}$, every term in the sum in
Lemma~\ref{lem:gen-strat.new} contains a 
$\zeta_{\A_x(\lri)}$-factor.  Hence, applying Lemma~\ref{lem:gen-strat.new}
recursively yields a sum indexed by the flags in
$\{ F \in \Delta(\intpos(\A))~|~ \hat{0}\in F,\; \hat{1}\notin F\}$, which is
in bijection with $\Delta(\proplat(\A))$.
    
Let $F=(x_1<\cdots <x_\ell)\in\Delta(\proplat(\A))$ be a (possibly
empty) flag, and set $G = (x_0<x_1<\cdots <x_{\ell})$ where
$x_0=\hat{0}$. We prove that the $F$-term
in~\eqref{eqn:general-sum-central.new} is the sum given by applying
Lemma~\ref{lem:gen-strat.new} $\ell+1$ times to $G$, starting with
$x_\ell$ and descending to~$x_0$.  By Lemma~\ref{lem:gen-strat.new},
the term associated with $y:= x_{\ell}$
in~\eqref{eqn:central-recursive} is
\begin{align*} 
    \dfrac{\pi_{\A^y}(-q^{-1})}{1 -
      q^{-\rank(\A)-\sum_{z\in\toplat(\A)} s_z}} q^{-\codim(y) -
      \sum_{z\in\toplat(\A_y)}
      s_z}\zeta_{\A_y(\lri)}\left((s_z)_{z\in\toplat(\A_y)}\right).
\end{align*}
If $\ell=0$, then this is indeed equal to the $F$-term. Thus, by
induction on~$\ell$, the $G$-term is
\begin{align*} 
    \dfrac{\prod_{k=0}^{\ell}\pi_{\A^{x_k}_{x_{k+1}}}(-q^{-1})}{1 -
      q^{-\rank(\A)-\sum_{z\in\toplat(\A)} s_z}} \prod_{y\in F}
    \dfrac{q^{-\codim(y) - \sum_{z\in\toplat(\A_y)}
        s_z}}{1-q^{-\codim(y) - \sum_{z\in\toplat(\A_y)} s_z}},
\end{align*} 
where $x_{\ell+1}=\emptyset$. Since this is the $F$-term
in~\eqref{eqn:general-sum-central.new},
equation~\eqref{eqn:general-sum-central.new} follows.

We proceed to the proof of~\eqref{eqn:general-sum.new}.  
For each $x\in \toplat(\A)$, the subarrangement $\A_x$ is central, so
by~\eqref{eqn:general-sum-central.new},
\begin{multline*} 
    \zeta_{\A_x(\lri)}(\bfs) =\\ \dfrac{1}{1 - q^{-\codim(x)-\sum_{y\in
          \toplat(\A_x)}s_y}}
    \sum_{F\in\Delta(\proplat(\A_x))} \pi_{F}(-q^{-1}) \prod_{y\in F}
    \dfrac{q^{-\codim(y)-\sum_{z\in \toplat(\A_y)}s_z}}{1 -
      q^{-\codim(y)-\sum_{z\in \toplat(\A_y)}s_z}}.
\end{multline*} 
Substituting this expression into~\eqref{eqn:rec} yields
\begin{align*} 
    \zeta_{\A(\lri)}(\bfs) &=
    \sum_{x\in\intpos(\A)}\dfrac{q^{-\codim(x)-\sum_{y\in\toplat(\A_x)}
        s_y} \pi_{\A^x}(-q^{-1})}{1 - q^{-\codim(x)-\sum_{y\in
          \toplat(\A_x)}s_y}}\\ &\quad \cdot
    \sum_{F\in\Delta(\proplat(\A_x))} \pi_{F}(-q^{-1}) \prod_{y\in F}
    \dfrac{q^{-\codim(y)-\sum_{z\in \toplat(\A_y)}s_z}}{1 -
      q^{-\codim(y)-\sum_{z\in \toplat(\A_y)}s_z}} \\ &=
    \sum_{F\in\Delta(\toplat(\A))} \pi_{F}(-q^{-1}) \prod_{x\in F}
    \dfrac{q^{-\codim(x)-\sum_{y\in \toplat(\A_x)}s_y}}{1 -
      q^{-\codim(x)-\sum_{y\in \toplat(\A_x)}s_y}}. \qedhere
\end{align*}
This concludes the proof of Theorem~\ref{thm:explicit-formulae.new} and thus
Theorem~\ref{thm:equivalence}. \hfill$\square$

\subsection{Topological zeta functions}
\label{subsec:topological}

Let $Q$ and $\bfs=(s_x)_{x\in\toplat(\A)}$ be indeterminates and
abbreviate $(\qvar^{-s_x})_{x\in \toplat(\A)}$ to~$\qvar^{-\bfs}$. By
expressing a rational function $W(Y, \bfT)\in\Q(Y,\bfT)$ as a power
series in $\Q(\bfs)\llbracket \qvar-1\rrbracket$, via $W(Y,
\bfT)\mapsto W(-\qvar^{-1}, \qvar^{-\bfs})$, one obtains a rational
function $W^{\mathrm{top}}(\bfs)\in \Q(\bfs)$ as the constant term of
the power series.  Informally speaking, this yields the limit of
$\zeta_{\A(\lri)}(\bfs)$ as ``$q\rightarrow 1$.'' For further details,
see Denef--Loeser~\cite[Sec.~2]{DL:top-zf} and
Rossmann~\cite[Sec.~5]{Rossmann:top-zf}.

For $a\in\Z[\bfs]$, note that $\qvar^{a} = (1 + (\qvar-1))^a =
\sum_{k=0}^\infty \binom{a}{k}(\qvar-1)^k$, where
$\binom{a}{k}=a(a-1)\cdots (a-k+1)/k!$. Thus, the constant term of
$\qvar^a\in \Q(\bfs)\llbracket \qvar-1 \rrbracket$ is $1$, and the
constant term of
\begin{align*} 
    \dfrac{\qvar-1}{\qvar^a-1} &=
    \left(\sum_{k=0}^\infty\binom{a}{k+1}(\qvar-1)^k\right)^{-1}
\end{align*}
is $1/a\in\Q(\bfs)$.

For a flag $F\in\Delta(\toplat(\A))$, recall
definitions~\eqref{eqn:pi-circ} of the rational
functions~$\picirc_F(Y)$ and~$\pibar_F(Y)$. The following lemma
records the observation that they are actually often polynomials.

\begin{lem}\label{lem:pi-circ}
    For a hyperplane arrangement $\A$ and $F\in\Delta(\toplat(\A))$,
    we have~$\picirc_{F}(Y)\in\Z[Y]$. If $\A$ is central and $F\in\Delta(\proplat(\A))$,
    then $\pibar_{F}(Y)\in\Z[Y]$.
\end{lem}

\begin{proof}
  If an arrangement $\A'$ is nonempty and central, then $\pi_{\A'}(Y)$
  is divisible by $1+Y$; see~\cite[Prop.~2.51]{OrlikTerao/92}. Suppose
  that $F=(x_1< \cdots <x_{\ell})\in\Delta(\toplat(\A))$, with
  $\ell\geq 0$.  Recalling that $x_0=\hat{0}$, we find that
  $$ \pi^{\circ}_{F}(Y) =
  \pi_{\A^{x_{\ell}}}(Y)\prod_{k=1}^{\ell}
  \frac{\pi_{\A_{x_{k}}^{x_{k-1}}}(Y)}{1+Y}.
  $$ For each $k\in [\ell]$, the arrangement $\A_{x_{k}}^{x_{k-1}}$ is
  central and nonempty; thus $\picirc_F(Y)\in\Z[Y]$. If $\A$ is
  central, then so is $\A^{x_{\ell}}$. Moreover $\A^{x_{\ell}}$ is
  nonempty if and only if $\hat{1}\notin F$. Hence,
  $\pibar_F(Y)\in\Z[Y]$ if~$F\in\Delta(\proplat(\A))$.
\end{proof}

We use the following notation in the next proof. For $f\geq 1$, let
$\lri^{(f)}$ be the (essentially unique) unramified degree-$f$
extension of~$\lri$. As in ~\cite[Def.~5.13]{Rossmann:top-zf}, we take
the following limit---well-defined for all cDVRs $\lri$ whose residue
characteristics avoid a finite set of rational primes, depending only
on $\A$---as the definition of the \emph{multivariate topological zeta
  function} associated with $\A$:
\begin{equation}\label{eqn:multi-top-zf}
  \zeta_{\A}^{\mathrm{top}}(\bfs) :=
        \lim_{\substack{f\rightarrow 0 \\ f\in\N}}
        \zeta_{\A(\lri^{(f)})}(\bfs).
\end{equation}

\begin{proof}[Proof of Corollary~\ref{cor:topological}]
  For $x\in\toplat(\A)$, let $g_x(\bfs)$ be as
  in~\eqref{def:gxs}. Then, by the  above,
  \begin{equation}\label{eqn:q-1.expansion}
      \begin{split}
      \fHP_{\A}\left(-\qvar^{-1},
      (\qvar^{-g_x(\bfs)})_{x\in\toplat(\A)}\right) &=
      \sum_{F\in\Delta(\toplat(\A))} \qvar^{-|F|}\pi_{F}^{\circ}(-\qvar^{-1}) \prod_{x\in
        F}\dfrac{\qvar-1}{\qvar^{g_x(\bfs)}-1}
      \\ &= \sum_{F\in\Delta(\toplat(\A))} \pi_{F}^{\circ}(-1)
      \prod_{x\in F}\dfrac{1}{g_x(\bfs)} + O((\qvar-1)).
      \end{split}
  \end{equation}
  By Theorem~\ref{thm:equivalence} and~\cite[Thm.~5.12]{Rossmann:top-zf},
  $\zeta_{\A}^{\mathrm{top}}(\bfs)$ is the constant term
  in~\eqref{eqn:q-1.expansion}:
  \begin{align*} 
      \zeta_{\A}^{\mathrm{top}}(\bfs) &= \lim_{\substack{f\rightarrow 0
          \\ f\in\N}}\fHP_{\A}\left(-q^{-f}, (q^{-f\cdot
        g_x(\bfs)})_{x\in\toplat(\A)}\right) =
      \sum_{F\in\Delta(\toplat(\A))} \pi_{F}^{\circ}(-1)
      \prod_{x\in F}\dfrac{1}{g_x(\bfs)}. 
  \end{align*} 
  If $\A$ is central, Lemma~\ref{lem:pi-circ} implies that
  $\pi_F^{\circ}(-1)=0$ for all $F\in\Delta(\proplat(\A))$. Hence
  \begin{align*} 
    \zeta^{\mathrm{top}}_{\A}(\bfs) &= \sum_{\substack{F\in\Delta(\toplat(\A)) \\ \hat{1}\in F}} \pi_{F}^{\circ}(-1)
    \prod_{x\in F}\dfrac{1}{g_x(\bfs)} = \dfrac{1}{g_{\hat{1}}(\bfs)} \sum_{F\in\Delta(\proplat(\A))} \overline{\pi}^{\circ}_F(-1) \prod_{x\in F} \dfrac{1}{g_x(\bfs)}. \qedhere
  \end{align*}
\end{proof}

\section{Self-reciprocity}\label{sec:reci}

In this section, we prove the self-reciprocity in
Theorem~\ref{thm:reciprocity.new} for $\fHP_{\A}(Y,\bfT)$ of a central
hyperplane arrangement~$\A$ over a field of characteristic zero---without loss
of generality, a number field $\nofield$ with ring of
integers~$\Gri_\nofield$. For this we first prove the corresponding result
(Corollary~\ref{cor:functional-eqn}) for the analytic zeta function
$\zeta_{\A(\lri)}(\bfs)$ over generic cDVRs $\lri$, and then we apply
Theorem~\ref{thm:equivalence}.

\subsection{Functional equations for multivariate Igusa zeta functions}\label{subsec:funeq}

Let $\mcL$ be a finite index set and $\mcF=(\bff_x)_{x\in\mcL}$ with
sets $\bff_x$ of polynomials over $\Gri_{\nofield}$ in $d$
variables~$\bfX=(X_1,\dots,X_d)$, each homogeneous of
degree~$d_x$. For a cDVR $\lri$ with $\Gri_{\nofield}$-module
structure, set
\begin{equation}\label{eqn:gen.int}
  Z_{\mcF(\lri)}(\bfs) := \int_{\lri^d} \prod_{x\in\mcL}
  \|\bff_x\|^{s_x} |\textup{d}\bfX|.
\end{equation}
As before, $|\textup{d}\bfX|$ denotes the
normalized additive Haar measure on~$\lri^d$. The next theorem follows
from well-known results.

\begin{thm}[Analytic self-reciprocity]\label{thm:funeq.mult}
  For such a family $\mathcal{F}$, there exists a finite set $S$ of
  primes such that, if $\lri$ is a cDVR and an $\Gri_\nofield$-module
  with residue field of cardinality $q$ and characteristic not in $S$,
  then the following holds:
  \begin{equation}\label{eqn:funeq.mult}
    \left. Z_{\mcF(\lri)}(\bfs)\right|_{q\rightarrow q^{-1}} =
    q^{-\sum_{x\in\mcL}d_xs_x} Z_{\mcF(\lri)}(\bfs).
    \end{equation}
\end{thm}

The proof of this local functional equation for a single homogeneous
polynomial (i.e.\ $\mcL=\{x\}$ and $\bff_x = \{f\}$) by Denef and
Meuser~\cite{DM/91} is based on an analysis of explicit formulae for
the $\mfp$-adic integral. In characteristic zero, the latter are
obtained from a (Hironaka) resolution of singularities of the
projective hypersurface defined by~$f$. Veys and Z\'u\~niga-Galindo
noted that these formulae and arguments extend to the case of
polynomial mappings (i.e.\ $|\mcL|=1$);
cf.~\cite{VeysZunigaGalindo/08}. The general case (i.e.~finite $\mcL$)
poses no additional conceptional difficulties; cf.~\cite{Voll/10}. At
the cost of discarding finitely many further residue class
characteristics, the transfer
principle~\cite[Thm.~9.2.4]{CluckersLoeser/10} implies that the same
formulae also hold in (sufficiently large) positive characteristic.

In general, the operation $q \rightarrow q^{-1}$ needs a precise
algebro-geometric definition; see, e.g.,
\cite[Rem.~1.7]{Voll/19}. If, however, there exists 
$W\in \Q(Y,(T_x)_{x\in\mcL})$ such that, for all cDVRs $\lri$ avoiding a
finite set of ``bad'' residue characteristics,
$$ Z_{\mcF(\lri)}(\bfs) =
W\left(-q^{-1},\left(q^{-s_x}\right)_{x\in\mcL}\right),$$ then the
inversion of $q$ amounts to the formal inversion of the variables $Y$
and $T_x$, for each $x\in\mcL$; cf.\ \cite[Sec.~4]{Rossmann/18}. This
is the case if $d_x=1$ for all $x\in\mcL$, as in the current
paper. The content of \eqref{eqn:funeq.mult} is that $W$ satisfies the
palindromic symmetry

$$W\left(Y^{-1},\left( T^{-1}_x\right)_{x\in\mcL}\right) = \left(\prod_{x\in
  L}T_x^{d_x}\right)W\left(Y,\left( T_x\right)_{x\in\mcL}\right).$$

\begin{cor}\label{cor:functional-eqn}
  Let $\A$ be a central hyperplane arrangement defined over a number
  field $\nofield$ and $\lri$ be a cDVR and an $\Gri_\nofield$-module
  with residue field cardinality $q$ such that $\A$ has good
  reduction over~$\Fq$. Then
  $$\left.\zeta_{\A(\lri)}(\bfs)\right|_{q\rightarrow q^{-1}} =
  q^{-\sum_{x\in \toplat(\A)}s_x} \zeta_{\A(\lri)}(\bfs).$$
\end{cor}

\begin{proof}
As $\A$ is central we may, without loss of generality, assume that
$\A$ consists of homogeneous polynomials. Therefore
$\zeta_{\A(\lri)}(\bfs)$ is of the form~\eqref{eqn:gen.int} with
$d=\dim(\A)$, $\mcL = \toplat(\A)$, $\bff_x=\A_x$, and $d_x=1$ for all
$x\in\toplat(\A)$. Instead of choosing an unspecified (and
uncontrollable) resolution of singularities of the hyperplane
arrangement $\A$, we observe that $\toplat(\A)$ is a prime example of
what Hu~\cite{Hu/03} calls a simple arrangement of smooth
subvarieties. For such arrangements a combinatorially defined chain of
blow-ups along the respective worst-intersection locus yields a
resolution of singularities~\cite[Thm.~1.1]{Hu/03}. This resolution
has good reduction over $\Fq$ if and only if $\A$ does. In this case,
the resulting formulae apply to cDVRs of characteristic zero and
positive characteristic alike.
\end{proof}

\subsection{Proof of Theorem~\ref{thm:reciprocity.new}}

Let $\A$, $\lri$, and $q$ be as in Corollary~\ref{cor:functional-eqn}. Recall
the definition~\eqref{def:hxr} of $h_x(\bm{r})$ for~$x\in\toplat(\A)$. By
M\"obius inversion we have $-\sum_{x\in \toplat(\A)}h_x(\bfr) = \codim(\hat{1})
- r_{\hat{1}} = \rank(\A) - r_{\hat{1}}$. Hence, by
Theorem~\ref{thm:equivalence}~\eqref{eqn:equiv2} and
Corollary~\ref{cor:functional-eqn},
\begin{align*} 
    \fHP_{\A}\left(-q,\left(q^{r_x}\right)_{x\in \toplat(\A)}\right) &=
    \left.\zeta_{\A(\lri)}\left(\left(h_x(\bfr)\right)_{x\in\toplat(\A)}\right)\right|_{q\rightarrow
      q^{-1}} \\ &= q^{-\sum_{x\in \toplat(\A)}h_x(\bfr)}
    \zeta_{\A(\lri)}\left(\left(h_x(\bfr)\right)_{x\in\toplat(\A)}\right)
    \\ &= q^{-\sum_{x\in \toplat(\A)}h_x(\bfr)}
    \fHP_{\A}\left(-q^{-1},\left(q^{-r_x}\right)_{x\in
      \toplat(\A)}\right)\\
    &=q^{\rank(\A) -
      r_{\hat{1}}}\fHP_{\A}\left(-q^{-1},\left(q^{-r_x}\right)_{x\in
      \toplat(\A)}\right).
\end{align*} 
As this equality holds for infinitely many~$q$, this proves Theorem~\ref{thm:reciprocity.new}. \hfill$\square$

\subsection{An alternative formulation}

Writing $\fHP_{\A}(Y,\bfT)$ over the common denominator $\prod_{x\in
  \toplat(\A)}(1-T_x)$ and comparing coefficients, we see that the
functional equation \eqref{eqn:reciprocity.new} is equivalent to the
following consequence of Theorem~\ref{thm:reciprocity.new}. For
$\mathcal{S}\subseteq \proplat(\A)$, we set $\mathcal{S}^\comp = \proplat(\A)\setminus \mathcal{S}$,
both viewed as subposets of $\proplat(\A)$.

\begin{cor}\label{cor:crunch-pt}
    For all $\mathcal{S}\subseteq \proplat(\A)$,
    \begin{align*} 
        \sum_{F\in\Delta(\mathcal{S})} (-1)^{|F|}\pi_F(Y) &=
        -(-Y)^{\rank(\A)}\sum_{F\in\Delta(\mathcal{S}^\comp)}
        (-1)^{|F|}\pi_F(Y^{-1}).
    \end{align*} 
\end{cor}

It is tempting to try and interpret this reciprocity result directly
in terms of the topology and combinatorics of~$\A$, bypassing the
analytic reciprocity result Theorem~\ref{thm:funeq.mult}. Such a proof
might help to relax the assumptions of
Theorem~\ref{thm:reciprocity.new} to include, perhaps, central
arrangements over fields of positive characteristic.

\begin{remark}\label{rem:jku.cor1.8}
  Corollary~\ref{cor:crunch-pt} for $\mathcal{S}=\proplat(\A)$ implies, in the
  case of a representable matroid, the identity in~\cite[Cor.~1.8]{JKU/21}.
\end{remark}

\section{Examples and applications}\label{sec:exa}

We discuss various examples and applications of flag Hilbert--Poincar\'e series.

\subsection{Arrangements of $m$ parallel lines}\label{subsec:m-para}

Let $m\in\N$ and consider the arrangement $\A = \{H_1,\dots, H_m\}
=\{X_2 - X_1 - k ~|~ k\in [m-1]_0\}$ of $m$ parallel lines
in~$\field^2$.  Note that if $m=1$, then $\A\cong \mathsf{A}_1$; if
$m\geq 2$, then $\A$ is not central. The flag Poincar\'e polynomials are
\begin{align*} 
  \pi_\varnothing(Y) &= 1 + mY, & \pi_{(H_i)}(Y) &= 1 + Y.
\end{align*}
With these data, the following is evident.

\begin{prop}
    For $m\in\N$,
    \begin{align*} 
        \fHP_{\A}(Y, \bfT) &= 1+mY + (1+Y)\sum_{i=1}^m \gp{T_i}, \\
        \zeta_{\A}^{\mathrm{top}}(\bfs) &= 1-m + \sum_{i=1}^m \dfrac{1}{1+s_i}.
    \end{align*}
\end{prop}

From the proposition, it follows that 
\begin{align}\label{eqn:non-reciprocal}
  \fHP_{\A}\left(Y^{-1}, (T_i^{-1})_{i=1}^m\right) &= m(1+Y^{-1}) - Y^{-1}\fHP_{\A}(Y, \bfT).
\end{align}
If $m=1$, then the right hand side of~\eqref{eqn:non-reciprocal} is
$-Y^{-1}T\fHP_{\A}(Y, T)$. If $m\geq 2$, however, then $\fHP_{\A}(Y,
\bfT)$ is not self-reciprocal.

\subsection{Arrangements of $m$ lines through the origin: $\mathsf{I}_2(m)$}\label{subsec:I2m}

Let $m\geq 2$ and $\zeta_m$ be a primitive $m$th root of unity. We consider the
Coxeter arrangement $\A=\{H_1,\dots,H_m\}$ of type~$\msfI_2(m)$, comprising $m$
distinct lines $H_i$ through the origin in~$K^2$, where $K=\Q(\zeta_m)$. Let
$\A$ be set of the $m$ linear factors of the polynomial $X_1^m+X_2^m\in
K[\bfX]$. The flag Poincar\'e polynomials are
\begin{align*}
    \pi_\varnothing(Y) &= (1+Y)(1+(m-1)Y), & \pi_{(H_i)}(Y) &=
    (1+Y)^2.
\end{align*}

With these data, it is easy to deduce the following.
\begin{prop}\label{pro:Im} For $m\geq 2$,
    \begin{align*} 
      \fHP_{\msfI_2(m)}(Y,\bfT) &= \frac{1 + Y}{1-T_{\hat{1}}} \left(1+(m-1)Y + (1+Y)\sum_{i=1}^m \gp{T_i}\right), \\
      \zeta_{\msfI_2(m)}^{\mathrm{top}}(\bfs) &= \dfrac{1}{2 + s_{\hat{1}} + s_1 + \cdots + s_m} \left(2-m + \sum_{i=1}^m \dfrac{1}{1+s_i}\right).
    \end{align*}
\end{prop}

\subsection{Shi arrangement $\mathcal{S}\mathsf{A}_2$}
\label{subsec:Shi-A2}

The Shi arrangement of type $\mathsf{A}_2$ is
\begin{align*} 
    \mathcal{S}\mathsf{A}_2 &= \{ X_i - X_j - \epsilon ~|~ 1\leq i < j
    \leq 3,\; \epsilon\in\{0,1\}\}.
\end{align*}
We write $H_{ij}$ and $H_{ij}^*$ for the affine subspaces determined
by $X_i-X_j$ and $X_i-X_j-1$ in $\field^3$, respectively. These six
(hyper-)planes intersect in six lines in $\field^3$. Three of these
lines are the intersection of three planes, three the intersections of
two planes. Let $\1$ be the linear subspace of $\field^3$ spanned by
$(1, 1, 1)$. The poset $\toplat(\mathcal{S}\mathsf{A}_2)$ is given in
Figure~\ref{fig:Shi-A_2}, and the flag Poincar\'e polynomials are
\begin{align*}
    \pi_{\emptyset}(Y) &= (1 + 3Y)^2, \\ \pi_{(H_{12})}(Y) =
    \pi_{(H_{23})}(Y) = \pi_{(H_{13}^*)}(Y) &= (1 + Y)(1 + 2Y),
    \\ \pi_{(H_{13})}(Y) = \pi_{(H_{12}^*)}(Y) = \pi_{(H_{23}^*)}(Y)
    &= (1 + Y)(1 + 3Y), \\ \pi_{(\1)}(Y) = \pi_{(\1 + e_1)}(Y) =
    \pi_{(\1 - e_3)}(Y) &= (1 + Y)(1 + 2Y), \\ \pi_{(\1 + e_2)}(Y) =
    \pi_{(\1 - e_2)}(Y) = \pi_{(\1 - e_2 - 2e_3)}(Y) &= (1 + Y)^2,
    \\ \pi_{F}(Y) &= (1 + Y)^2,
\end{align*}
where $F\in\Delta(\toplat(\mathcal{S}\mathsf{A}_2))$ is any maximal chain.
Referring to Figure~\ref{fig:Shi-A_2}, we label the planes (resp.~lines) from
left to right by the integers $\{1,\dots,6\}$ (resp.~$\{7,\dots, 12\}$). Recall
$\GP{X}$ and $\GPZ{X}$ are the respective geometric progressions
$\frac{X}{1-X}$ and $\frac{1}{1-X}$.

\begin{prop}\label{pro:shi.A3} We have
\begin{align*} 
    \fHP_{\mathcal{S}\mathsf{A}_2}(Y,\bfT) &= (1+3Y)^2 +
    (1+Y)(1+2Y)\Big( \GP{T_1} + \GP{T_3} + \GP{T_5} + \GP{T_7}
    \\ &\quad + \GP{T_9} + \GP{T_{11}}\Big) + (1 + Y)(1 + 3Y)\Big(
    \GP{T_2} + \GP{T_4} \\ &\quad + \GP{T_6}\Big) + (1 + Y)^2 \Big(
    \GP{T_8}\left(1 + \GP{T_2} + \GP{T_4}\right) \\ &\quad +
    \GP{T_{10}}\left(1 + \GP{T_2} + \GP{T_6}\right) +
    \GP{T_{12}}\left(1 + \GP{T_4} + \GP{T_6}\right) \\ &\quad +
    \GP{T_7}\left(\GP{T_1} + \GP{T_2} + \GP{T_3}\right) +
    \GP{T_9}(\GP{T_1} + \GP{T_4} \\ &\quad + \GP{T_5}) +
    \GP{T_{11}}\left(\GP{T_3} + \GP{T_5} + \GP{T_6}\right) \Big) ,
\end{align*}
\begin{align*} 
  \zeta_{\mathcal{S}\mathsf{A}_2}^{\mathrm{top}}(\bfs) &= 4 - \sum_{i=1}^6\GPZ{-s_i} - \sum_{i=1}^3\GPZ{-s_{2i}} \\
  &\quad + \GPZ{-s_1-s_2-s_3-s_7-1}\left(\GPZ{-s_1}+
  \GPZ{-s_2} + \GPZ{-s_3}-1\right)  \\ 
  &\quad + \GPZ{-s_1-s_4-s_5-s_9-1}\left(\GPZ{-s_1} +
  \GPZ{-s_4} + \GPZ{-s_5} -1\right) \\
  &\quad + \GPZ{-s_3-s_5-s_6-s_{11}-1}\left(\GPZ{-s_3} +
  \GPZ{-s_5} + \GPZ{-s_6}-1\right)  \\
  &\quad + \GPZ{-s_2-s_4-s_8-1}\left(\GPZ{-s_2} +
  \GPZ{-s_4}\right) \\
  &\quad + \GPZ{-s_2-s_6-s_{10}-1}\left(\GPZ{-s_2} +
  \GPZ{-s_6}\right) \\
  &\quad + \GPZ{-s_4-s_6-s_{12}-1}\left(\GPZ{-s_4} +
  \GPZ{-s_6}\right).
\end{align*}
\end{prop}

\begin{figure}[h]
    \centering
    \begin{tikzpicture}
        \pgfmathsetmacro{\y}{1.25}

        \node(H1) at (-5,\y) {$H_{12}$}; \node(H2) at (-3,\y)
             {$H_{13}$}; \node(H3) at (-1,\y) {$H_{23}$}; \node(H6) at
             (1,\y) {$H_{23}^*$}; \node(H5) at (3,\y) {$H_{13}^*$};
             \node(H4) at (5,\y) {$H_{12}^*$}; \node(I1) at
             (-5,2*\y+0.5) {$\1$}; \node(I3) at (-1,2*\y+0.5) {$\1 -
               e_3$}; \node(I4) at (1,2*\y+0.5) {$\1 - e_2$};
             \node(I2) at (3,2*\y+0.5) {$\1 + e_1$}; \node(I5) at
             (-3,2*\y+0.5) {$\1 + e_2$}; \node(I6) at (5,2*\y+0.5)
                  {$\1 - e_2 - 2e_3$};

        \draw[-stealth] (H1) -- (I1); \draw[-stealth] (H2) -- (I1);
        \draw[-stealth] (H3) -- (I1); \draw[-stealth] (H3) -- (I2);
        \draw[-stealth] (H4) -- (I2); \draw[-stealth] (H5) -- (I2);
        \draw[-stealth] (H1) -- (I3); \draw[-stealth] (H5) -- (I3);
        \draw[-stealth] (H6) -- (I3); \draw[-stealth] (H2) -- (I4);
        \draw[-stealth] (H4) -- (I4); \draw[-stealth] (H2) -- (I5);
        \draw[-stealth] (H6) -- (I5); \draw[-stealth] (H4) -- (I6);
        \draw[-stealth] (H6) -- (I6);
    \end{tikzpicture}
    \caption{The poset $\toplat(\mathcal{S}\mathsf{A}_2)$ for the Shi
      arrangement $\mathcal{S}\mathsf{A}_2$.}
    \label{fig:Shi-A_2}
\end{figure}

Formulae for the coarse flag Hilbert--Poincar\'e series associated with additional Shi
arrangements are given in Section~\ref{subsec:shi.coarse}.

\subsection{Braid arrangement~$\mathsf{A}_3$}
\label{subsec:A2}

The Coxeter (or {braid}) arrangement of type $\mathsf{A}_3$ is
\begin{align*} 
  \{X_i - X_j ~|~ 1\leq i < j \leq 4\}.
\end{align*}
We write $H_{ij} = V(X_i - X_j)$. For distinct $i,j,k,\ell\in[4]$ and
maximal flags $F$, the flag Poincar\'e polynomials are
\begin{align*} 
    \pi_{\emptyset}(Y) &= (1+Y)(1+2Y)(1 + 3Y), \\ \pi_{(H_{ij})}(Y) =
    \pi_{(H_{ij}\cap H_{ik})}(Y) &= (1+Y)^2(1+2Y), \\ \pi_{(H_{ij}\cap
      H_{k\ell})}(Y) = \pi_{F}(Y) &= (1+Y)^3.
\end{align*}
We associate to $H_{ij}$, $H_{ij}\cap H_{ik}$, and $H_{ij}\cap H_{k\ell}$ the
indeterminates $T_I$, $T_J$, and $T_{I|K}$ and $s_I$, $s_J$, and $s_{I|K}$
respectively, provided $I=\{i,j\}$, $J=\{i,j,k\}$, and $K=\{k,\ell\}$. 

\begin{prop}\label{prop:braid.A3}
\begin{align*} 
    (1-T_{\hat{1}})\fHP_{\mathsf{A}_3}(Y, \bfT) &= (1+Y)(1+2Y)(1+3Y) \\ &\quad +
    (1+Y)^2(1+2Y)\left(\sum_{I\in\mcP(4;2)} \GP{T_I} +
    \sum_{J\in\mcP(4;3)} \GP{T_J} \right) \\ &\quad + (1+Y)^3
    \left(\sum_{\substack{I,J\in \mcP(4;2) \\ I\cup J = [4]}}
    \GP{T_{I|J}}(1 + \GP{T_I} + \GP{T_J}) \right) \\ &\quad + (1 +
    Y)^3 \left(\sum_{I\in\mcP(4;3)}\sum_{J\in\mcP(I;2)}
    \GP{T_I}\GP{T_J} \right).
\end{align*}
With $$Q(\bfs)=3+s_{\hat{1}}+\sum_{I\in\mcP(4;2)\cup \mcP(4;3)}s_I+\sum_{\substack{I, J\in \mcP(4;2) \\ I\cup J=[4]}}s_{I|J},$$ 
\begin{align*} 
  Q(\bfs)\cdot\zeta_{\mathsf{A}_3}^{\mathrm{top}}(\bfs) &= 2 - \left(\sum_{I\in\mcP(4;2)} \GPZ{-s_I} +
  \sum_{J\in\mcP(4;3)} \GPZ{-1-s_J - \sum_{K\in \mcP(J;2)}s_K} \right) \\ &\quad + \left(\sum_{\substack{I,J\in \mcP(4;2) \\ I\cup J = [4]}}
  \GPZ{-1-s_{I|J}-s_I-s_J}(\GPZ{-s_I} + \GPZ{-s_J}) \right) \\ &\quad + \left(\sum_{I\in\mcP(4;3)}\sum_{J\in\mcP(I;2)}\GPZ{-1-s_I-\sum_{K\in\mcP(I;2)}s_K}\GPZ{-s_J} \right).
\end{align*}
\end{prop}
    
In Example~\ref{ex:A3} we obtain formulae for the Igusa and univariate
topological zeta functions associated with $\mathsf{A}_3$ in terms of
certain labeled rooted trees.

\subsection{Boolean arrangements}
\label{subsec:Boolean} 

Consider the arrangement $\A$ comprising all the coordinate
hyperplanes in $\field^n$, also known as the Boolean arrangement and
equivalent to~$\mathsf{A}_1^n$.  Since $\A_x^y$ is Boolean for all
$x,y\in\intpos(\A)$ with $y<x$, it follows that $\pi_F(Y)=(1+Y)^n$ for
all $F\in \Delta(\toplat(\A))$. Thus,
\begin{equation*}\label{eqn:alg.boolean}
    \fHP_{\mathsf{A}_1^n}(Y,\bfT) = (1+Y)^n \sum_{F\in
      \Delta(\toplat(\A))} \prod_{x\in F}\gp{T_x}.
\end{equation*}
We identify $\Delta(\toplat(\A))$ with $\widetilde{\textup{WO}}([n])$,
the poset of chains of nonempty subsets of $[n]$ and rewrite
$\fHP_{\A}(Y,(T_I)_{I\in \mcP(n)\setminus\{\varnothing\}})$ in terms of
the \emph{weak order zeta function} (cf.~\cite[Def.~2.9]{SV1/15}) 
\begin{equation}\label{eqn:WO}
  \mathcal{I}^{\textup{WO}}_n\left((T_I)_{I\in
    \mcP(n)\setminus\{\varnothing\}}\right) := \sum_{F\in
    \widetilde{\textup{WO}}([n])}\prod_{I\in F} \gp{T_I};
\end{equation} 

\begin{prop}\label{prop:boolean} 
    For $n\geq 1$,
    \begin{align*}
      \fHP_{\mathsf{A}_1^n}(Y,\bfT) &= (1+Y)^n \mathcal{I}^{\textup{WO}}_n(\bfT), \\
      \zeta_{\mathsf{A}_1^n}^{\mathrm{top}}(\bfs) &= \sum_{\substack{F\in
      \widetilde{\textup{WO}}([n]) \\ |F|=n}} \prod_{I\in F}\dfrac{1}{|I| + \sum_{J\in\mathcal{P}(I)\setminus\{\varnothing\}}s_J}.
    \end{align*}
\end{prop}

\begin{proof}
  By Lemma~\ref{lem:pi-circ}, $\pibar_F(-1)=0$ if and only if $F\in\widetilde{\textup{WO}}([n])$ is not maximal.
\end{proof}

The neat factorization of $\fHP_{\mathsf{A}_1^n}(Y,\bfT)$ as a product
of a polynomial in $Y$ and a rational function in $\bfT$ seems to be
atypical for these series. We have not observed such a factorization
anywhere outside the family of Boolean arrangements.

\subsection{Generic central arrangements}\label{subsec:gen.central.arr}

Let $m,n\in\N$ with $n\leq m$. We consider the arrangement
$\mathcal{U}_{n,m}$ of $m$ hyperplanes through the origin in
$\field^n$ in general position. That is, for each $k\leq n$, every
$k$-set of hyperplanes intersects in a codimension-$k$ subspace. This
is also known as the $m$-element uniform matroid of rank~$n$. Observe
that $\mathcal{U}_{2, m} \cong \mathsf{I}_2(m)$, covered in
Section~\ref{subsec:I2m}, and $\mathcal{U}_{n, n}$ is Boolean, seen in
Section~\ref{subsec:Boolean}.

\begin{lem}\label{lem:Poincare-uniform}
    Let $m,n\in \N$ with $n\leq m$. The Poincar\'e polynomial of
    $\mathcal{U}_{m,n}(Y)$ is
    \begin{align}\label{eqn:Poincare-uniform}
        \pi_{\mathcal{U}_{n, m}}(Y) &= (1 + Y)
        \sum_{k=0}^{n-1}\binom{m-1}{k}Y^k.
    \end{align}
\end{lem}

\begin{proof}
  If $n=m$, then $\pi_{\mathcal{U}_{n, n}}(Y) = (1+Y)^n$, and the
  lemma follows in this case, so we assume $m>n$.  We proceed by
  induction on $n$, with base case
  $\mathcal{U}_{1,m}\cong\mathcal{U}_{1,1}$.  Let
  $H\in\mathcal{U}_{n,m}$, so $\mathcal{U}_{n, m}^H \cong
  \mathcal{U}_{n-1, m-1}$ and $\mathcal{U}_{n, m}\setminus \{H\} \cong
  \mathcal{U}_{n, m-1}$. Hence by the Deletion-Restriction
  Theorem~\cite[Prop.~2.56]{OrlikTerao/92},
    \begin{align}\label{eqn:uniform-DRT}
      \pi_{\mathcal{U}_{n, m}}(Y) &= \pi_{\mathcal{U}_{n, m-1}}(Y) +
                                    Y\pi_{\mathcal{U}_{n-1, m-1}}(Y).
    \end{align}
    By induction, the lemma follows. 
\end{proof}

We augment our power-set notation to accommodate multiple, prescribed
subset sizes. For $I,J\subset \N_0$, we define $\mcP(I;J) =
\bigcup_{j\in J} \mcP(I; j) \subseteq \mcP(I)$. Observe that for
$n\leq m$ we have $\proplat(\mathcal{U}_{n,m})\cong \mathcal{P}([m];
       [n-1])$ and $\toplat(\mathcal{U}_{n,n}) \cong \mathcal{P}([n];
       [n])$.

The next proposition generalizes
Proposition~\ref{prop:boolean}, as~$\mathsf{A}_1^n \cong
\mathcal{U}_{n,n}$, and Proposition~\ref{pro:Im}, since $\mathsf{I}_2(m) \cong \mathcal{U}_{2,m}$.

\begin{prop}\label{prop:uniform-zeta}
Let $m,n\in\N$ with $n\leq m$. For $\bfT=(T_{\hat{1}}, (T_I)_{I\in
\mcP([m]; [n-1])})$ and $Q(\bfs) = n +
s_{\hat{1}}+\sum_{J\in\mathcal{P}([m];[n-1])}s_J$, 
    \begin{align*} 
        \fHP_{\mathcal{U}_{n,m}}\left(Y, \bfT\right) &= \dfrac{1 + Y}{1 - T_{\hat{1}}}\sum_{k=0}^{n-1}\binom{m-1}{k}Y^k +
        \sum_{\ell=1}^{n-1}\sum_{k=0}^{n-\ell-1} \dfrac{(1+Y)^{\ell+1}}{1 - T_{\hat{1}}} \binom{m-\ell-1}{k}Y^k \\
        &\quad \cdot \sum_{I\in \mcP([m]; \ell)}  T_I
        \mathcal{I}^{\textup{WO}}_\ell\left((T_J)_{J\in \mcP(I;
        [\ell])}\right) , \\
         \zeta_{\mathcal{U}_{n,m}}^{\mathrm{top}}(\bfs) \cdot Q(\bfs) &= \sum_{k=0}^{n-1}(-1)^k\binom{m-1}{k} \\
        &\quad + \sum_{\ell=1}^{n-1}\sum_{I\in\mathcal{P}([m];\ell)} \sum_{\substack{F\in\widetilde{\mathrm{WO}}(I) \\ |F|=\ell}} \dfrac{\sum_{k=0}^{n-\ell-1}(-1)^k\binom{m-\ell-1}{k}}{\prod_{J\in F}\left(|J| + \sum_{K\in\mathcal{P}(J)\setminus\{\varnothing\}}s_K\right)}.
    \end{align*} 
\end{prop}

\begin{proof}
  For all $x\in\proplat(\mathcal{U}_{n, m})$ with $\ell = \codim(x)$, we find
  that $(\mathcal{U}_{n, m})_x\cong \mathsf{A}_1^{\ell}$ and $\mathcal{U}_{n,
  m}^x\cong \mathcal{U}_{n-\ell, m-\ell}$. Apply Theorem~\ref{thm:equivalence},
  Lemmas~\ref{lem:gen-strat.new} and~\ref{lem:Poincare-uniform}, and
  Proposition~\ref{prop:boolean}, yielding the flag Hilbert--Poincar\'e
  series. For the topological zeta function, let $F=(x_1 < \cdots < x_{\ell})\in
  \Delta(\proplat(\mathcal{U}_{n,m}))$ with $\ell\geq 1$, and assume
  $\rk(x_{\ell})=r\in[n-1]$. If $r > \ell$, then $(x_1 < \cdots < x_{\ell-1})$
  is not a maximal flag in $\mathcal{U}_{n,m}^{x_{\ell}}\cong\mathsf{A}_1^r$.
  Hence, $\pibar_F(-1)=0$. Thus, we only consider flags $(x_1<\cdots <x_{\ell})$ such that 
  $\rk(x_i)=i$ for all $i\in [\ell]$. Applying Corollary~\ref{cor:topological}, Proposition~\ref{prop:boolean}, and Lemma~\ref{lem:Poincare-uniform}, the proposition follows.
\end{proof}

\begin{remark}
  Motivic zeta functions associated with generic (central) hyperplane
  arrangements are a focus of \cite[Sec.~5
    and~6]{budur2020contact}. We consider that the combinatorial
  considerations there may yield descriptions of the
  \emph{coefficients} of (specific substitutions of) the flag
  Hilbert--Poincar\'e series $\fHP_{\mathcal{U}_{n,m}}$ described in
  Proposition~\ref{prop:uniform-zeta}.
\end{remark}

\subsection{The Fano plane -- an example in characteristic two}
\label{subsec:Fano}

Let $\A$ be the Fano arrangement, comprising the seven (hyper-)planes
in~$\Ftwo^3$. We compute 
$\fHP_{\A}(Y, \bfT)$ and verify that it
satisfies~\eqref{eqn:reciprocity.new} of
Theorem~\ref{thm:reciprocity.new}, despite the facts that $\A$ is
\emph{a priori} an arrangement in characteristic two and has \emph{a
  posteriori} no equivalent arrangement in characteristic zero;
see~\cite[Prop.~6.4.8]{Oxley}.

For $x\in\proplat(\A)$ and maximal flags
$F\in\Delta(\proplat(\A))$, the flag Poincar\'e polynomials are
\begin{align*} 
    \pi_{\emptyset}(Y) &= (1+Y)(1+2Y)(1+4Y), \\ \pi_{(x)}(Y) &=
    (1+Y)^2(1+2Y), \\ \pi_F(Y) &= (1+Y)^3.
\end{align*}
Define sets of integers:
\begin{align*} 
  \mathcal{X}_8 &= \{1, 3, 5\}, & \mathcal{X}_9 &= \{1, 4, 6\}, & \mathcal{X}_{10} &= \{1, 2, 7\}, & \mathcal{X}_{11} &= \{3, 4, 7\}, \\
  \mathcal{X}_{12} &= \{5, 6, 7\}, & \mathcal{X}_{13} &= \{2, 4, 5\}, & \mathcal{X}_{14} &= \{2, 3, 6\}.
\end{align*}
Thus, for an appropriate labeling of the seven planes by
$\{1,\dots,7\}$ and the seven lines by~$\{8,\dots, 14\}$, we obtain
the following. 
\begin{prop} 
  For $\bfT= (T_{\hat{1}},T_1,\dots,T_{14})$ and
  $Q(\bfs)=3+s_{\hat{1}}+s_1+\cdots +s_{14}$,
  \begin{align*} 
      (1 - T_{\hat{1}}) \fHP_{\A}(Y, \bfT) &= 
      (1+Y)(1+2Y)(1+4Y) + (1+Y)^2(1+2Y)\sum_{i=1}^{14} \GP{T_i} \\ 
      &\quad + (1+Y)^3 \sum_{i=8}^{14} \sum_{j\in \mathcal{X}_i}\GP{T_i}\GP{T_j},
  \end{align*}
  \vspace{-1em}
  \begin{align*}
      Q(\bfs)\zeta_{\A}^{\mathrm{top}}(\bfs) &= 3 - \sum_{i=1}^7\GPZ{-s_i} \\
      &\quad + \sum_{i=8}^{14} \GPZ{-1-s_i-\sum_{j\in\mathcal{X}_i}s_j} \left(-1+\sum_{j\in\mathcal{X}_i} \GPZ{-s_j}\right)
  \end{align*}
\end{prop}
One may verify with a computer algebra system, e.g.\ Maple\footnote{Maple is a
  trademark of Waterloo Maple Inc.}, that
\begin{align*} 
    \fHP_{\A}(Y^{-1}, T_{\hat{1}}^{-1}, T_1^{-1}, \dots, T_{14}^{-1}) &=
    (-Y)^{-3}T_{\hat{1}} \fHP_{\A}(Y, \bfT).
\end{align*}

\subsection{An application to ask zeta functions}\label{subsec:ask}

We briefly explain a connection between the zeta functions we
associate with Boolean arrangements (see Section~\ref{subsec:Boolean})
and zeta functions associated with hypergraphs introduced and studied
in~\cite{RV:CICO}.

Hypergraphs are used to parametrize modules of
$n\times m$-matrices satisfying combinatorially defined support
constraints. More precisely, if $\mathsf{H}$ is a hypergraph on $[n]$
with hyperedges $e_1,\dots, e_m$, then let $M_\mathsf{H}\subseteq \mathrm{M}_{n\times m}(\Z)$
be the module of integral matrices $[a_{ij}]$ with $a_{ij}=0$ whenever
vertex $i$ and hyperedge $e_j$ are not incident.  A hypergraph is uniquely determined by its
\emph{hyperedge multiplicities} $\bm{\mu} = (\mu_I)_{I\subseteq [n]}
\in \N_0^{\mathcal{P}(n)}$, in particular, $m=\sum_{I\subseteq
  [n]}\mu_I$. For an arbitrary cDVR
$\lri$ with maximal ideal~$\mfp$ of index $q$ and
$k\in\N_0$, we write $M_\mathsf{H}(\lri/\mfp^k) := M_\mathsf{H}\otimes
\lri/\mfp^k$ and $M_\mathsf{H}(\lri) := M_\mathsf{H}\otimes
\lri$. Informally speaking, these are the modules of
$n\times m$-matrices over the respective rings satisfying the support
constraints defining $M_\mathsf{H}$. 

The \emph{average size of the kernels} in $M_\mathsf{H}(\lri/\mfp^k)$
is
\begin{align*} 
    \mathrm{ask}(M_\mathsf{H}(\lri/\mfp^k))&= \dfrac{1}{|M_\mathsf{H}(\lri/\mfp^k)|}\sum_{a\in M_\mathsf{H}(\lri/\mfp^k)}
    |\ker(a)|.
\end{align*}
The \emph{(analytic) ask zeta function of $\mathsf{H}$} over $\lri$ is
\begin{align*} 
    \zeta^{\mathsf{ask}}_{M_\mathsf{H}(\lri)}(s) &= \sum_{k=0}^\infty
    \mathrm{ask}(M_\mathsf{H}(\lri/\mfp^k)) q^{-ks}.
\end{align*} 
(In the notation of \cite[Sec.~3.2]{RV:CICO}, the function
$\zeta^{\mathsf{ask}}_{M_\mathsf{H}(\lri)}(s)$ coincides with the ask
zeta function $\zeta_{\eta^{\lri}}^{\mathsf{ask}}(s)$ associated with
the incidence representation $\eta^\lri$ of $\mathsf{H}$ over $\lri$.)

In ~\cite[Thm.~A]{RV:CICO}, Rossmann and the second author prove the
existence of a bivariate rational function $W_{\mathsf{H}}(X, T)\in
\Q(X, T)$ such that~$\zeta^{\mathsf{ask}}_{M_\mathsf{H}(\lri)}(s) =
W_{\mathsf{H}}(q, q^{-s})$ for all $\lri$ with residue field
cardinality $q$. This rational function is given in terms of weak
orders on~$[n]$; cf.~\cite[Thm.~C]{RV:CICO}.  We show that these
rational functions may be expressed as substitutions of analytic zeta
functions associated with Boolean arrangements, which are determined
by Proposition~\ref{prop:boolean}.

Since ${\intpos}(\mathsf{A}_1^{n+1})$ is isomorphic to the subset
lattice of $[n+1]$, we index the indeterminates of
$\zeta_{\mathsf{A}_1^{n+1}}$ in the next proposition by the
nonempty subsets of~$[n+1]$. We further define polynomials for each $J
\subseteq [n+1]$,

$$ f_{n,J}\left(X, Y, (Z_I)_{I\subseteq [n]}\right) := \begin{cases} X
  + Y - Z_\emptyset - n - 1 & \textup{ if }J = \{n+1\},
  \\ -Z_{J\setminus\{n+1\}} & \textup{ if }\{n+1\} \subsetneq J, \\ 0
  & \text{ otherwise}.
   \end{cases}
$$

\begin{prop}\label{prop:ask}
  Let $\mathsf{H}$ be a hypergraph on $[n]$ with hyperedge
  multiplicities $\bm{\mu} = (\mu_I)_{I\subseteq [n]} \in
  \N_0^{\mathcal{P}(n)}$, and set~$m=\sum_{I\subseteq [n]}\mu_I$. If
  $\lri$ is a cDVR as above, then
    \begin{align*} 
        \zeta^{\mathsf{ask}}_{M_\mathsf{H}(\lri)}(s) &= (1 -
        q^{-1})^{-1}\zeta_{\mathsf{A}_1^{n+1}}\left(\left(f_{n,J}\left(s,
        m, (\mu_I)_{I\subseteq [n]}\right)\right)_{\varnothing \neq
          J\subseteq[n+1]}\right).
    \end{align*} 
\end{prop}

\begin{proof}
    In~\cite[\S~5.1]{RV:CICO}, Rossmann and the second author define
    the $\mfp$-adic integral
    \begin{align*} 
        \mathsf{Z}_{[n], \emptyset}\left(s_0, (s_I)_{I\subseteq
          [n]}\right) &= \int_{\lri^{n+1}} |Y|^{s_0}
        \prod_{I\subseteq [n]} \norm{Z_I; Y}^{s_I}\,
        |\textup{d}\bfZ||\textup{d}Y|,
    \end{align*} 
    where $Z_I:= \{Z_i ~|~ i \in I\}$, and one integrates with regards
    to the normalized additive Haar measure on
    $\lri^{n+1}$. Applying~\cite[(5.2)]{RV:CICO}, we obtain
    \begin{align*} 
      \MoveEqLeft{\zeta_{\mathsf{A}_1^{n+1}(\lri)}\left((f_{n,J}(s,m,(\mu_I)_{I\subseteq
          [n]})_{\varnothing \neq J\subseteq [n+1]})\right)}\\ &=
      \int_{\lri^{n+1}} \prod_{J \subseteq [n+1]}
      \norm{X_J}^{f_{n,J}(s,m,(\mu_I)_{I\subseteq [n]})} \,
      |\textup{d}\bfX| \\ &= \int_{\lri^{n+1}} |Y|^{s-n+m-1}
      \prod_{I\subseteq [n]} \norm{Z_I; Y}^{-\mu_I}\,
      |\textup{d}\bfZ||\textup{d}Y| \\ &=
      \mathsf{Z}_{[n],\emptyset}\left(s-n+m-1, (-\mu_I)_{I\subseteq
        [n]}\right) \\ &= (1 - q^{-1})
      \zeta^{\mathsf{ask}}_{M_\mathsf{H}(\lri)}(s) .
    \end{align*}
    (For us, the invariant called $d$ in \cite[(5.2)]{RV:CICO} is zero.)
  \end{proof}
  It is of interest to explore whether there are comparable
  applications of zeta functions associated with hyperplane
  arrangements other than Boolean ones, say to ask zeta functions
  associated with more general module representations.

\section{Classical Coxeter arrangements, total partitions, and rooted
  trees}\label{sec:total-partitions}

We focus now on the classical Coxeter arrangements. In
Section~\ref{subsec:proof.thm.braid.detail} we prove
Theorem~\ref{thm:total-partition}, the precise version of
Theorem~\ref{thm:atom.braid}. It provides a fully explicit, finitary
combinatorial description of the atom zeta functions associated with
these Coxeter arrangements in terms of total partitions or,
equivalently, certain labeled rooted trees. The groundwork for this
result is laid in Sections~\ref{subsec:braid.total}
and~\ref{subsec:TP}.  Throughout this section, let
$\mathsf{X}\in\{\mathsf{A},\mathsf{B},\mathsf{D}\}$ and $n\in\N$, with
$n\geq 2$ if $\mathsf{X}=\mathsf{D}$.

\subsection{Classical Coxeter arrangements and set partitions} \label{subsec:braid.total}

We recall some standard facts and develop notation concerning classical
Coxeter arrangements and their intersection posets. The classical Coxeter
arrangements of rank $n$ are
\begin{equation} \label{eqn:classical-Coxeter}
  \begin{aligned}
    \mathsf{A}_n &:= \left\{ X_i - X_j ~\middle|~ 1\leq i < j \leq n+1
    \right\}, \\ \mathsf{B}_n &:= \left\{ X_i - \varepsilon X_j
    ~\middle|~ 1\leq i < j \leq n, \; \epsilon \in\{-1,1\} \right\}
    \cup \left\{ X_k ~\middle|~ k\in[n] \right\}, \\ \mathsf{D}_n &:=
    \left\{ X_i - \varepsilon X_j ~\middle|~ 1\leq i < j \leq n, \;
    \epsilon \in\{-1,1\} \right\}. 
  \end{aligned}
\end{equation} 
We set $\mathsf{D}_1 := \mathsf{A}_1$, and
$\mathsf{A}_0=\mathsf{B}_0=\mathsf{D}_0$ is the empty arrangement.

The intersection posets of these hyperplane arrangements are
isomorphic to the posets of set partitions of the respective types. In
the sequel, we recall Bj\"orner and
Wachs's~\cite{BW:partition-lattices} description of these posets.  

The set $\Pitwo{A}{n}$ is the poset of set partitions of
$[n+1]$, ordered by refinement, with minimal element $\hat{0}=
\{\{1\},\{2\},\dots,\{n+1\}\}$ and maximal element~$\hat{1} =
\{[n+1]\}$. For example,
\begin{align*}
    \Pitwo{A}{2} &= \{0|1|2, \; 01|2, \; 02|1, \; 0|12, \;
    012\}.
\end{align*}
Here and below, we use the abbreviated notation such as $0|12$ instead
of~$\{\{0\},\{1,2\}\}$. We will refer to elements of set partitions as
\emph{blocks}.

The elements of $\Pitwo{B}{n}$ are obtained from set partitions of
$[n]_0$, where the nonminimal integers of blocks not containing $0$
may each be independently \emph{barred}, i.e.\ decorated with a
special symbol called \emph{bar}. For example,
\begin{align*}
    \Pitwo{B}{2} &= \{ 0|1|2, \;01|2,\; 02|1, \;0|12,\;
    0|1\bar{2}, \;012\}.
\end{align*}
We say that two (possibly barred) integers $i,j\in [n]_0$ have the
\emph{same parity} in $P\in\Pitwo{B}{n}$ if they are both barred or
both not barred in $P$; otherwise they have \emph{different
  parity}. The \emph{zero block} of $P\in\Pitwo{B}{n}$ is the unique
block containing~$0$, denoted by $P_0$.  Before defining the poset
structure on $\Pitwo{B}{n}$, we define $\Pitwo{D}{n}\subset
\Pitwo{B}{n}$ as the sub(po-)set of elements whose zero block is not
of size two. For example,
\begin{align*}
    \Pitwo{D}{2} &= \{ 0|1|2, \;0|12,\; 0|1\bar{2}, \;012\}.
\end{align*}

To define the poset structure on $\Pitwo{B}{n}$ and $\Pitwo{D}{n}$, we
describe two operations on the blocks of a set partition in $\Pitwo{B}{n}$,
viz.\ \emph{bar} and \emph{unbar}.  To bar a block is to put a bar over all
numbers without a bar (even the minimal one) and to remove the bar over all
numbers with a bar. To unbar a block is to remove all bars. For example, if
$b=1\bar{2}3\bar{4}$ is a block, then the bar and unbar of $b$ are
$\overline{b}= \bar{1}2\bar{3}4$ and $\widetilde{b}= 1234$ respectively. By
definition, partitions $\rho,\sigma\in\Pitwo{B}{n}$ satisfy $\rho\leq\sigma$
if, for each block $b$ of~$\rho$, either
\begin{enumerate}
    \item $b$ is contained in a nonzero block of $\sigma$,
    \item $\bar{b}$ is contained in a nonzero block of $\sigma$, or
    \item $\widetilde{b}$ is contained in the zero block of $\sigma$.
\end{enumerate}

We recall the respective explicit isomorphisms $\Pitwo{X}{n} \cong
\intpos(\mathsf{X}_n)$ given in
\cite[Sec.~6--8]{BW:partition-lattices}. For $i,j,k\in[n]$ with $i\neq
j$, define hyperplanes $H_k = V\left(X_k\right)$ and
\begin{align*} 
    H_{J} &= \left\{ \begin{array}{ll} V\left(X_i - X_j\right) &
      \text{if }J = \{i,j\}\text{ or } J = \{\bar{i}, \bar{j}\},
      \\ [1em] V\left(X_i + X_j\right) & \text{if }J =
      \{\bar{i},j\}\text{ or } J = \{i, \bar{j}\}.
    \end{array}\right.
\end{align*} 
Let $\rhorep{A}{n}: \Pitwo{A}{n} \rightarrow
\intpos(\mathsf{A}_n)$ be the isomorphism
\begin{align}\label{eqn:type-A-iso}
  P\mapsto \bigcap_{\substack{I\in P \\ |I|\geq
      2}}\bigcap_{\substack{J\subseteq I \\ |J|=2}} H_{J},
\end{align} 
and let $\rhorep{B}{n} : \Pitwo{B}{n} \rightarrow
\intpos(\mathsf{B}_n)$ be the isomorphism
\begin{align} \label{eqn:type-B-iso} P\mapsto \left(\bigcap_{k\in P_0\setminus\{0\}}
  H_k \right)\cap \bigcap_{\substack{I\in P\setminus\{P_0\} \\ |I|\geq
      2}}\bigcap_{\substack{J\subseteq I \\ |J|=2}} H_{J} .
\end{align} 
The isomorphism $\rhorep{D}{n} : \Pitwo{D}{n} \rightarrow
\intpos(\mathsf{D}_n)$ is obtained by restricting
$\rhorep{B}{n}$ to $\Pitwo{D}{n}$.

For a set $I$, let $\mcP(I)$ be the power set of~$I$. If
$I\subset\N_0$, set $\mcP_{\mathsf{A}}(I):= \mcP(I)$.  Let
$\mcP_\mathsf{B}(I)$ be the set of subsets obtained from elements of
$\mcP_\mathsf{A}(I)$, where nonminimal elements of sets not
containing~$0$ may be independently barred. For example,
$$\mcP_\mathsf{B}(\{0,1,2\}) = \mcP_\mathsf{A}(\{0,1,2\}) \cup
\{\{1,\overline{2}\}\}.$$ Set
$\mcP_\mathsf{D}(I) = \mcP_\mathsf{B}(I) \setminus\{\{0, i\}~|~ 0\neq i\in
I\}$. We write $\mathcal{P}_{\mathsf{X}}(I;2)$ for the set of $2$-element
subsets of $\mathcal{P}_{\mathsf{X}}(I)$, which are in bijection with the
atoms of~$\Pitwo{X}{n}$.

\subsection{Total partitions and labeled rooted trees}
\label{subsec:TP}

A \emph{total partition} (\emph{of type~$\mathsf{X}_n$}) is a flag of
partitions in $\Pitwo{X}{n}$ of the form
\begin{equation}\label{eqn:TP-chain}
    \hat{0} = P_0 < P_1 < \cdots < P_h < P_{h+1}=\hat{1}
\end{equation}
such that $P_{i}\cap P_{i+1}$ is either empty or contains only
singletons for all $i\in[h-1]$. The set of total partitions of type
$\mathsf{X}$ is denoted by $\TPnew{\mathsf{X}}{n}$. We also write
$\TP{n+1}$ for~$\TPnew{\mathsf{A}}{n}$.  The sequence
$(\#\TP{n})_{n\in\N}$ is also known as the solution of Schr\"oder's
fourth problem;
cf.\ \cite[\href{http://oeis.org/A000311}{A000311}]{OEIS}.

Below we give a description of total partitions in terms of labeled
rooted trees.

\subsubsection{$\mathsf{X}_n$-labeled rooted trees}\label{subsubsec:Xn.labeled.rooted.trees}
To describe the sets $\TPnew{\mathsf{X}}{n}$ 
in terms of labeled rooted trees, we first
establish some terminology, borrowing
from~\cite[Appendix]{Stanley:Vol1}. Given a rooted tree $\tau$, we
denote by $V(\tau)$ the set of vertices and by~$P(\tau)$ the set of
\emph{parents} of~$\tau$, viz.\ vertices which are not leaves. A
vertex $u$ is a \emph{descendant} of a vertex $v$ if either $u$ is a
child of $v$ or $u$ is a descendant of a child of~$v$. The term
\emph{ancestor} is defined analogously.  Two distinct vertices are
\emph{siblings} if they have the same parent. We call a vertex $u$
\emph{unbranched} if either $u$ is a leaf or $u$ has exactly one leaf
descendant. Given a labeled rooted tree~$\tau$, we denote by
$v_\alpha:=v_\alpha(\tau)\in V(\tau)$ the leaf of $\tau$
labeled~$\alpha$. Write $V(\tau,\alpha)$ for the set of all ancestors
of~$v_\alpha$. We denote by $v_\alpha^+:=v_\alpha^+(\tau)\in V(\tau,
\alpha)$ the first ancestor of $v_\alpha$ with at least two children.

\begin{defn}\label{def:X-label}
Suppose that $\tau$ is a rooted tree with $n+1$ leaves. We call $\tau$
\begin{enumerate}
  \item \emph{$\mathsf{A}$-labeled} if each integer in $[n+1]$ labels
    a unique leaf of $\tau$, 
  \item \emph{$\mathsf{B}$-labeled} if each integer in $[n]_0$ labels
    a unique leaf of $\tau$, where each positive integer may be
    barred,
  \item\label{def:D-label} \emph{$\mathsf{D}$-labeled} if $\tau$ is 
    $\mathsf{B} $-labeled and the following holds: if all of the children of 
    $v_0^+$ are unbranched, then $v_0^+$ has at least three children.
\end{enumerate}
We also say that $\tau$ is \emph{$\mathsf{X}_n$-labeled} if $\tau$ is
$\mathsf{X}$-labeled with $n+1$ leaves.

We write $\LRTtwo{\mathsf{X}}{n}$ for the respective sets of
$\mathsf{X}_n$-labeled rooted trees.
\end{defn}

\begin{defn}\label{defn:DLL}
  Suppose $\tau$ is an $\mathsf{X}$-labeled rooted tree and $u\in
  V(\tau)$. If $u$ is a leaf, let $\mathrm{DLL}(\tau, u)$ be the
  (singleton) set containing the label of $u$; otherwise, let
  $\mathrm{DLL}(\tau, u)$ be the set of labels of all descendants of
  $u$ which are leaves.
\end{defn}

\begin{defn}\label{defn:standard-form}
  We call $\tau\in\LRTtwo{\mathsf{X}}{n}$ in \emph{standard form} if,
  for all $u \in V(\tau,0)$ and for all children of $v$ of~$u$, the
  minimal label of $\mathrm{DLL}(\tau,v)$ is not barred. Note that
  every tree $\tau\in\LRTtwo{A}{n}$ is in standard form.
\end{defn}

We now identify $\TPnew{\mathsf{X}}{n}$ with the subset of
$\LRTtwo{\mathsf{X}}{n}$ comprising trees in standard form with the property
that every parent has at least two children. See, for example,
Figure~\ref{fig:B2-normal-form} for the set $\TPnew{\mathsf{B}}{2}$; the first,
fourth, and fifth trees in Figure~\ref{fig:B2-normal-form} form the
set~$\TPnew{\mathsf{D}}{2}$. See Figure~\ref{fig:tree-total-partition} for an
example tree in~$\TPnew{\mathsf{B}}{9}$.

\begin{figure}[h]
    \centering
    \begin{subfigure}[b]{0.18\textwidth}
        \centering
        \begin{tikzpicture} 
            \node(1) at (0, 0.5) [circle, fill=black, inner sep=1.5pt]
                 {}; \node(2) at (-0.5, 0) [circle, fill=black, inner
                   sep=1.5pt, label={below:$0$}] {}; \node(3) at (0,
                 0) [circle, fill=black, inner sep=1.5pt,
                   label={below:$1$}] {}; \node(4) at (0.5, 0)
                 [circle, fill=black, inner sep=1.5pt,
                   label={below:$2$}] {}; \draw[-] (1) -- (2);
                 \draw[-] (1) -- (3); \draw[-] (1) -- (4);
        \end{tikzpicture}
    \end{subfigure}~%
    \begin{subfigure}[b]{0.18\textwidth}
        \centering
        \begin{tikzpicture} 
            \node(1) at (0, 1) [circle, fill=black, inner sep=1.5pt]
                 {}; \node(2) at (-0.33, 0.5) [circle, fill=black,
                   inner sep=1.5pt] {}; \node(3) at (0.33, 0.5)
                 [circle, fill=black, inner sep=1.5pt,
                   label={below:$2$}] {}; \node(4) at (-0.66, 0)
                 [circle, fill=black, inner sep=1.5pt,
                   label={below:$0$}] {}; \node(5) at (0, 0) [circle,
                   fill=black, inner sep=1.5pt, label={below:$1$}] {};
                 \draw[-] (1) -- (2); \draw[-] (1) -- (3); \draw[-]
                 (2) -- (4); \draw[-] (2) -- (5);
        \end{tikzpicture}
    \end{subfigure}~%
    \begin{subfigure}[b]{0.18\textwidth}
        \centering
        \begin{tikzpicture} 
            \node(1) at (0, 1) [circle, fill=black, inner sep=1.5pt]
                 {}; \node(2) at (-0.33, 0.5) [circle, fill=black,
                   inner sep=1.5pt] {}; \node(3) at (0.33, 0.5)
                 [circle, fill=black, inner sep=1.5pt,
                   label={below:$1$}] {}; \node(4) at (-0.66, 0)
                 [circle, fill=black, inner sep=1.5pt,
                   label={below:$0$}] {}; \node(5) at (0, 0) [circle,
                   fill=black, inner sep=1.5pt, label={below:$2$}] {};
                 \draw[-] (1) -- (2); \draw[-] (1) -- (3); \draw[-]
                 (2) -- (4); \draw[-] (2) -- (5);
        \end{tikzpicture}
    \end{subfigure}~%
    \begin{subfigure}[b]{0.18\textwidth}
        \centering
        \begin{tikzpicture} 
            \node(1) at (0, 1) [circle, fill=black, inner sep=1.5pt]
                 {}; \node(2) at (-0.33, 0.5) [circle, fill=black,
                   inner sep=1.5pt] {}; \node(3) at (0.33, 0.5)
                 [circle, fill=black, inner sep=1.5pt,
                   label={below:$0$}] {}; \node(4) at (-0.66, 0)
                 [circle, fill=black, inner sep=1.5pt,
                   label={below:$1$}] {}; \node(5) at (0, 0) [circle,
                   fill=black, inner sep=1.5pt, label={below:$2$}] {};
                 \draw[-] (1) -- (2); \draw[-] (1) -- (3); \draw[-]
                 (2) -- (4); \draw[-] (2) -- (5);
        \end{tikzpicture}
    \end{subfigure}~%
    \begin{subfigure}[b]{0.18\textwidth}
        \centering
        \begin{tikzpicture} 
            \node(1) at (0, 1) [circle, fill=black, inner sep=1.5pt]
                 {}; \node(2) at (-0.33, 0.5) [circle, fill=black,
                   inner sep=1.5pt] {}; \node(3) at (0.33, 0.5)
                 [circle, fill=black, inner sep=1.5pt,
                   label={below:$0$}] {}; \node(4) at (-0.66, 0)
                 [circle, fill=black, inner sep=1.5pt,
                   label={below:$1$}] {}; \node(5) at (0, 0) [circle,
                   fill=black, inner sep=1.5pt,
                   label={[label distance=-1pt]below:$\bar{2}$}] {}; \draw[-] (1) -- (2);
                 \draw[-] (1) -- (3); \draw[-] (2) -- (4); \draw[-]
                 (2) -- (5);
        \end{tikzpicture}
    \end{subfigure}
    \caption{All of the trees in $\TPnew{B}{2}$.}
    \label{fig:B2-normal-form}
\end{figure}

\subsubsection{From leaf labels to blocks}
\label{subsubsec:labels}

Suppose that $\tau$ is an $\mathsf{X}_n$-labeled tree. The leaf labels of $\tau$
determine a unique block for each $u\in V(\tau)$, which we denote by
$\lambda(\tau, u)$. These are the blocks that make up the set partitions in the
flags. Below we describe how a labeled rooted tree
$\tau\in\TPnew{\mathsf{X}}{n}$ uniquely determines a flag in
$\Pitwo{\mathsf{X}}{n}$ like in~\eqref{eqn:TP-chain}, but we emphasize that this
works for all $\tau\in \LRTtwo{X}{n}$. 

For a leaf $v_\alpha$ labeled $\alpha$, we let $\lambda(\tau, v_\alpha)$ be the
unbarred (singleton) block~$\widetilde{\{\alpha\}}$. If $u\in P(\tau)$ and
$\mathsf{X}=\mathsf{A}$, then $\lambda(\tau,u)$ is the union of the
$\lambda(\tau, u')$, where $u'$ ranges over the children of~$u$. This determines
all blocks $\lambda(\tau, u)$ in type $\mathsf{A}$, so we assume that
$\mathsf{X}\in\{\mathsf{B},\mathsf{D}\}$. If $u\in V(\tau, 0)$, then
$\lambda(\tau, u)$ is the union of the unbarred blocks $\lambda(\tau, u')$,
where $u'$ ranges over the children of~$u$. If $u\in P(\tau)\setminus
V(\tau, 0)$, then $\lambda(\tau, u)$ is the union of the $\lambda(\tau, u')$,
where $u'$ ranges over the children of~$u$, each of which is barred, if
necessary, so that the following hold.
\begin{enumerate}
    \item\label{one} For the unique child $u'$ of $u$ such that $\lambda(\tau,
    u')$ contains the minimal label in $\mathrm{DLL}(\tau,u)$, the block
    $\lambda(\tau, u')$ is not barred, and 
    \item\label{two} for all other children $u''$ of $u$, the block
    $\lambda(\tau, u'')$ is barred, if necessary, so the minimal labels in
    $\mathrm{DLL}(\tau,u')$ and $\mathrm{DLL}(\tau,u'')$ have the same parity
    (as labels) if and only if they have the same parity in $\lambda(\tau, u)$.
\end{enumerate} 

\begin{figure}[h]
  \centering 
  \begin{tikzpicture} 
    \pgfmathsetmacro{\x}{1.5}
    \pgfmathsetmacro{\y}{0.5}
    \node (v1) at (-3*\x, \y) [circle, fill=black, inner sep=1.5pt, label={above:$u_1$}] {};
    \node (l1) at (-3.25*\x, 0) [circle, fill=black, inner sep=1.5pt, label={below:$1$}] {};
    \node (r1) at (-2.75*\x, 0) [circle, fill=black, inner sep=1.5pt, label={below:$2$}] {};
    \node (v2) at (-1*\x, \y) [circle, fill=black, inner sep=1.5pt, label={above:$u_2$}] {};
    \node (l2) at (-1.25*\x, 0) [circle, fill=black, inner sep=1.5pt, label={below:$1$}] {};
    \node (r2) at (-0.75*\x, 0) [circle, fill=black, inner sep=1.5pt, label={[label distance=-1pt]below:$\bar{2}$}] {};
    \node (v3) at (\x, \y) [circle, fill=black, inner sep=1.5pt, label={above:$u_3$}] {};
    \node (r3) at (1.25*\x, 0) [circle, fill=black, inner sep=1.5pt, label={below:$2$}] {};
    \node (l3) at (0.75*\x, 0) [circle, fill=black, inner sep=1.5pt, label={[label distance=-1pt]below:$\bar{1}$}] {};
    \node (v4) at (3*\x, \y) [circle, fill=black, inner sep=1.5pt, label={above:$u_4$}] {};
    \node (r4) at (3.25*\x, 0) [circle, fill=black, inner sep=1.5pt, label={[label distance=-1pt]below:$\bar{2}$}] {};
    \node (l4) at (2.75*\x, 0) [circle, fill=black, inner sep=1.5pt, label={[label distance=-1pt]below:$\bar{1}$}] {};
    \node at (-3*\x, -2*\y) {$\lambda(\tau_1,u_1)=12$};
    \node at (-1*\x, -2*\y) {$\lambda(\tau_2,u_2)=1\bar{2}$};
    \node at (\x, -2*\y) {$\lambda(\tau_3,u_3)=1\bar{2}$};
    \node at (3*\x, -2*\y) {$\lambda(\tau_4,u_4)=12$};
    \draw (v1) -- (l1);
    \draw (v1) -- (r1);
    \draw (v2) -- (l2);
    \draw (v2) -- (r2);
    \draw (v3) -- (l3);
    \draw (v3) -- (r3);
    \draw (v4) -- (l4);
    \draw (v4) -- (r4);
  \end{tikzpicture}
  \caption{Four examples of blocks, contained in four different $\mathsf{B}$-labeled trees $\tau_i$, uniquely determined by~\eqref{one} and~\eqref{two}.}
  \label{fig:property-two}
\end{figure}
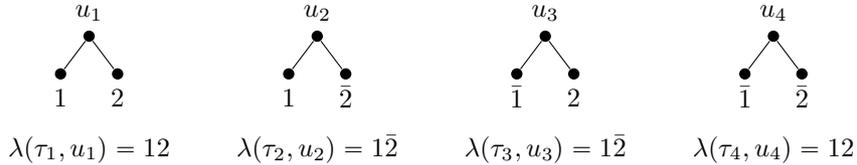

Figure~\ref{fig:tree-total-partition} demonstrates the role of blocks
in the translation between total partitions as flags and labeled
rooted trees. With this identification, we will freely use
tree-centric terminology in our further discussion of total
partitions.

\begin{figure}[h]
    \centering
    \begin{subfigure}[b]{0.45\textwidth}
        \centering
        \begin{tikzpicture}
            \node(S) at (0, 3) [circle, fill=black, inner sep=1.5pt]
                 {}; \node(L) at (-1.5, 2) [circle, fill=black, inner
                   sep=1.5pt] {}; \node(M) at (0, 2) [circle,
                   fill=black, inner sep=1.5pt] {}; \node(R) at (1.5,
                 2) [circle, fill=black, inner sep=1.5pt,
                   label={below:$4$}] {}; \node(LL) at (-1.75, 1)
                 [circle, fill=black, inner sep=1.5pt,
                   label={below:$2$}] {}; \node(LR) at (-1.25, 1)
                 [circle, fill=black, inner sep=1.5pt] {}; \node(ML)
                 at (-0.75, 1) [circle, fill=black, inner sep=1.5pt,
                   label={below:$0$}] {}; \node(MM1) at (-0.25, 1)
                 [circle, fill=black, inner sep=1.5pt,
                   label={below:$1$}] {}; \node(MM2) at (0.25, 1)
                 [circle, fill=black, inner sep=1.5pt,
                   label={below:$8$}] {}; \node(MR) at (0.75, 1)
                 [circle, fill=black, inner sep=1.5pt] {}; \node(MRL)
                 at (0.25, 0) [circle, fill=black, inner sep=1.5pt,
                   label={below:$3$}] {}; \node(MRM) at (0.75, 0)
                 [circle, fill=black, inner sep=1.5pt,
                   label={[label distance=-1pt]below:$\bar{6}$}] {}; \node(MRR) at (1.25,
                 0) [circle, fill=black, inner sep=1.5pt,
                   label={below:$7$}] {}; \node(LRL) at (-1.5,
                 0) [circle, fill=black, inner sep=1.5pt,
                   label={[label distance=-1pt]below:$\bar{5}$}] {}; \node(LRR) at (-1, 0)
                 [circle, fill=black, inner sep=1.5pt,
                   label={[label distance=-1pt]below:$\bar{9}$}] {};

            \draw (S) -- (L); \draw (S) -- (M); \draw (S) -- (R);
            \draw (L) -- (LL); \draw (L) -- (LR); \draw (M) -- (ML);
            \draw (M) -- (MM1); \draw (M) -- (MM2); \draw (M) -- (MR);
            \draw (MR) -- (MRL); \draw (MR) -- (MRM); \draw (MR) --
            (MRR); \draw (LR) -- (LRL); \draw (LR) -- (LRR);
        \end{tikzpicture}
        \caption{A labeled rooted tree in $\TPnew{B}{9}$.}
    \end{subfigure}~%
    \begin{subfigure}[b]{0.45\textwidth}
        \centering \raisebox{3mm}{
        \begin{tikzpicture}
            \node at (0, 3) {$\widehat{1}=0123456789$}; \node at (0,
            2) {$2\bar{5}\bar{9}|013678|4$}; \node at (0, 1)
            {$2|59|0|1|8|3\bar{6}7|4$}; \node at (0, 0)
            {$\widehat{0}=2|5|9|0|1|8|3|6|7|4$};
        \end{tikzpicture}
        }
        \caption{A flag of set partitions in $\Pitwo{B}{9}$.}
    \end{subfigure}
    \caption{An illustration of the bijection between $\TPnew{B}{9}$ and certain
      flags of set partitions in $\Pitwo{B}{9}$.}
    \label{fig:tree-total-partition}
\end{figure}

\subsection{Atom zeta functions and rooted trees}\label{subsec:atom.root.trees}

Theorem~\ref{thm:atom.braid} is a paraphrase of the more precise
Theorem~\ref{thm:total-partition}, which gives uniform formulae for the atom
zeta function $\zeta_{\mathsf{X}_n(\lri)}^{\atom}(\bfs)$, for
$\mathsf{X}\in\{\mathsf{A},\mathsf{B},\mathsf{D}\}$, in terms of
$\mathsf{X}_n$-labeled rooted trees. We prove Theorem~\ref{thm:total-partition} in Section~\ref{subsec:proof.thm.braid.detail}. In Corollary~\ref{cor:A_n-unlab} we record
a variant of this formula for Igusa's local zeta function associated with the
braid arrangements in
terms of unlabeled trees. In Corollary~\ref{cor:B.typeA} we give formulae for
types $\mathsf{B}$ and $\mathsf{D}$ in terms of total partitions of
type~$\mathsf{A}$. 

We begin by giving definitions of the numerical data in the
statement of Theorem~\ref{thm:atom.braid}. For a rooted tree $\tau$
and $u\in V(\tau)$, we define the counting functions
\begin{align}\label{def:count.c}
    \kind(\tau, u) &= \text{the number of children of $u$},
    \\ \kb(\tau, u) &= \text{the number of unbranched children of
      $u$}.\label{def:count.u}
\end{align}
If $\tau\in\TPnew{\mathsf{X}}{n}$, then $\kind(\tau, u)\neq 1$ for all
$u\in V(\tau)$.

For $n\geq 0$, define polynomials $\pb[Y]{n} = 1 + nY$ and
\begin{align*}
    \pb[Y]{n}! &= \prod_{i=1}^n \pb[Y]{i}, & \pb[Y]{n}!! &=
    \prod_{i=1}^n \pb[Y]{2i-1}.
\end{align*}
Note that $\pb[Y]{0}! = \pb[Y]{0}!! = 1$. For $\tau \in\LRTtwo{X}{n}$, we define
three polynomials in $\Z[Y]$, depending on the type $\mathsf{X}\in
\{\mathsf{A}, \mathsf{B}, \mathsf{D}\}$:
\begin{equation}\label{eqn:pi.X.tau}
    \begin{split}
    \pitwo{A}{\tau}(Y) &= \prod_{u\in P(\tau)}
    \pb[Y]{\kind(\tau, u)-1}!, \\
    \pitwo{B}{\tau}(Y) &=
    \left(\prod_{u\in V(\tau,0)} \pb[Y]{\kind(\tau,
      u)-1}!!\right)\left( \prod_{u\in P(\tau)\setminus V(\tau,0)}
    \pb[Y]{\kind(\tau, u)-1}!\right), \\
    \pitwo{D}{\tau}(Y) &=
    \pitwo{B}{\tau}(Y)
    \cdot\frac{\pb[Y]{2\kind(\tau,v_0^+(\tau))-\kb(\tau,v_0^+(\tau))-2}
      \pb[Y]{\kind(\tau,v_0^+(\tau))-2}!!}{\pb[Y]{\kind(\tau,v_0^+(\tau))-1}!!};
    \end{split}
\end{equation}
see, for instance, Examples~\ref{ex:char.poly} and
\ref{ex:numerical-data}. In fact, $\pitwo{X}{\tau}(Y)$ is the
Poincar\'e polynomial $\pi_F(Y)$, where $F$ is a flag determined by
$\mathsf{X}$ and $\tau$; cf.\ Lemma~\ref{lem:general-Poincare}.

\begin{ex}[Poincar\'e polynomials of classical Coxeter arrangements]
  \label{ex:char.poly} 
  Let $\tau$ be the rooted tree whose root vertex has exactly $n+1$ children,
  all of which are leaves. For each
  $\mathsf{X}\in\{\mathsf{A}, \mathsf{B}, \mathsf{D}\}$, there is a unique
  $\mathsf{X}$-labeling for~$\tau$. The polynomials $\pitwo{X}{\tau}(Y)$ are
  just the Poincar\'e polynomials $\pi_{\mathsf{X}_n}(Y)$ of the respective
  arrangements:
    \begin{align*}
        \pitwo{A}{\tau}(Y) &= \prod_{i=1}^n(1+iY) =
        \pi_{\mathsf{A}_n}(Y), \\
        \pitwo{B}{\tau}(Y) &=
        \prod_{i=1}^n(1+(2i-1)Y) = \pi_{\mathsf{B}_n}(Y),
        \\
        \pitwo{D}{\tau}(Y) &= (1 + (n-1)Y)
        \prod_{i=1}^{n-1}(1+(2i-1)Y) = \pi_{\mathsf{D}_n}(Y);
    \end{align*}
cf.~\cite[Thm.~4.137]{OrlikTerao/92}.
\end{ex} Suppose that $\tau$ is an
$\mathsf{X}$-labeled rooted tree with root vertex $v$.  Let
$V_\circ(\tau) = P(\tau)\setminus \{v\}$ and $V_\circ(\tau, \alpha) =
V(\tau, \alpha)\setminus\{v\}$. Recalling that $\GP{X}$ denotes the
geometric progression $\frac{X}{1-X}$, we define the geometric
progression
\begin{equation}\label{eqn:Cgp}
    \begin{split}
    \treeCgp_{\mathsf{X},\tau}\left(Z, \bfT\right) &=
    \left(\prod_{u\in V_\circ(\tau)\setminus
      V_\circ(\tau,0)}\GP{Z^{|\lambda(\tau, u)|-1}\prod_{J\in
        \mcP(\lambda(\tau, u); 2)}T_J}\right) \\ &\quad \cdot \left(
    \prod_{u\in V_\circ(\tau,0)} \GP{Z^{|\lambda(\tau,
        u)|-1}\prod_{J\in \mcP_{\mathsf{X}}(\lambda(\tau, u); 2)}T_J}
    \right).
    \end{split}
\end{equation}
If $\mathsf{X}=\mathsf{A}$, then the second factor is $1$
since no leaf is labeled with~$0$. Examples of $\treeCgp_{\mathsf{X},\tau}(Z, \bfT)$ are seen in Example~\ref{ex:numerical-data}.

\begin{ex}[Numerical data] 
  \label{ex:numerical-data}
    Figure~\ref{fig:total-partition-data} shows the numerical
    data associated with the tree $\tau\in\TPnew{B}{9}$ from
    Figure~\ref{fig:tree-total-partition}, where
    \begin{alignat*}{2}
        C_{\mathsf{A}_1} &= \GP{Z\, T_{59}}, & C_{\mathsf{A}_2} &=
        \GP{Z^2\, T_{2\bar{5}} T_{2\bar{9}} T_{\bar{5}\bar{9}}},
        \\ C_{\mathsf{A}_2}' &= \GP{Z^2\,
          T_{3\overline{6}}T_{37}T_{\overline{6}7}},\quad
        & C_{\mathsf{B}_5} &= \GP{Z^5 \prod_{J\in
            \mcP_{\mathsf{B}}(\{0,1,3,6,7,8\};2)}T_J}.
    \end{alignat*}
    Thus, $\treeCgp_{\mathsf{B},\tau}(Z, \bfT) =
    C_{\mathsf{A}_1}C_{\mathsf{A}_2}C_{\mathsf{A}_2}'C_{\mathsf{B}_5}$
    and $\pitwo{B}{\tau}(Y) =
    \pb[Y]{2}!!\pb[Y]{3}!!\pb[Y]{1}!^2\pb[Y]{2}!$.
    \begin{figure}[h]
        \centering
        \begin{subfigure}[b]{0.45\textwidth}
            \centering
            \begin{tikzpicture}
                \pgfmathsetmacro{\y}{1}
                \node(S) at (0, 3*\y) {$\pb[Y]{2}!!$}; \node(L) at
                (-1.75, 2*\y) {$\pb[Y]{1}!$}; \node(M) at (0, 2*\y)
                     {$\pb[Y]{3}!!$}; \node(R) at (1.5, 2*\y) [circle,
                       fill=black, inner sep=1.5pt, label={below:$4$}]
                     {}; \node(LL) at (-2.1, \y) [circle, fill=black,
                       inner sep=1.5pt, label={below:$2$}] {};
                     \node(LR) at (-1.45, \y) {$\pb[Y]{1}!$}; \node(ML)
                     at (-0.75, \y) [circle, fill=black, inner
                       sep=1.5pt, label={below:$0$}] {}; \node(MM1) at
                     (-0.25, \y) [circle, fill=black, inner sep=1.5pt,
                       label={below:$1$}] {}; \node(MM2) at (0.25, \y)
                     [circle, fill=black, inner sep=1.5pt,
                       label={below:$8$}] {}; \node(MR) at (1, \y)
                     {$\pb[Y]{2}!$}; \node(MRL) at (0.5, 0) [circle,
                       fill=black, inner sep=1.5pt, label={below:$3$}]
                     {}; \node(MRM) at (1, 0) [circle, fill=black,
                       inner sep=1.5pt, label={[label distance=-1pt]below:$\bar{6}$}] {};
                     \node(MRR) at (1.5, 0) [circle, fill=black, inner
                       sep=1.5pt, label={below:$7$}] {}; \node(LRL) at
                     (-1.7, 0) [circle, fill=black, inner sep=1.5pt,
                       label={[label distance=-1pt]below:$\bar{5}$}] {}; \node(LRR) at
                     (-1.2, 0) [circle, fill=black, inner sep=1.5pt,
                       label={[label distance=-1pt]below:$\bar{9}$}] {};

                \draw (S) -- (L); \draw [ultra thick] (S) -- (M);
                \draw (S) -- (R); \draw (L) -- (LL); \draw (L) --
                (LR); \draw [ultra thick] (M) -- (ML); \draw (M) --
                (MM1); \draw (M) -- (MM2); \draw (M) -- (MR); \draw
                (MR) -- (MRL); \draw (MR) -- (MRM); \draw (MR) --
                (MRR); \draw (LR) -- (LRL); \draw (LR) -- (LRR);
            \end{tikzpicture}
            \caption{The $\pitwo{B}{\tau}$-factors.}
        \end{subfigure}~%
        \begin{subfigure}[b]{0.45\textwidth}
            \centering
            \begin{tikzpicture}
                \pgfmathsetmacro{\y}{1}
                \node(S) at (0, 3*\y) [circle, fill=black, inner
                  sep=1.5pt] {}; \node(L) at (-1.5, 2*\y)
                     {$C_{\mathsf{A}_2}$}; \node(M) at (0, 2*\y)
                     {$C_{\mathsf{B}_5}$}; \node(R) at (1.5, 2*\y)
                     [circle, fill=black, inner sep=1.5pt,
                       label={below:$4$}] {}; \node(LL) at (-1.75, \y)
                     [circle, fill=black, inner sep=1.5pt,
                       label={below:$2$}] {}; \node(LR) at (-1.25, \y)
                     {$C_{\mathsf{A}_1}$}; \node(ML) at (-0.75, \y)
                     [circle, fill=black, inner sep=1.5pt,
                       label={below:$0$}] {}; \node(MM1) at (-0.25, \y)
                     [circle, fill=black, inner sep=1.5pt,
                       label={below:$1$}] {}; \node(MM2) at (0.25, \y)
                     [circle, fill=black, inner sep=1.5pt,
                       label={below:$8$}] {}; \node(MR) at (1, \y)
                     {$C_{\mathsf{A}_2}'$}; \node(MRL) at (0.5, 0)
                     [circle, fill=black, inner sep=1.5pt,
                       label={below:$3$}] {}; \node(MRM) at (1, 0)
                     [circle, fill=black, inner sep=1.5pt,
                       label={[label distance=-1pt]below:$\bar{6}$}] {}; \node(MRR) at
                     (1.5, 0) [circle, fill=black, inner sep=1.5pt,
                       label={below:$7$}] {}; \node(LRL) at (-1.5, 0)
                     [circle, fill=black, inner sep=1.5pt,
                       label={[label distance=-1pt]below:$\bar{5}$}] {}; \node(LRR) at (-1,
                     0) [circle, fill=black, inner sep=1.5pt,
                       label={[label distance=-1pt]below:$\bar{9}$}] {};

                \draw (S) -- (L); \draw [ultra thick] (S) -- (M);
                \draw (S) -- (R); \draw (L) -- (LL); \draw (L) --
                (LR); \draw [ultra thick] (M) -- (ML); \draw (M) --
                (MM1); \draw (M) -- (MM2); \draw (M) -- (MR); \draw
                (MR) -- (MRL); \draw (MR) -- (MRM); \draw (MR) --
                (MRR); \draw (LR) -- (LRL); \draw (LR) -- (LRR);
            \end{tikzpicture}
            \caption{The $\m_{\mathsf{B},\tau}$-factors.}
        \end{subfigure}
        \caption{The tree $\tau\in\TPnew{\mathsf{B}}{9}$ from
          Figure~\ref{fig:tree-total-partition} revisited.}
        \label{fig:total-partition-data}
    \end{figure}
\end{ex}

\begin{thm}[Theorem~\ref{thm:atom.braid} made precise]\label{thm:total-partition}
  Let $\mathsf{X}\in\{\mathsf{A}, \mathsf{B},\mathsf{D}\}$ and~$n\in
  \N$, with $n\geq 2$ if~$\mathsf{X}=\mathsf{D}$. For all cDVR $\lri$
  with residue field cardinality~$q$, assumed to be odd unless
  $\mathsf{X} = \mathsf{A}$,
    \begin{multline*}
        \zeta_{\mathsf{X}_n(\lri)}^{\atom}\left((s_J)_{J\in
          \mcP_{\mathsf{X}}(n;2)}\right) =
        \\ \GPZ{q^{-n-\sum_{J\in\mcP_{\mathsf{X}}(n;2)}s_J}}
        \sum_{\tau\in\TPnew{\mathsf{X}}{n}}\pitwo{X}{\tau}(-q^{-1})
        \cdot \treeCgp_{\mathsf{X},\tau}\left(q^{-1},(q^{-s_J})_{J\in
          \mcP_{\mathsf{X}}(\lambda(\tau);2)}\right) .
    \end{multline*}
\end{thm}

\subsubsection{Corollaries of Theorem~\ref{thm:atom.braid}}\label{subsubsec:A.B}

Let $\UTP{n+1}$ be the set of unlabeled rooted trees with $n+1$ leaves
where every parent has at least two children. (As the notation
suggests, $\UTP{n+1}$ is obtained from $\TP{n+1}$ by removing the
labels.) The number of distinct $\mathsf{A}$-labelings for
$\tau\in\UTP{n+1}$ is $(n+1)!/|\Aut(\tau)|$, where $\Aut(\tau)$ is the
subgroup of the graph automorphism group of $\tau$ stabilizing the
root. The sequence $(\#\UTP{n})_{n\in\N}$ is
\cite[\href{http://oeis.org/A000669}{A000669}]{OEIS}. For all
$\tau\in\UTP{n+1}$, if $\tau'$ and $\tau''$ are two
$\mathsf{A}$-labelings of~$\tau$, then $\pitwo{A}{\tau'}(Y) =
\pitwo{A}{\tau''}(Y)$, so we
set~$\pi_\tau(Y):=\pitwo{A}{\tau'}(Y)$. For $\tau\in \UTP{n+1}$ and
$u\in V(\tau)$, let $\nb(\tau, u)$ be the number of descendants of
$u$, including $u$ itself, that are leaves. Recall that
$f_{\mathsf{A}_n}(\bfX)=\prod_{1\leq i < j \leq {n+1}}(X_i-X_j)$.

\begin{cor}\label{cor:A_n-unlab}
  For all cDVR $\lri$ with residue field cardinality~$q$, the Igusa
  local zeta function associated with $f_{\mathsf{A}_n}$ over $\lri$
  is
    \begin{multline*} 
        \mathsf{Z}_{f_{\mathsf{A}_n},\lri}(s)
        =\\ \GPZ{q^{-n-\binom{n+1}{2}s}}\sum_{\tau\in\UTP{n+1}}\dfrac{(n+1)!}{|\Aut(\tau)|}\pi_{\tau}(-q^{-1})\prod_{u\in
          V_\circ(\tau)} \GP{q^{1-\nb(\tau, u) - \binom{\nb(\tau,
              u)}{2}s}} .
    \end{multline*} 
\end{cor}

\begin{ex}[Braid arrangement $\mathsf{A}_3$]\label{ex:A3}
  Consider the braid arrangement $\mathsf{A}_3$, with
  $f_{\mathsf{A}_3}(\bfX) = \prod_{1\leq i < j \leq 4}(X_i - X_j)$. All of the
  numerical data can be read off from the five trees comprising the set
  $\UTP{4}$, listed in Figure~\ref{fig:rooted-trees-4}.
\begin{figure}[h]
    \centering
    \begin{subfigure}[b]{0.18\textwidth} 
        \centering
        \begin{tikzpicture}
            \pgfmathsetmacro{\x}{0.75} \pgfmathsetmacro{\y}{0.5}
            \node(1) at (0, \y) [circle, fill=black, inner sep=1.5pt]
                 {}; \node(2) at (-\x, 0) [circle, fill=black, inner
                   sep=1.5pt] {}; \node(3) at (-0.33*\x, 0) [circle,
                   fill=black, inner sep=1.5pt] {}; \node(4) at
                 (0.33*\x, 0) [circle, fill=black, inner sep=1.5pt]
                 {}; \node(5) at (\x, 0) [circle, fill=black, inner
                   sep=1.5pt] {}; \node at (0, -1*\y) {$\tau_1$};
                 \draw[-] (1) -- (2); \draw[-] (1) -- (3); \draw[-]
                 (1) -- (4); \draw[-] (1) -- (5);
        \end{tikzpicture}
    \end{subfigure}~%
    \begin{subfigure}[b]{0.18\textwidth} 
        \centering
        \begin{tikzpicture}
            \pgfmathsetmacro{\x}{0.75} \pgfmathsetmacro{\y}{0.5}
            \node(1) at (0, 2*\y) [circle, fill=black, inner
              sep=1.5pt] {}; \node(2) at (-0.66*\x, \y) [circle,
              fill=black, inner sep=1.5pt] {}; \node(3) at (0, \y)
                 [circle, fill=black, inner sep=1.5pt] {}; \node(4) at
                 (0.66*\x, \y) [circle, fill=black, inner sep=1.5pt]
                 {}; \node(5) at (-1*\x, 0) [circle, fill=black, inner
                   sep=1.5pt] {}; \node(6) at (-0.33*\x, 0) [circle,
                   fill=black, inner sep=1.5pt] {}; \node at (0,
                 -1*\y) {$\tau_2$}; \draw[-] (1) -- (2); \draw[-] (1)
                 -- (3); \draw[-] (1) -- (4); \draw[-] (2) -- (5);
                 \draw[-] (2) -- (6);
        \end{tikzpicture}
    \end{subfigure}~%
    \begin{subfigure}[b]{0.18\textwidth} 
        \centering
        \begin{tikzpicture}
            \pgfmathsetmacro{\x}{0.75} \pgfmathsetmacro{\y}{0.5}
            \node(1) at (0, 2*\y) [circle, fill=black, inner
              sep=1.5pt] {}; \node(2) at (-0.66*\x, \y) [circle,
              fill=black, inner sep=1.5pt] {}; \node(3) at (0.66*\x,
            \y) [circle, fill=black, inner sep=1.5pt] {}; \node(4) at
            (-1*\x, 0) [circle, fill=black, inner sep=1.5pt] {};
            \node(5) at (-0.33*\x, 0) [circle, fill=black, inner
              sep=1.5pt] {}; \node(6) at (\x, 0) [circle, fill=black,
              inner sep=1.5pt] {}; \node(7) at (0.33*\x, 0) [circle,
              fill=black, inner sep=1.5pt] {}; \node at (0, -1*\y)
                 {$\tau_3$}; \draw[-] (1) -- (2); \draw[-] (1) -- (3);
                 \draw[-] (1) -- (3); \draw[-] (2) -- (4); \draw[-]
                 (2) -- (5); \draw[-] (3) -- (6); \draw[-] (3) -- (7);
        \end{tikzpicture}
    \end{subfigure}~%
    \begin{subfigure}[b]{0.18\textwidth} 
        \centering
        \begin{tikzpicture}
            \pgfmathsetmacro{\x}{0.75} \pgfmathsetmacro{\y}{0.5}
            \node(1) at (0, 2*\y) [circle, fill=black, inner
              sep=1.5pt] {}; \node(2) at (-0.5*\x, \y) [circle,
              fill=black, inner sep=1.5pt] {}; \node(3) at (0.5*\x,
            \y) [circle, fill=black, inner sep=1.5pt] {}; \node(4) at
            (-1*\x, 0) [circle, fill=black, inner sep=1.5pt] {};
            \node(5) at (-0.5*\x, 0) [circle, fill=black, inner
              sep=1.5pt] {}; \node(6) at (0, 0) [circle, fill=black,
              inner sep=1.5pt] {}; \node at (0, -1*\y) {$\tau_4$};
            \draw[-] (1) -- (2); \draw[-] (1) -- (3); \draw[-] (2) --
            (4); \draw[-] (2) -- (5); \draw[-] (2) -- (6);
        \end{tikzpicture}
    \end{subfigure}~%
    \begin{subfigure}[b]{0.18\textwidth} 
        \centering
        \begin{tikzpicture}
            \pgfmathsetmacro{\x}{0.75} \pgfmathsetmacro{\y}{0.5}
            \node(1) at (0, 3*\y) [circle, fill=black, inner
              sep=1.5pt] {}; \node(2) at (-0.5*\x, 2*\y) [circle,
              fill=black, inner sep=1.5pt] {}; \node(3) at (0.5*\x,
            2*\y) [circle, fill=black, inner sep=1.5pt] {}; \node(4)
            at (-1*\x, \y) [circle, fill=black, inner sep=1.5pt] {};
            \node(5) at (0, \y) [circle, fill=black, inner sep=1.5pt]
                 {}; \node(6) at (-1.5*\x, 0) [circle, fill=black,
                   inner sep=1.5pt] {}; \node(7) at (-0.5*\x, 0)
                 [circle, fill=black, inner sep=1.5pt] {}; \node at
                 (0, -\y) {$\tau_5$}; \draw[-] (1) -- (2); \draw[-]
                 (1) -- (3); \draw[-] (2) -- (4); \draw[-] (2) -- (5);
                 \draw[-] (4) -- (6); \draw[-] (4) -- (7);
        \end{tikzpicture}
    \end{subfigure}
    \caption{The five rooted trees in $\UTP{4}$.}
    \label{fig:rooted-trees-4}
\end{figure}
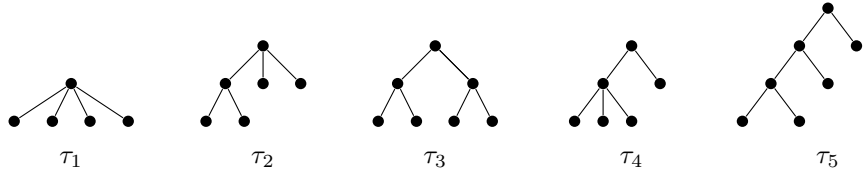
Specifically we find that
\begin{align*} 
    \Aut(\tau_1) &= S_4, & \Aut(\tau_2) &= S_2\times S_2, &
    \Aut(\tau_3) &= S_2\wr S_2, \\ \Aut(\tau_4) &= S_3, & \Aut(\tau_5)
    &= S_2,
\end{align*} 
\begin{align*} 
    \pi_{\tau_1}(Y) &= (1+Y)(1+2Y)(1+3Y), & \pi_{\tau_2}(Y) &=
    \pi_{\tau_4}(Y) = (1+Y)^2(1+2Y), \\ \pi_{\tau_3}(Y) &=
    \pi_{\tau_5}(Y) = (1+Y)^3.
\end{align*} 
By Corollary~\ref{cor:A_n-unlab}, for all cDVR $\lri$ with residue field cardinality~$q$,
Igusa's zeta function associated with $f:=f_{\mathsf{A}_3}$ is
\begin{align*} 
    \mathsf{Z}_{f,\lri}(s) &= \left.\dfrac{1 -
      q^{-1}}{1-q^{-3 - 6s}}\middle((1-2q^{-1})(1-3q^{-1}) \right. \\  
    &\quad +
    \dfrac{6q^{-1-s}(1-q^{-1})(1-2q^{-1})}{1 - q^{-1-s}} + \dfrac{3q^{-2-2s}(1-q^{-1})^2}{(1 - q^{-1-s})^2} \\
    &\quad \left. +
    \dfrac{4q^{-2-3s}(1-q^{-1})(1-2q^{-1})}{1 - q^{-2-3s}} + 
    \dfrac{12q^{-3-4s}(1-q^{-1})^2}{(1-q^{-1-s})(1 - q^{-2-3s})} \right).
\end{align*}
It might be instructive to compare this formula with the formula for
$\fHP_{\mathsf{A}_3}$ given in
Proposition~\ref{prop:braid.A3}. Arguing as in
Section~\ref{subsec:topological}, we obtain a formula for the
topological zeta function of $f$ in terms of the five trees in
$\UTP{4}$:
\begin{align*} 
  \mathsf{Z}_{f}^{\mathrm{top}}(s) &= \dfrac{1}{3+6s} \left(2 - \dfrac{6}{1+s} + \dfrac{3}{(1+s)^2} - \dfrac{4}{2+3s} + \dfrac{12}{(1+s)(2+3s)}\right) \\ 
  &= \dfrac{2 - s - 2s^2 + 2s^3}{(1+s)^2(1+2s)(2+3s)} .
\end{align*} 
\end{ex}

Turning now to atom zeta functions associated with Coxeter
arrangements of type $\mathsf{B}$, we define an embedding $\phi :
\TPnew{\mathsf{A}}{n} \rightarrow \TPnew{\mathsf{B}}{n}$ by replacing
the label $n+1$ with the label~$0$. Thus, the image of $\phi$
comprises all the trees whose labels have no bars, so each tree in the
image is in standard form. For $\tau\in\LRTtwo{B}{n}$, set
\begin{align*} 
    \mathrm{bars}(\tau) &:= 2^n\prod_{u\in V(\tau,
      0)}2^{1-\kind(\tau,u)}\in\N.
\end{align*}

\begin{cor}\label{cor:B.typeA}
  For all $n\geq 1$ and cDVR $\lri$ with residue field of odd cardinality~$q$,
    \begin{multline*}
        \zeta_{\mathsf{B}_n(\lri)}^{\atom}\left((s_J)_{J\in
          \mcP_{\mathsf{B}}(n;2)}\right) = 
        \GPZ{q^{-n-\sum_{J\in\mcP_{\mathsf{B}}(n;2)}s_J}} \\
        \cdot \sum_{\tau\in\phi(\TPnew{\mathsf{A}}{n})}\mathrm{bars}(\tau) \pitwo{B}{\tau}(-q^{-1})\treeCgp_{\mathsf{B},\tau}\left(q^{-1},(q^{-s_J})_{J\in
          \mcP_{\mathsf{B}}(\lambda(\tau);2)}\right) ,
    \end{multline*}
    \begin{multline*}
      \zeta_{\mathsf{D}_n(\lri)}^{\atom}\left((s_J)_{J\in
        \mcP_{\mathsf{D}}(n;2)}\right) = 
      \GPZ{q^{-n-\sum_{J\in\mcP_{\mathsf{D}}(n;2)}s_J}} \\
      \cdot \sum_{\tau\in\phi(\TPnew{\mathsf{A}}{n})\cap\TPnew{\mathsf{D}}{n}}\mathrm{bars}(\tau) \pitwo{D}{\tau}(-q^{-1})\treeCgp_{\mathsf{D},\tau}\left(q^{-1},(q^{-s_J})_{J\in
        \mcP_{\mathsf{D}}(\lambda(\tau);2)}\right) .
  \end{multline*}
\end{cor}

\begin{proof}
  Let $\beta : \TPnew{\mathsf{B}}{n} \rightarrow \TPnew{\mathsf{A}}{n}$ be the
  map given by removing bars from all leaves and replacing the label $0$ with
  $n+1$. Thus $\beta\circ \phi$ is the identity map on $\TPnew{\mathsf{A}}{n}$.
  From Definition~\ref{defn:standard-form},
  $|\beta^{-1}(\tau)|=\mathrm{bars}(\tau)$. Since $\TPnew{\mathsf{D}}{n}\subset
  \TPnew{\mathsf{B}}{n}$, the corollary follows.
\end{proof}

\subsection{Proof of Theorem~\ref{thm:atom.braid}}\label{subsec:proof.thm.braid.detail}

We specify Lemma~\ref{lem:gen-strat.new} for the three specific
families of Coxeter arrangements, also providing combinatorial
reinterpretations. We use the three resulting
Lemmas~\ref{lem:type-A_n-recursion}, \ref{lem:type-B_n-recursion}, and
\ref{lem:type-D_n-recursion} to recursively prove
Theorem~\ref{thm:total-partition}, and thus also
Theorem~\ref{thm:atom.braid}. We begin, however, with two key lemmas
concerning the restrictions of classical Coxeter arrangements.

\begin{lem}\label{lem:AB-Coxeter-res}
    Let $\mathsf{X}\in\{\mathsf{A}, \mathsf{B}\}$ and $n\in\N$. If
    $x\in\intpos(\mathsf{X}_n)$, then $\mathsf{X}_n^x \cong
    \mathsf{X}_{n-\rank(x)}$.
\end{lem}

\begin{proof}
  The idea is the same for both types; we spell it out for
  type~$\mathsf{B}$.  Set $P =
  \rhorep{B}{n}^{-1}(x)\in\Pi_{\mathsf{B},n}$, where $\rhorep{B}{n}$
  is as in~\eqref{eqn:type-B-iso}. It follows that $\rank(x) = n
  - |P| + 1$. Fix $I, J, K\in P\setminus\{P_0\}$ with $I\ne J$ and
  $\epsilon\in \{-1,1\}$, and unbar $I$, $J$, and $K$. Then, for all
  $i\in I$ and $j\in J$,
    \begin{align*} 
        x\cap V(X_i -\epsilon X_j) = x\cap V(X_{\min(I)}-\epsilon
        X_{\min(J)}).
    \end{align*}
    Furthermore, for all $k\in K$, $\ell\in P_0\setminus\{0\}$, and
    $\lambda\in \{-1, 0, 1\}$,
    \begin{align*} 
        x\cap V(X_k -\lambda X_{\ell}) = x\cap V(X_{\min(K)}).
    \end{align*}    
    Let $\{X_I ~|~ I\in P\setminus\{P_0\}\}$ be independent
    variables. Then
    \begin{align*} 
        \mathsf{B}_n^x &\cong \{X_I \pm X_J ~|~ I, J\in P\setminus
        \{P_0\}, I\neq J\} \cup \{X_K ~|~ K \in P\setminus\{P_0\}\}
        \cong \mathsf{B}_{n - \rank(x)}. \qedhere
    \end{align*}
\end{proof}

The case for type $\mathsf{D}$ is not as simple as
Lemma~\ref{lem:AB-Coxeter-res} is for types $\mathsf{A}$ and
$\mathsf{B}$, but follows similar reasoning. In fact, some
restrictions in $\mathsf{D}_n$ are not even Coxeter arrangements, and
these non-Coxeter restrictions are the subject of
Lemma~\ref{lem:Poincare-D-res}.

\begin{defn}\label{defn:D-restriction}
  For $n\in \N$ and $m\in [n]_0$, let
  \begin{align*} 
      \A_{n,m} &:= \mathsf{D}_n \cup \left\{X_k
      ~\middle|~ k\in [n-m]\right\}.
  \end{align*}
\end{defn}

\begin{lem}\label{lem:D-Coxeter-res}
    Let $n\geq 2$ and $x\in \intpos(\mathsf{D}_n)$. If
    there exists $k\in [n]$ such that $x\subseteq V(X_k)$, then
    $\mathsf{D}_n^x\cong \mathsf{B}_{n-\rank(x)}$. Otherwise, $\mathsf{D}_n^x\cong
    \A_{n-\rank(x),m}$, for some $m\in[n]_0$. 
\end{lem}

\begin{proof}
  Set $P = \rhorep{D}{n}^{-1}(x)\in\Pi_{\mathsf{D},n}$, where $\rhorep{D}{n}$
  is the restriction of $\rhorep{B}{n}$ defined in~\eqref{eqn:type-B-iso}. We
  distinguish two cases: either there exists $k\in[n]$ such that
  $x\subseteq V(X_k)$ or not. In the first case, if such a $k$ exists, then
  $|P_0| \geq 3$. Let $K\in P\setminus\{P_0\}$ and $\epsilon\in\{-1,1\}$, and
  unbar $K$. So, for all $k\in K$ and $\ell\in P_0\setminus\{0\}$,
    \begin{align*} 
        x\cap V(X_k - \epsilon X_\ell) = x\cap V(X_{\min(K)} - \epsilon X_{\max(P_0)}).
    \end{align*}
    Thus, similar to Lemma~\ref{lem:AB-Coxeter-res}, with independent
    variables $\{X_I ~|~ I\in P\setminus\{P_0\}\}$,
    \begin{align}\label{eqn:type-D-res-1}
        \mathsf{D}_n^x &\cong \{X_I \pm X_J ~|~ I, J\in P\setminus
        \{P_0\}, I\neq J\} \cup \{X_K ~|~ K \in P\setminus\{P_0\}\}
        \cong \mathsf{B}_{n-\rank(x)}.
    \end{align}

    In the second case, if no such $k$ exists, then $P_0 = \{0\}$. Let
    $I, J, K\in P\setminus\{P_0\}$ with $I\ne J$ and $\epsilon \in
    \{-1,1\}$, and unbar $I$, $J$, and $K$. For all
    $i\in I$ and $j\in J$,
    \begin{align*} 
        x\cap V(X_i - \epsilon X_j) &= x\cap V(X_{\min(I)} -\epsilon
        X_{\min(J)}).
    \end{align*}
    If $|K|\geq 2$, then, for any two distinct $k,k'\in K$, the
    element $x$ is contained in either $V(X_k - X_{k'})$ or $V(X_k +
    X_{k'})$ but not both since $K\neq P_0$. Hence, for all
    $\epsilon\in\{-1, 1\}$,
    \begin{align*} 
      x\cap V(X_k - \epsilon X_{k'}) &= \begin{cases} x & \text{if }x\subseteq
        V(X_k -\epsilon X_{k'}), \\ x\cap V(X_{\min(K)}) &
        \text{otherwise}.
        \end{cases}
    \end{align*} 
    Therefore, for $m=|\{I ~|~ I\in P\setminus \{P_0\}, |I|=1\}|$, 
    \begin{align}\label{eqn:type-D-res-2}
        \mathsf{D}_n^x &\cong \{X_I \pm X_J ~|~ I, J\in P\setminus
        \{P_0\}, I\neq J\} \cup \{X_K ~|~ K \in P, |K|\geq
        2\}\cong \A_{n-\rk(x), m}. \qedhere
    \end{align} 
\end{proof}

Although the arrangements $\A_{n,m}$ are not necessarily Coxeter
arrangements, they are close enough: their Poincar\'e polynomials are
determined by Poincar\'e polynomials of classical Coxeter
arrangements. The next lemma applies the Deletion-Restriction
Theorem~\cite[Thm.~2.56]{OrlikTerao/92} to compute them.

\begin{lem}\label{lem:Poincare-D-res}
  For $n\in \N$ and $m\in [n]_0$, we have
  \begin{equation*}
      \pi_{\A_{n,m}}(Y) = (1 + (2n-m-1)Y) \prod_{i=1}^{n-1} (1 +
      (2i-1)Y).
  \end{equation*} 
\end{lem}

\begin{proof}
  We proceed by induction on $m$.  If $m=0$, then $\A_{n,m}\cong
  \mathsf{B}_n$ for all $n\in\N$, and the lemma holds; see, for
  instance, Example~\ref{ex:char.poly}.
    
    Now, let $m\in [n-1]_0$. Since $\A_{n, m+1} = \A_{n,m}
    \setminus\{X_{n-m}\}$, let $x= V(X_{n-m})$. It follows that
    $\A_{n,m}^x \cong \mathsf{B}_{n-1}$. By the Deletion-Restriction
    Theorem,
    \begin{align*} 
        \pi_{\A_{n, m+1}}(Y) &= \pi_{\A_{n,m}}(Y) -
        Y\pi_{\mathsf{B}_{n-1}}(Y),
    \end{align*}
    so the result follows by the induction hypothesis.
  \end{proof}

  We now turn to the combinatorial specifications of
  Lemma~\ref{lem:gen-strat.new} for the respective types.

  \subsubsection{Type $\mathsf{A}$}
\begin{lem}\label{lem:type-A_n-recursion}
  For all $n\geq 1$ and all cDVR $\lri$ with residue field
  cardinality~$q$,
    \begin{equation}\label{eqn:partition-recursion}
        \begin{aligned}
            \zeta_{\mathsf{A}_n(\lri)}^{\atom}\left((s_J)_{J\in\mcP(n+1;2)}\right)
            &= \dfrac{1}{1 - q^{-n-\sum_{J\in\mcP(n+1;2)}s_J}}
            \sum_{\substack{P \in \Pi_{n+1} \\ |P| \geq 2}}
            \pb[-q^{-1}]{|P|-1}! \\ &\quad \cdot \prod_{I\in P}
            q^{-(|I|-1)-\sum_{J\in \mcP(I;2)}
              s_J}\zeta_{\mathsf{A}_{|I|-1}(\lri)}^{\atom}\left((s_J)_{J\in\mcP(I;2)}\right)
            .
        \end{aligned}
    \end{equation}
\end{lem}

\begin{proof} 
  Set $\A=\mathsf{A}_n$. Then
  Lemma~\ref{lem:gen-strat.new}~\eqref{eqn:central-recursive} implies
  that
  \begin{multline}\label{eqn:from-lemma} 
            \zeta_{\A(\lri)}^{\atom}\left((s_L)_{L\in\A}\right) =
            \dfrac{1}{1 - q^{-n-\sum_{L\in\A} s_L}}
            \sum_{x\in\intpos(\A)\setminus\{\hat{1}\}}
            q^{-\codim(x)-\sum_{L\in\A_x} s_L} \\ \quad\cdot
            \pi_{\A^x}(-q^{-1})
            \zeta_{\A_x(\lri)}^{\atom}\left((s_L)_{L\in\A_x}\right) .
        \end{multline}        
        Let $P\in \Pi_{n+1}$ such that $|P|\geq 2$ and set
        $x = \rhorep{A}{n}(P)\in \toplat(\A)\setminus\{\hat{1}\}$, where
        $\rhorep{A}{n}$ is the isomorphism defined
        in~\eqref{eqn:type-A-iso}. We show that the $x$-summand
        in~\eqref{eqn:from-lemma} is the $P$-summand
        in~\eqref{eqn:partition-recursion}. By Lemma~\ref{lem:AB-Coxeter-res},
        $\A^x \cong \mathsf{A}_{|P|-1}$ and
        $\A_x\cong \prod_{I\in P}\mathsf{A}_{|I|-1}$, so
    \begin{align*} 
        \pi_{\A^x}(Y) &= \prod_{k=1}^{|P|-1} (1 + kY) =
        \pb[Y]{|P|-1}!,
    \end{align*} 
    and $\rk(x)=n+1-|P|$. For $\bfs = \left((s_J)_{J\in
      \mcP(I;2)}\right)_{I\in P}$, since $\A_x\cong \prod_{I\in
      P}\mathsf{A}_{|I|-1}$,
    \begin{align*} 
        \zeta_{\A_x(\lri)}^{\atom}\left(\bfs\right) &= \prod_{I\in P}
        \zeta_{\mathsf{A}_{|I|-1}(\lri)}^{\atom}\left((s_J)_{J\in\mcP(I;2)}\right).
    \end{align*}
    Hence, the lemma follows.
\end{proof}

\subsubsection{Type $\mathsf{B}$}
The proof of the recursive formula for type $\mathsf{B}$ is similar to
the type-$\mathsf{A}$ case in Lemma~\ref{lem:type-A_n-recursion}, so
we omit some details.

\begin{lem}\label{lem:type-B_n-recursion}
  For all $n\geq 1$ and all cDVRs $\lri$ with odd residue field
  cardinality~$q$,
    \begin{equation*}
        \begin{aligned} 
          \zeta_{\mathsf{B}_n(\lri)}^{\atom}\left((s_J)_{J\in\mcP_{\mathsf{B}}([n]_0;2)}\right)
          &= \dfrac{1}{1 - q^{-n-\sum_{J\in\mcP_{\mathsf{B}}([n]_0;2)}
              s_J}} \sum_{\substack{P\in \Pitwo{B}{n} \\ |P|\geq
              2}} \pb[-q^{-1}]{|P|-1}!! \\ &\quad \cdot q^{-(|P_0|-1)
            - \sum_{J\in\mcP_{\mathsf{B}}(P_0;2)} s_J}
          \zeta^{\atom}_{\mathsf{B}_{|P_0|-1}(\lri)}\left((s_J)_{J\in\mcP_{\mathsf{B}}(P_0;2)}\right)
          \\ &\quad \cdot \prod_{I\in P\setminus\{P_0\}} q^{-(|I|-1) -
            \sum_{J\in \mcP(I;2)}
            s_J}\zeta^{\atom}_{\mathsf{A}_{|I|-1}(\lri)}\left((s_J)_{J\in
            \mcP(I;2)}\right).
        \end{aligned}
    \end{equation*}
\end{lem}

\begin{proof}
  Set $\A= \mathsf{B}_n$. Let $P\in \Pitwo{B}{n}$ such that $|P|\geq 2$, and set
  $x=\rhorep{B}{n}(P)\in\toplat(\A)\setminus\{\hat{1}\}$, where $\rhorep{B}{n}$
  is the isomorphism defined in~\eqref{eqn:type-B-iso}. Using
  Lemma~\ref{lem:AB-Coxeter-res}, $\A_x\cong \mathsf{B}_{|P_0|-1}\times
  \prod_{I\in P\setminus\{P_0\}} \mathsf{A}_{|I|-1}$ and $\A^x\cong
  \mathsf{B}_{|P|-1}$. The rank of $x$ is $n+1-|P|$, the Poincar\'e polynomial
  is $\pi_{\A^x}(Y) = \pb[Y]{|P|-1}!!$. For $\bfs = \left((s_J)_{J\in
  \mcP_{\mathsf{B}}(I;2)}\right)_{I\in P}$,
    \begin{align*} 
        \zeta_{\A_x(\lri)}^{\atom}\left(\bfs\right) &=
        \zeta_{\mathsf{B}_{|P_0|-1}(\lri)}^{\atom}\left((s_J)_{J\in
          \mcP_{\mathsf{B}}(P_0;2)}\right) \prod_{I\in
          P\setminus\{P_0\}}
        \zeta_{\mathsf{A}_{|I|-1}(\lri)}^{\atom}\left((s_J)_{J\in
          \mcP(I;2)}\right). \qedhere
    \end{align*}
\end{proof}

\subsubsection{Type $\mathsf{D}$}

Recall that, for a central arrangement $\A$, the sum in
Lemma~\ref{lem:gen-strat.new}~\eqref{eqn:central-recursive} runs through
$\intpos(\A)\setminus\{\hat{1}\}$. As seen in Lemma~\ref{lem:D-Coxeter-res},
there are two different cases of set partitions in
$\Pitwo{D}{n}\cong\intpos(\mathsf{D}_n)$, so we split the sum in
Lemma~\ref{lem:gen-strat.new}~\eqref{eqn:central-recursive} into two parts as
follows: let
\begin{itemize}
\item $\zeta_{\mathsf{D}_n(\lri), 0}^{\atom} \left(\bfs\right)$ be the
  sum running through the $x\in\intpos(\mathsf{D}_n)$ such that $x$ is
  not contained in $V(X_k)$ for all $k\in[n]$, and
\item $\zeta_{\mathsf{D}_n(\lri), +}^{\atom}\left(\bfs\right)$ be the
  sum running through the $x\in\intpos(\mathsf{D}_n)$ such that there
  exists $k\in [n]$ such that $x\subseteq V(X_k)$.
\end{itemize} In other words, we split the sum for
$\zeta_{\mathsf{D}_n(\lri)}^{\atom}(\bfs)$ based on whether
$x\in\intpos(\mathsf{D}_n)$ is on a coordinate hyperplane or not. In terms of
set partitions of type $\mathsf{D}$, the summand
$\zeta_{\mathsf{D}_n(\lri), 0}^{\atom} \left(\bfs\right)$ is the sum over set
partitions with zero block equal to~$\{0\}$, whereas
$\zeta_{\mathsf{D}_n(\lri), +}^{\atom} \left(\bfs\right)$ is the sum over set
partitions with zero block not equal to~$\{0\}$.  Of course,
$\zeta_{\mathsf{D}_n(\lri)}^{\atom}(\bfs) = \zeta_{\mathsf{D}_n(\lri),
  0}^{\atom}(\bfs) + \zeta_{\mathsf{D}_n(\lri), +}^{\atom}(\bfs)$.

For $P\in\Pitwo{D}{n}$, denote the number of singleton nonzero blocks of $P$
by
\begin{align}\label{eqn:singletons}
    \mathsf{Sng}(P) &:= |\{I ~|~ I\in P\setminus\{P_0\},\; |I|=1\}|.
\end{align}

\begin{lem}\label{lem:type-D_n-recursion}
  Let $n\geq 2$. For all cDVRs $\lri$ with odd residue field
  cardinality~$q$,
    \begin{align*} 
        \zeta_{\mathsf{D}_n(\lri), 0}^{\atom}\left(\bfs\right) &=
        \dfrac{1}{1 - q^{-n-\sum_{J\in\mcP_{\mathsf{D}}([n]_0;2)}
            s_J}}\!\!  \sum_{\substack{P\in \Pitwo{D}{n} \\ P_0
            = \{0\}}}\!  \pb[-q^{-1}]{2|P| - \mathsf{Sng}(P) -
          3} \\ &\quad\cdot \pb[-q^{-1}]{|P|-2}!!\prod_{I\in
          P\setminus\{P_0\}} q^{-(|I|-1) -
          \sum_{J\in\mcP(I;2)}s_J}\zeta_{\mathsf{A}_{|I|-1}(\lri)}^{\atom}\left((s_J)_{J\in\mcP(I;2)}\right),
        \\ \zeta_{\mathsf{D}_n(\lri), +}^{\atom}\left(\bfs\right) &=
        \dfrac{1}{1 - q^{-n-\sum_{J\in\mcP_{\mathsf{D}}([n]_0;2)}
            s_J}}\!\!  \sum_{\substack{P\in \Pitwo{D}{n}
            \\ 3\leq |P_0| \leq n }}\!  \pb[-q^{-1}]{|P|-1}!!
        \\ &\quad\cdot q^{-(|P_0|-1) -
          \sum_{J\in\mcP_{\mathsf{D}}(P_0;2)}s_J}\zeta_{\mathsf{D}_{|P_0|-1}(\lri)}^{\atom}\left((s_J)_{J\in\mcP_{\mathsf{D}}(P_0;2)}\right)
        \\ &\quad\cdot \prod_{I\in P\setminus\{P_0\}} q^{-(|I|-1) -
          \sum_{J\in\mcP(I;2)}s_J}\zeta_{\mathsf{A}_{|I|-1}(\lri)}^{\atom}\left((s_J)_{J\in\mcP(I;2)}\right).
    \end{align*}
\end{lem}

\begin{proof}
  Let $P\in\Pitwo{D}{n}$ such that $|P|\geq 2$, and set $x =
  \rhorep{D}{n}(P)$, where $\rhorep{D}{n}$ is the
  restriction of the isomorphism defined in~\eqref{eqn:type-B-iso}.
  It follows that $(\mathsf{D}_n)_x \cong \mathsf{D}_{|P_0|-1}\times
  \prod_{I\in P\setminus\{P_0\}} \mathsf{A}_{|I| - 1}$ and that $\rk(x)=n-|P|+1$.
        
  Now we consider two cases based on the cardinality~$|P_0|$. First, suppose
  $|P_0|=1$, and recall from Definition~\ref{defn:D-restriction} the arrangement
  $\A_{n, m}$. The proof of Lemma~\ref{lem:D-Coxeter-res} implies that $\mathsf{D}_n^x\cong
  \A_{|P|-1, \mathsf{Sng}(P)}$.  By Lemma~\ref{lem:Poincare-D-res},
    \begin{align*} 
        \pi_{\mathsf{D}_n^x}(Y) &= (1 + (2|P| - \mathsf{Sng}(P) - 3)Y)
        \prod_{i=1}^{|P|-2} (1 + (2i-1)Y) \\ &= \pb[Y]{2|P| -
          \mathsf{Sng}(P) - 3} \pb[Y]{|P| - 2}!!.
    \end{align*} 
    Now suppose that $|P_0|>1$. Recall that set partitions
    of type $\mathsf{D}$ have zero blocks of size different from~$2$,
    so $3 \leq |P_0|\leq n$ in this case. Then $\mathsf{D}_n^x\cong
    \mathsf{B}_{|P|-1}$ by Lemma~\ref{lem:D-Coxeter-res}, so the lemma
    follows.
\end{proof}

\subsubsection{Proof of Theorem~\ref{thm:total-partition}} 

Apply 
Lemmas~\ref{lem:type-A_n-recursion},~\ref{lem:type-B_n-recursion},
or~\ref{lem:type-D_n-recursion} recursively. We observe that because
the summations in each of the lemmas require the set partition
$P\in\Pitwo{X}{n}$ to have size at least 2, the final sum runs through
all total partitions. At every instance in the recursion, $|P|$ is the
number of children of some parent vertex in $\tau$. In type
$\mathsf{D}$, if $P\in\Pitwo{\mathsf{D}}{n}$ such that $|P_0|=1$,
then, for some $\tau$, the number $\mathsf{Sng}(P)+1$ is the number of
children of~$v_0^+(\tau)$ that are unbranched.

This completes the proof of Theorem~\ref{thm:total-partition} and thus
establishes Theorem~\ref{thm:atom.braid}. \hfill$\square$

\section{Coarse flag Hilbert--Poincar\'e series}\label{sec:coarse}

Our main aim in this section is to prove
Theorem~\ref{thm:coarse.Y=1}. In Section~\ref{subsec:gen.coarse} we
record some results pertaining to coarse flag Hilbert--Poincar\'e
series of general (not necessarily Coxeter) hyperplane arrangements,
including some of their special values, behavior at $Y=0$, and
Hadamard products. We give explicit formulae for Boolean and generic
central arrangements in Section~\ref{subsec:exa.coarse}. (Many more
can be found in Appendix~\ref{sec:app.exa.coarse}.) In
Section~\ref{subsec:setup.coarse} we prove Theorem~\ref{thm:f-vector},
the special case of Theorem~\ref{thm:coarse.Y=1} for classical
arrangements. In Section~\ref{subsec:proof.coarse.Y=1} we combine
these results to prove Theorem~\ref{thm:coarse.Y=1}.

\subsection{General properties of coarse flag Hilbert--Poincar\'e series}
\label{subsec:gen.coarse} 

In this section, let $\A$ be an arbitrary hyperplane arrangement, i.e.~not
necessarily defined over a field of characteristic zero. Recall the
definition~\eqref{def:numer.coarse} of the numerator polynomial
$\mcN_{\A}(Y,T)$ of the coarse flag Hilbert--Poincar\'e series
$\cfHP_{\A}(Y,T)$.

We start with a general observation. We denote by $\max(\Delta)$ the
set of maximal-dimensional faces of a simplicial complex~$\Delta$.

\begin{lemma}\label{lem:basic.coarse}
  The following hold:
    \begin{enumerate}
        \item\label{eins} $\cfHP_{\A}(Y, 0) = \pi_{\A}(Y)$,
        \item $\mathcal{N}_{\A}(Y, 1) = |\max(\Delta(\toplat(\A)))| (1
          + Y)^{\rank(\A)}$.
    \end{enumerate}
\end{lemma}

\begin{proof}
    Setting $m = \rk(\A) - \delta_{\widehat{1}\in\intpos(\A)}$, both
    follow from the equation
    \begin{equation*} \label{eqn:coarse-numerator}
        \mathcal{N}_{\A}(Y, T) = \sum_{F\in\Delta(\proplat(\A))}
        \pi_F(Y) T^{|F|} (1 - T)^{m-|F|}. \qedhere
    \end{equation*}
\end{proof}

It is known that there exist inequivalent arrangements with the same
Poincar\'e polynomial~\cite[Ex.~2.61]{OrlikTerao/92}. All of our
results and computations suggest that the following question may have
a negative answer.

\begin{quest} Do there exist two inequivalent arrangements $\A$ and
  $\mathcal{B}$ such that $\cfHP_{\A}(Y, T)=\cfHP_{\mathcal{B}}(Y,
  T)$?
  \end{quest}

\subsubsection{Behavior at $Y=0$}

Conjecture~\ref{conj} is partly motivated by Proposition~\ref{prop:coarse.Y=0},
which follows immediately from deep yet well-known results in poset topology.

\begin{proof}[Proof of Proposition~\ref{prop:coarse.Y=0}]
    The first part of the statement follows from
    Lemma~\ref{lem:basic.coarse}\eqref{eins}.  For the second part, let $\Delta
    := \Delta(\intpos(\A))$. Let $\fieldSR$ be a field and
    $\fieldSR[\Delta]$ be the Stanley--Reisner ring of~$\Delta$.
    Since $\pi_F(0)=1$ for all $F\in \Delta$, we find that
    \begin{align*} 
       \frac{1}{1-T_{\hat{0}}} \fHP_{\A}(0, \bfT) &= \sum_{F\in\Delta}
       \prod_{x\in F} \gp{T_x} = \Hilb(\fieldSR[\Delta], \bfT),
    \end{align*}
the fine Hilbert series of $\fieldSR[\Delta]$;
cf.~\cite[Chap.~10.6]{Petersen/15}. As $\intpos(\A)$ is a geometric
semilattice, $\Delta$ is Cohen--Macaulay over $\fieldSR$; see, for
instance, \cite{WW:semilattices}.
By~\cite[Cor.~3.2]{Stanley:commutative}, there exists $h_k\in \N_0$,
for $k\in [\rank(\A)]_0$, such that
    \begin{align*} 
        \Hilb(\fieldSR[\Delta], T) &= \dfrac{\sum_{k=0}^{\rank(\A)}
          h_kT^k}{(1-T)^{\rank(\A)+1}}=\frac{1}{1-T}\cfHP_{\A}(0, T)
        . \qedhere
    \end{align*} 
\end{proof}

\subsubsection{Hadamard products}
\label{subsec:hadamard}

In this section, we consider the effect of taking direct products of
hyperplane arrangements on flag Hilbert--Poincar\'e series. This turns out to
be described by Hadamard products of a moderate variant of~$\cfHP_{\A}(Y,
T)$. Set
\begin{equation*}
    \widehat{\cfHP_{\A}}(Y, T) := \sum_{F\in\Delta(\intpos(\A))}
    \pi_F(Y) \left(\dfrac{T}{1-T}\right)^{|F|} = \dfrac{\cfHP_{\A}(Y,
      T)}{1- T}.
\end{equation*}

Let us expand $\widehat{\cfHP_{\A}}(Y, T)$ as a generating
function. First define $\alpha_0(\A;Y) := \pi_{\A}(Y)$ and, for $k\geq 1$,
\begin{align*} 
    \alpha_k(\A;Y) &:= \sum_{\substack{F\in\Delta(\mathcal{L}(\mathcal{A})) \\ 1\leq |F|\leq k}} \binom{k-1}{|F| - 1}\pi_F(Y).
\end{align*}
Since $\left(T/(1-T)\right)^f = \sum_{k\geq 0}\binom{k+f-1}{f-1}T^{f+k}$ for all
$f\in\N$, it follows that
\begin{align}\label{equ:gen-func-cfHP}
  \widehat{\cfHP_{\A}}(Y, T) &= \sum_{k\geq 0}\alpha_k(\A; Y)T^k.
\end{align}
Recall that, given sequences $(\phi_k), (\psi_k)\in\Q^{\N_0}$, the
\emph{Hadamard product} of the generating functions
$f(T)= \sum_{k\geq 0} \phi_kT^k$ and $g(T) = \sum_{k\geq 0} \gamma_kT^k$ along
$T$ is
\begin{align*}
    f(T) \star_T g(T) &= \sum_{k\geq 0} \phi_k\gamma_kT^k.
\end{align*}

\begin{prop}\label{prop:Hadamard}
  Given hyperplane arrangements $\A_1$ and $\A_2$, we have
    \begin{align*} 
        \widehat{\cfHP_{\A_1\times\A_2}}(Y, T) &=
        \widehat{\cfHP_{\A_1}}(Y, T) \star_T
        \widehat{\cfHP_{\A_2}}(Y, T).
    \end{align*}
\end{prop}

Before proving Proposition~\ref{prop:Hadamard}, we prove a lemma using basic
facts about the \emph{Delannoy numbers} $D(m,n)$. These count the number of
lattice paths from $(0,0)$ to $(m,n)\in\N_0\times\N_0$ using steps in
$\{(1,0), (1,1), (0,1)\}$; cf., e.g., \cite[p.~81]{Comtet/74}.  Let
$D(m,n,\ell) = \binom{\ell-1}{m} \binom{m}{\ell-n-1}$ be the number of such
lattice paths traversing exactly $\ell$ vertices. Hence
\begin{align*} 
    D(m, n) = \sum_{\ell=1}^{m+n+1}D(m,n,\ell) = \sum_{\ell = 1}^{m+n+1} \binom{\ell-1}{m} \binom{m}{\ell-n-1}.
\end{align*}

We shall use the following notation in the sequel. If $F$ is a weakly
increasing flag, we denote by $F^<$ the maximal strictly increasing sub-flag
of~$F$.

\begin{lem}\label{lem:Hadamard-product}
 Given hyperplane arrangements $\A_1$ and $\A_2$, we have
    \begin{align*} 
        \widehat{\cfHP_{\A_1\times\A_2}}(0, T) &=
        \widehat{\cfHP_{\A_1}}(0, T) \star_T \widehat{\cfHP_{\A_2}}(0, T).
    \end{align*}
\end{lem}   

\begin{proof}
  By~\eqref{equ:gen-func-cfHP}, it suffices to show that, for all $k\geq 0$,
    \begin{align}\label{eqn:alpha-mult}
        \alpha_k(\A_1\times \A_2;0) &= \alpha_k(\A_1;0)\alpha_k(\A_2;0). 
    \end{align}
    This is clear for $k=0$, so we assume $k\geq 1$ is fixed.  Note that
    $\intpos(\A_1\times \A_2)\cong \intpos(\A_1)\times \intpos(\A_2)$. The
    latter is partially ordered as follows: $(x_1,x_2)\leq (y_1,y_2)$ if
    $x_1\leq y_1$ and $x_2\leq y_2$; see~\cite[Prop.~2.14]{OrlikTerao/92}. For
    $i\in \{1,2\}$, let
    $\phi_i : \intpos(\A_1\times \A_2)\rightarrow \intpos(\A_i)$ be the
    respective projection.
    
    For $i\in\{1,2\}$, fix nonempty $G_i\in\Delta(\intpos(\A_i))$, and set 
    \begin{multline*} 
      \mathcal{F}_k(G_1, G_2) = \{(x_1 < \cdots < x_{\ell})\in\Delta(\intpos(\A_1\times\A_2)) ~|~ \\
      \ell\in [k],\;\forall i\in\{1,2\}:\; G_i = (\phi_i(x_1) \leq \cdots \leq \phi_i(x_{\ell}))^< \}.
    \end{multline*}
    The set of flags $\mathcal{F}_k(G_1, G_2)$ is in bijection with the set of
    lattice paths from $(0,0)$ to $(|G_1|-1, |G_2|-1)$ with steps in
    $\{(1,0), (1,1), (0,1)\}$, traversing no more than $k$ lattice points. For
    $\ell\in \N$ we define, in addition,
    \begin{multline*} 
      \mathcal{F}(G_1, G_2, \ell) = \{(x_1 < \cdots < x_{\ell})\in\Delta(\intpos(\A_1\times\A_2)) ~|~ \\
      \forall i\in\{1,2\}:\; G_i = (\phi_i(x_1) \leq \cdots \leq \phi_i(x_{\ell}))^< \} .
    \end{multline*}
    Flags in $\mathcal{F}_k(G_1, G_2)$ have lengths between $1$ and $k$,
    whereas flags in $\mathcal{F}(G_1, G_2, \ell)$ all have length $\ell$. The
    latter set is in bijection with the set of Delannoy paths from $(0,0)$ to
    $(|G_1|-1,|G_2|-1)$ traversing $\ell$ vertices. Considering the partition
    \begin{align*} 
      \mathcal{F}_k(G_1, G_2) &= \bigcup_{\ell=1}^{k}\mathcal{F}(G_1, G_2,\ell).
    \end{align*}  and setting $m=|G_1|-1$ and $n=|G_2|-1$, we obtain
    \begin{equation}\label{equ:like.saalschuetz} 
      \begin{split}
        \sum_{F\in\mathcal{F}_k(G_1, G_2)} \binom{k - 1}{|F|-1} &= \sum_{\ell=1}^{k}
        \binom{k-1}{\ell-1}D(m,n,\ell) \\ &=
        \binom{k-1}{m}\binom{k-1}{n} = \binom{k-1}{|G_1| - 1}
        \binom{k-1}{|G_2| - 1}.
      \end{split}
    \end{equation} 
    Here, the penultimate equality is obtained by counting the subsets of the
    form $A\times B$, for $A,B\subseteq [k-1]$ with $|A|=m$ and $|B|=n$,
    setting $|A\cup B|=\ell-1$. Since~\eqref{equ:like.saalschuetz} holds for
    all nonempty $G_1\in\Delta(\intpos(\A_1))$ and
    $G_2\in\Delta(\intpos(\A_2))$, equation~\eqref{eqn:alpha-mult} holds.
\end{proof}

\begin{remark}
  The second equality in~\eqref{equ:like.saalschuetz} seems reminiscent of,
  but distinct from Saalsch\"utz's identity; see \cite[(5.28)]{Knuth/94}.
\end{remark}

\begin{quest}
    Can the Hadamard factorization in Lemma~\ref{lem:Hadamard-product}
    be deduced directly from the interpretation
    $\widehat{\cfHP_{\A}}(0, T) =
    \Hilb(\fieldSR[\Delta(\mathcal{L}(\A))],T)$?
\end{quest}

\begin{proof}[Proof of Proposition~\ref{prop:Hadamard}]
  Fix an isomorphism
  $\varphi : \intpos(\A_1\times \A_2) \rightarrow \intpos(\A_1) \times
  \intpos(\A_2)$. For $x\in\intpos(\A_1\times \A_2)$, write $\phi_i(x)$ for
  $\phi(x) = (\phi_1(x), \phi_2(x))$, and for
  $F=(x_1 < \cdots < x_k)\in \Delta(\intpos(\A_1\times \A_2))$, define
  $G_i = (\phi_i(x_1) \leq \cdots \leq \phi_i(x_k))$ for each
  $i\in\{1,2\}$. Note $\pi_{G_i}(Y) = \pi_{G_i^{<}}(Y)$. Since
  $\pi_{\A\times \A'}(Y) = \pi_{\A}(Y)\pi_{\A'}(Y)$ for hyperplane
  arrangements $\A$ and $\A'$ (cf.\ \cite[Lem.~2.50]{OrlikTerao/92}), it
  follows that $\pi_F(Y) = \pi_{G_1^{<}}(Y) \pi_{G_2^{<}}(Y)$. The proposition
  follows from Lemma~\ref{lem:Hadamard-product}.
\end{proof}

\subsection{Families of examples}\label{subsec:exa.coarse}

We record formulae of coarse flag Hilbert--Poincar\'e series for some infinite
families of examples. Appendix~\ref{sec:app.exa.coarse} contains many more.

\subsubsection{Boolean arrangements}
\label{subsec:Boolean-arrangements}

Recall the definition~\eqref{def:euler.poly} of the $n$th Eulerian
polynomial~$E_n(T)$ and the definition~\eqref{eqn:WO} of the weak order zeta
function~$\mathcal{I}_n^{\mathrm{WO}}(\bfT)$.  Specializing $T_I=T$ for each
nonempty $I\subseteq [n]$ in~\eqref{eqn:WO} yields
\begin{align*} 
    \mathcal{I}_n^{\mathrm{WO}}(T) &= \dfrac{E_n(T)}{(1 - T)^n}.
\end{align*}
The following result is thus an immediate consequence of
Proposition~\ref{prop:boolean}.
\begin{prop}
    For $n\geq 1$,
    \begin{align*} 
        \cfHP_{\mathsf{A}_1^n}(Y,T) &= (1+Y)^n \dfrac{E_n(T)}{(1-T)^n}.
    \end{align*}
\end{prop}

\subsubsection{Generic central arrangements}
\label{subsec:generic-cen.coarse}

Proposition~\ref{prop:uniform-zeta} implies the following result.

\begin{prop}\label{prop:gen-cen.coarse}
    Let $m,n\in\N$ with $n\leq m$. Then
    \begin{align*}
      \cfHP_{\mathcal{U}_{n, m}}(Y, T) &= \dfrac{1+Y}{1-T} \sum_{k=0}^{n-1}\binom{m-1}{k}Y^k \\
      &\quad + \sum_{\ell=1}^{n-1}\sum_{k=0}^{n-\ell-1} \binom{m-\ell-1}{k} \binom{m}{\ell} \dfrac{(1+Y)^{\ell+1}Y^k T \cdot E_{\ell}(T)}{(1-T)^{\ell+1}}.
    \end{align*} 
\end{prop}

For $m\geq n=2$, Proposition~\ref{prop:gen-cen.coarse} recovers a consequence
of Proposition~\ref{pro:Im}, viz. 
\begin{align*} 
 \mathcal{N}_{\mathcal{U}_{2, m}}(Y,T) &= (1+Y)(1 + (m-1)(Y+T) + YT).
\end{align*}  
For $m\geq n=3$, Proposition~\ref{prop:gen-cen.coarse} states that
\begin{multline*}
\mathcal{N}_{\mathcal{U}_{3, m}}(Y,T) = (1+Y) \left\{ (1+Y^2T^2)+(m-1)Y(1+T^2) +2(m-1)^2YT \right.\\\left. + \left( \binom{m+1}{2}-2\right)T(1+Y^2) + \left(\frac{m(m-3)}{2}+1\right)(Y^2+T^2)\right\}.
\end{multline*}
With Lemma~\ref{lem:basic.coarse}, we obtain $\pi_{\mathcal{U}_{3,m}}(1)
= \mathcal{N}_{\mathcal{U}_{3,m}}(1,0) = m^2-m+2$ and hence
\begin{align}\label{eqn:U3m-Y=1}
  \frac{\mathcal{N}_{\mathcal{U}_{3,
        m}}(1,T)}{\pi_{\mathcal{U}_{3,m}}(1)} &= 1 + \left(6 -
  \dfrac{16}{m^2-m+2}\right)T + T^2.
\end{align} 
Thus the polynomial in~\eqref{eqn:U3m-Y=1} is in $\Z[T]$ if and only if
$m=3$. We note that $m=3$ is the unique value for which $\mathcal{U}_{3,m}$ is
a Coxeter arrangement.

We have found only three pairs $(n,m)\in [1000]^2$ with $m>n>3$ with
the property that
$\mathcal{N}_{\mathcal{U}_{n,m}}(1,T)/\pi_{\mathcal{U}_{n,m}}(1)
\in\Z[T]$, namely those in $\{(4,5),(4,7),(4,8)\}$; see
Section~\ref{subsec:gen.cen.arr} for explicit formulae. The associated
normalized integral polynomials at $Y=1$ are, respectively,
\begin{align*} 
    &1 +15T + 15T^2 + T^3, & &1 +19T + 19T^2 + T^3, & &1 +20T + 20T^2 + T^3.
\end{align*} 
From this perspective, it seems rare that $\mathcal{N}_{\A}(1,T)/\pi_{\A}(1)
\in\Z[T]$; indeed, of all the non-Coxeter examples computed in
Appendix~\ref{sec:app.exa.coarse}, only $\mathcal{U}_{4,5}$,
$\mathcal{U}_{4,7}$, and $\mathcal{U}_{4,8}$ satisfy this property. 

\begin{quest}\label{quest:integrality}\
  \begin{enumerate}
  \item For which $m>n$ does the following hold: 
    $$\mathcal{N}_{\mathcal{U}_{n,m}}(1,T)/\pi_{\mathcal{U}_{n,m}}(1)\in\Z[T]?$$
  \item\label{quest:int-1} Is there an infinite set $\mathcal{F}$ of
    non-Coxeter arrangements such that
    $$\mathcal{N}_{\A}(1,T)/\pi_{\A}(1) \in\Z[T]$$ for all $\A\in\mathcal{F}$?
  \item Is there an infinite set $\mathcal{G}$ of arrangements with the property
    that $$\mathcal{N}_{\A}(1,T)/\pi_{\A}(1)\neq E_m(T)$$ such that for all
    $\A\in\mathcal{G}$ and~$m\in\N$?
  \end{enumerate}
\end{quest}

We remark that the non-Coxeter restrictions of $\mathsf{D}_n$ are a good
candidate family for answering Question~\ref{quest:integrality}
(\ref{quest:int-1}) positively; see Definition~\ref{defn:D-restriction} and
Section~\ref{sec:app.coarse.res-D}.

\subsection{Classical Coxeter arrangements}\label{subsec:setup.coarse}

Here we prove Theorem~\ref{thm:coarse.Y=1} for classical Coxeter
arrangements. This will be the key step in the proof of the general case given
in~Section~\ref{subsec:proof.coarse.Y=1}. Recall that $S(n,k)$ are the
Stirling numbers of the second kind.  We denote the set of flags of
$\proplat(\A)$ with length $k\in \N_0$ by $\Delta_{k}(\proplat(\A))$.

\begin{thm}[Theorem~\ref{thm:coarse.Y=1} for classical types]\label{thm:f-vector}
    For $1 \leq k \leq n$ and $\mathsf{X}\in\{\mathsf{A},
    \mathsf{B},\mathsf{D}\}$,
    \begin{align*} 
      \sum_{\substack{F\in\Delta_{k-1}(\proplat(\mathsf{X}_n))}}
      \dfrac{\pi_F(1)}{\pi_{\mathsf{X}_n}(1)} &= k! \; S(n,k).
    \end{align*} 
\end{thm}

The proof of Theorem~\ref{thm:f-vector} will occupy most of the rest
of this section and be completed in
Section~\ref{subsubsec:stirling.proof}.

\subsubsection{Rooted and plane trees}

To prove Theorem~\ref{thm:f-vector}, we enumerate sets of rooted trees
using various maps which we describe in turn. We build on terminology
from~\cite[Appendix]{Stanley:Vol1}, specifically for (plane) rooted
trees, and from Sections~\ref{subsec:TP}
and~\ref{subsec:atom.root.trees}. In particular we transfer
terminology for $\mathsf{X}_n$-labeled rooted trees introduced in
Definition~\ref{def:X-label} and ensuing identifications to plane
trees.

The \emph{$k$th generation} of a rooted tree $\tau$ comprises all the
vertices of distance $k$ from the root; i.e.\ one traverses $k$ distinct edges
from the root to such a vertex. The $k$th generation of $\tau$ is
\emph{nontrivial} if there exists a vertex, of distance $k$ from the root,
with at least two children.

\begin{defn}\label{def:PT.LPT.LRT}
  Let $n,k\in\N$ and $\mathsf{X}\in\{\mathsf{A}, \mathsf{B},
  \mathsf{D}\}$.
  \begin{itemize} 
    \item Let $\PT{n}{k}$ be the set of (unlabeled) plane trees with
      $n$ leaves and $k$ generations, all of which are nontrivial;
      without loss of generality, all leaves are in the $k$th
      generation.
    \item Let $\LPT{n}{k}{X}$ resp.\ $\LRT{n}{k}{X}$ be the set of
    $\mathsf{X}_n$-labeled plane resp.\ rooted trees with $k$ generations, all
    of which are nontrivial; without loss of generality, all leaves are in the
    $k$th generation.
 \end{itemize}
\end{defn}
We remark that trees of $\PT{n}{k}$ are said to be of \emph{length}~$k$
in~\cite[Appendix]{Stanley:Vol1}, and the trees in $\PT{n}{k}$ have one less
leaf than trees in $\LPT{n}{k}{X}$.

The following lemma, which gives a simple formula for the cardinality of
$\PT{n+1}{k}$ in terms of Stirling numbers of the second kind, may be well
known. It is implicit in Cayley's work~\cite{cayley_2009} from~1859. Having
failed to locate it in the modern literature, we include its proof.

\begin{lem}\label{lem:plane-trees}
  For $1 \leq k \leq n$, $\left|\PT{n+1}{k}\right| = k! \;S(n, k)$.
\end{lem}

\begin{proof}
  The quantity $k!\;S(n, k)$ is the number of words of length $n$ on the
  alphabet $[k-1]_0$, where each integer is used at least once. It thus
  suffices to identify such words with trees in~$\PT{n+1}{k}$.

  Let $\tau\in\PT{n+1}{k}$ and label the leaves of $\tau$ from $1$ to $n+1$
  from left to right. We obtain a word $w(\tau)$ of length $n$ on the alphabet
  $[k-1]_0$ whose $i$th character is the generation number of the common
  ancestor of leaves labeled $i$ and~$i+1$. If the leaves labeled $i$ and
  $i+1$ are siblings, then the $i$th character of $w(\tau)$ is $k-1$.
  
  We claim that each $j\in[k-1]_0$ appears in $w(\tau)$. Indeed, since $\tau$
  has $k$ generations there exists a vertex in the $j$th generation with at
  least two consecutive children, say $u_1$ followed by~$u_2$. Then the
  rightmost leaf that is a descendant of $u_1$ is the left neighbor of the
  leftmost leaf that is a descendant of~$u_2$. Hence, $j$ appears in $w(\tau)$.
  
  Reversing these steps for a word of length $n$ on the alphabet
  $[k-1]_0$ where each integer is used at least once determines a
  unique tree~$\tau\in\PT{n+1}{k}$. 
\end{proof}

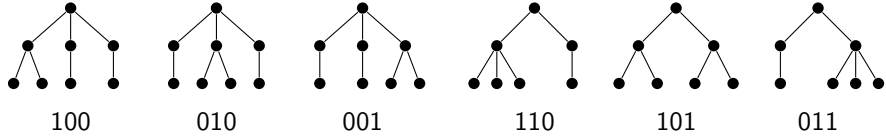
\begin{figure}[h]
  \centering
  \begin{subfigure}{0.15\textwidth}
    \centering
    \begin{tikzpicture}
      \pgfmathsetmacro{\x}{0.75} 
      \pgfmathsetmacro{\y}{0.5}
      \node(1) at (0, 2*\y) [circle, fill=black, inner sep=1.5pt] {}; 
      \node(2) at (-0.75*\x, \y) [circle, fill=black, inner sep=1.5pt] {}; 
      \node(3) at (0, \y) [circle, fill=black, inner sep=1.5pt] {}; 
      \node(4) at (0.75*\x, \y) [circle, fill=black, inner sep=1.5pt] {}; 
      \node(5) at (-1*\x, 0) [circle, fill=black, inner sep=1.5pt] {}; 
      \node(6) at (-0.5*\x, 0) [circle, fill=black, inner sep=1.5pt] {}; 
      \node(7) at (0, 0) [circle, fill=black, inner sep=1.5pt] {}; 
      \node(8) at (0.75*\x, 0) [circle, fill=black, inner sep=1.5pt] {}; 
      \node at (0,-1*\y) {$\mathsf{100}$};
      
      \draw[-] (1) -- (2); 
      \draw[-] (1) -- (3);
      \draw[-] (1) -- (4); 
      \draw[-] (2) -- (5);
      \draw[-] (2) -- (6);
      \draw[-] (3) -- (7);
      \draw[-] (4) -- (8);
    \end{tikzpicture}
  \end{subfigure}~%
  \begin{subfigure}{0.15\textwidth}
    \centering
    \begin{tikzpicture}
      \pgfmathsetmacro{\x}{0.75} 
      \pgfmathsetmacro{\y}{0.5}
      \node(1) at (0, 2*\y) [circle, fill=black, inner sep=1.5pt] {}; 
      \node(2) at (-0.75*\x, \y) [circle, fill=black, inner sep=1.5pt] {}; 
      \node(3) at (0, \y) [circle, fill=black, inner sep=1.5pt] {}; 
      \node(4) at (0.75*\x, \y) [circle, fill=black, inner sep=1.5pt] {}; 
      \node(5) at (-0.75*\x, 0) [circle, fill=black, inner sep=1.5pt] {}; 
      \node(6) at (-0.25*\x, 0) [circle, fill=black, inner sep=1.5pt] {}; 
      \node(7) at (0.25*\x, 0) [circle, fill=black, inner sep=1.5pt] {}; 
      \node(8) at (0.75*\x, 0) [circle, fill=black, inner sep=1.5pt] {}; 
      \node at (0,-1*\y) {$\mathsf{010}$};
      
      \draw[-] (1) -- (2); 
      \draw[-] (1) -- (3);
      \draw[-] (1) -- (4); 
      \draw[-] (2) -- (5);
      \draw[-] (3) -- (6);
      \draw[-] (3) -- (7);
      \draw[-] (4) -- (8);
    \end{tikzpicture}
  \end{subfigure}~%
  \begin{subfigure}{0.15\textwidth}
    \centering
    \begin{tikzpicture}
      \pgfmathsetmacro{\x}{0.75} 
      \pgfmathsetmacro{\y}{0.5}
      \node(1) at (0, 2*\y) [circle, fill=black, inner sep=1.5pt] {}; 
      \node(2) at (-0.75*\x, \y) [circle, fill=black, inner sep=1.5pt] {}; 
      \node(3) at (0, \y) [circle, fill=black, inner sep=1.5pt] {}; 
      \node(4) at (0.75*\x, \y) [circle, fill=black, inner sep=1.5pt] {}; 
      \node(5) at (-0.75*\x, 0) [circle, fill=black, inner sep=1.5pt] {}; 
      \node(6) at (0, 0) [circle, fill=black, inner sep=1.5pt] {}; 
      \node(7) at (0.5*\x, 0) [circle, fill=black, inner sep=1.5pt] {}; 
      \node(8) at (1*\x, 0) [circle, fill=black, inner sep=1.5pt] {}; 
      \node at (0,-1*\y) {$\mathsf{001}$};
      
      \draw[-] (1) -- (2); 
      \draw[-] (1) -- (3);
      \draw[-] (1) -- (4); 
      \draw[-] (2) -- (5);
      \draw[-] (3) -- (6);
      \draw[-] (4) -- (7);
      \draw[-] (4) -- (8);
    \end{tikzpicture}
  \end{subfigure}~%
  \begin{subfigure}{0.15\textwidth}
    \centering
    \begin{tikzpicture}
      \pgfmathsetmacro{\x}{0.75} 
      \pgfmathsetmacro{\y}{0.5}
      \node(1) at (0, 2*\y) [circle, fill=black, inner sep=1.5pt] {}; 
      \node(2) at (-0.66*\x, \y) [circle, fill=black, inner sep=1.5pt] {}; 
      \node(3) at (0.66*\x, \y) [circle, fill=black, inner sep=1.5pt] {}; 
      \node(4) at (-1.06*\x, 0) [circle, fill=black, inner sep=1.5pt] {}; 
      \node(5) at (-0.66*\x, 0) [circle, fill=black, inner sep=1.5pt] {}; 
      \node(6) at (-0.26*\x, 0) [circle, fill=black, inner sep=1.5pt] {}; 
      \node(7) at (0.66*\x, 0) [circle, fill=black, inner sep=1.5pt] {}; 
      \node at (0,-1*\y) {$\mathsf{110}$};
      
      \draw[-] (1) -- (2); 
      \draw[-] (1) -- (3);
      \draw[-] (2) -- (4); 
      \draw[-] (2) -- (5);
      \draw[-] (2) -- (6);
      \draw[-] (3) -- (7);
    \end{tikzpicture}
  \end{subfigure}~%
  \begin{subfigure}{0.15\textwidth}
    \centering
    \begin{tikzpicture}
      \pgfmathsetmacro{\x}{0.75} 
      \pgfmathsetmacro{\y}{0.5}
      \node(1) at (0, 2*\y) [circle, fill=black, inner sep=1.5pt] {}; 
      \node(2) at (-0.66*\x, \y) [circle, fill=black, inner sep=1.5pt] {}; 
      \node(3) at (0.66*\x, \y) [circle, fill=black, inner sep=1.5pt] {}; 
      \node(4) at (-1*\x, 0) [circle, fill=black, inner sep=1.5pt] {}; 
      \node(5) at (-0.33*\x, 0) [circle, fill=black, inner sep=1.5pt] {}; 
      \node(6) at (0.33*\x, 0) [circle, fill=black, inner sep=1.5pt] {}; 
      \node(7) at (\x, 0) [circle, fill=black, inner sep=1.5pt] {}; 
      \node at (0,-1*\y) {$\mathsf{101}$};
      
      \draw[-] (1) -- (2); 
      \draw[-] (1) -- (3);
      \draw[-] (2) -- (4); 
      \draw[-] (2) -- (5);
      \draw[-] (3) -- (6);
      \draw[-] (3) -- (7);
    \end{tikzpicture}
  \end{subfigure}~%
  \begin{subfigure}{0.15\textwidth}
    \centering
    \begin{tikzpicture}
      \pgfmathsetmacro{\x}{0.75} 
      \pgfmathsetmacro{\y}{0.5}
      \node(1) at (0, 2*\y) [circle, fill=black, inner sep=1.5pt] {}; 
      \node(2) at (-0.66*\x, \y) [circle, fill=black, inner sep=1.5pt] {}; 
      \node(3) at (0.66*\x, \y) [circle, fill=black, inner sep=1.5pt] {}; 
      \node(4) at (-0.66*\x, 0) [circle, fill=black, inner sep=1.5pt] {}; 
      \node(5) at (0.26*\x, 0) [circle, fill=black, inner sep=1.5pt] {}; 
      \node(6) at (0.66*\x, 0) [circle, fill=black, inner sep=1.5pt] {}; 
      \node(7) at (1.06*\x, 0) [circle, fill=black, inner sep=1.5pt] {}; 
      \node at (0,-1*\y) {$\mathsf{011}$};
      
      \draw[-] (1) -- (2); 
      \draw[-] (1) -- (3);
      \draw[-] (2) -- (4); 
      \draw[-] (3) -- (5);
      \draw[-] (3) -- (6);
      \draw[-] (3) -- (7);
    \end{tikzpicture}
  \end{subfigure}
  \caption{Illustration of the proof of Lemma~\ref{lem:plane-trees} for $n=3$ and~$n=2$. The six elements of $\PT{4}{2}$ match with the words in $\{0,1\}$ of length $3$ which feature both characters.}
\end{figure}

To define the following maps, let $1 \leq k \leq n$
and~$\mathsf{X}\in\{\mathsf{A}, \mathsf{B},\mathsf{D}\}$.  Let
$$\Lam{n}{k}{\mathsf{X}}: \LPT{n}{k}{X} \rightarrow \PT{n+1}{k}$$ be the map
that forgets the labels. For types $\mathsf{X}\in\{\mathsf{A},\mathsf{B}\}$,
these are surjections; the map $\Lam{n}{k}{\mathsf{D}}$, however, need not be
surjective. Observe that, for $\tau'\in\PT{n+1}{k}$,
\begin{align*} 
    \left|\Lam{n}{k}{A}^{-1}(\tau')\right| &= (n+1)!, &
    \left|\Lam{n}{k}{B}^{-1}(\tau')\right| &= 2^n(n+1)!.
\end{align*} 
Additionally, let
$$\Ome{n}{k}{\mathsf{X}} : \LPT{n}{k}{X} \twoheadrightarrow
\LRT{n}{k}{X}$$ be the surjection that forgets the ordering. 

The isomorphism $\rhorep{A}{n} : \Pitwo{A}{n} \rightarrow
\mathcal{L}(\mathsf{A}_n)$ in \eqref{eqn:type-A-iso} induces the
isomorphism
$$\Ph{n}{k}{A} : \LRT{n}{k}{A} \rightarrow
\Delta_{k-1}(\proplat(\mathsf{A}_n)).$$
As explained in Section~\ref{subsubsec:labels}, every generation of a
tree $\tau\in\LRT{n}{k}{B}$ determines a partition in $\Pitwo{B}{n}$.  In this way $\tau$ determines a chain
of length $k-1$ in~$\proplat(\mathsf{B}_n)$, and every such chain
arises in this way.
Let $$\Ph{n}{k}{B} : \LRT{n}{k}{B} \twoheadrightarrow
\Delta_{k-1}(\proplat(\mathsf{B}_n))$$ denote this surjection, and
define
$$\Ph{n}{k}{D} : \LRT{n}{k}{D} \twoheadrightarrow
\Delta_{k-1}(\proplat(\mathsf{D}_n))$$ as the restriction of~$\Ph{n}{k}{B}$.
Let $$\beta_{n, k} : \LRT{n}{k}{B}\rightarrow
\LRT{n}{k}{A}$$ be the function that forgets the bars and replaces the
label $0$ with the label~$n+1$. (There are similarly defined maps on
the order complex and sets of labeled plane trees, for which we will
have no use.) Since $\Pitwo{D}{n} \subset \Pitwo{B}{n}$, there are
natural embeddings defined on the order complex and labeled plane
resp.\ rooted trees. Figure~\ref{fig:maps} summarizes these maps.

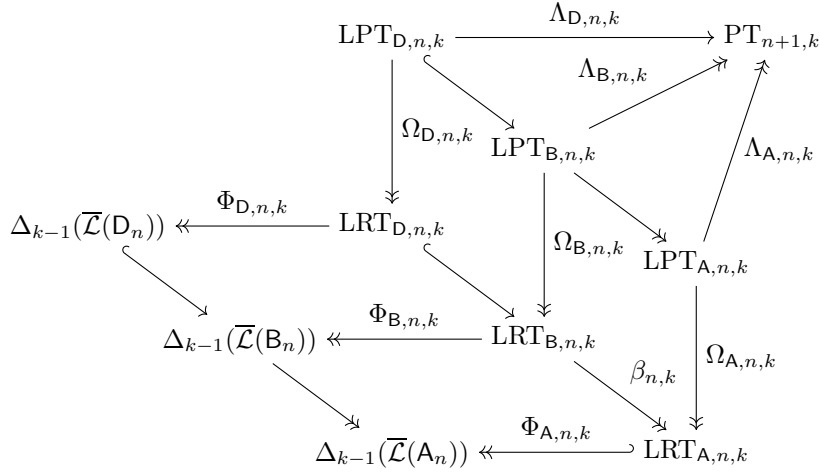
\begin{figure}[h]
    \centering
    \begin{tikzpicture}
        \node(LRTA) at (0,0) {$\LRT{n}{k}{A}$};
        \node(LRTB) at (-2,1.5) {$\LRT{n}{k}{B}$};
        \node(LRTD) at (-4,3) {$\LRT{n}{k}{D}$};
        \node(DelA) at (-4,0) {$\Delta_{k-1}(\proplat(\mathsf{A}_n))$};
        \node(DelB) at (-6,1.5) {$\Delta_{k-1}(\proplat(\mathsf{B}_n))$};
        \node(DelD) at (-8,3) {$\Delta_{k-1}(\proplat(\mathsf{D}_n))$};
        \node(LPTA) at (0,2.5) {$\LPT{n}{k}{A}$};
        \node(LPTB) at (-2,4) {$\LPT{n}{k}{B}$};
        \node(LPTD) at (-4,5.5) {$\LPT{n}{k}{D}$};
        \node(PT) at (1,5.5) {$\PT{n+1}{k}$};
        \draw[->>] (LPTA) -- node[anchor=west] {$\Lam{n}{k}{A}$} (PT);
        \draw[->>] (LPTB) -- node[anchor=south east] {$\Lam{n}{k}{B}$} (PT);
        \draw[->] (LPTD) -- node[anchor=south] {$\Lam{n}{k}{D}$} (PT);
        \draw[->>] (LPTA) -- node[anchor=west] {$\Ome{n}{k}{A}$} (LRTA);
        \draw[->>] (LPTB) -- node[anchor=west] {$\Ome{n}{k}{B}$} (LRTB);
        \draw[->>] (LPTD) -- node[anchor=west] {$\Ome{n}{k}{D}$} (LRTD);
        \draw[left hook->>] (LRTA) -- node[anchor=south] {$\Ph{n}{k}{A}$}  (DelA);
        \draw[->>] (LRTB) -- node[anchor=south] {$\Ph{n}{k}{B}$} (DelB);
        \draw[->>] (LRTD) -- node[anchor=south] {$\Ph{n}{k}{D}$} (DelD);
        \draw[->>] (LPTB) -- node[anchor=south west] {} (LPTA);
        \draw[->>] (LRTB) -- node[anchor=south west] {$\beta_{n, k}$} (LRTA);
        \draw[->>] (DelB) -- node[anchor=south west] {} (DelA);
        \draw[right hook->] (LPTD) -- node[anchor=south west] {} (LPTB);
        \draw[right hook->] (LRTD) -- node[anchor=south west] {} (LRTB);
        \draw[right hook->] (DelD) -- node[anchor=south west] {} (DelB);
    \end{tikzpicture}
    \caption{A commutative diagram showing all of the maps defined above with $1\leq k\leq n$.}
    \label{fig:maps}
\end{figure}

Recall Definition~\ref{defn:standard-form} for the standard form of
$\tau\in\LRT{n}{k}{X}\subset \LRTtwo{X}{n}$.

\begin{lem}\label{lem:Phi-cardinality}
  Let $\mathsf{X}\in \{\mathsf{A},\mathsf{B},\mathsf{D}\}$ and $1\leq k \leq n$.
  For all $F\in \Delta_{k-1}(\proplat(\mathsf{X}_n))$, the fiber
  $\Ph{n}{k}{X}^{-1}(F)$ contains a unique tree in standard form and for all
  $\tau\in \Ph{n}{k}{X}^{-1}(F)$,
  \begin{align*}
    \left|\Ph{n}{k}{\mathsf{X}}^{-1}(F)\right| &= \prod_{u\in V(\tau,0)} 2^{\kind(\tau,u)-1}. 
  \end{align*} 
\end{lem}

\begin{proof}
  Every tree in $\LRT{n}{k}{A}$ is standard, and $\Ph{n}{k}{A}$ is an
  isomorphism. It suffices, then, to prove this for type $\mathsf{B}$
  since $\Ph{n}{k}{D}$ is the restriction of $\Ph{n}{k}{B}$. If $k=1$,
  this is immediate, so we suppose $k\geq 2$ and $F = (x_1 < \cdots <
  x_{k-1})$. For all $i\in [k-1]$, let $P_i = \rhorep{B}{n}(x_i) \in
  \Pitwo{B}{n}$, where $\rhorep{B}{n}$ is as in
  Section~\ref{subsec:braid.total}, and set~$P_{k}=\hat{1}$.
  
  Let $\tau\in\Ph{n}{k}{B}^{-1}(F)$ and $u\in V(\tau, 0)$, where $u$ is in the
  $i$th generation of $\tau$, for $i\in [k-1]_0$. The vertices in the $i$th
  generation of $\tau$ are in one-to-one correspondence with the blocks of
  $P_{k-i}$.  In particular, $\lambda(\tau,u)=P_{k-i,0}$, the zero block of
  $P_{k-i}$. For each child $u'$ of $u$ not contained in $V(\tau,0)$, let
  $\alpha'\in [n]_0$ be the minimal integer label in $\mathrm{DLL}(\tau, u')$;
  cf.~Definition~\ref{defn:DLL}. It follows that $\alpha'$ is not barred in
  any $P_j$ since $\alpha'$ is always either the minimal integer or contained
  in the zero block. Thus, if $\tau$ is labeled such that $\alpha'$ is barred,
  then there exists $\tau'\in \LRT{n}{k}{B}$ such that $\alpha'$ is not barred
  and $\Ph{n}{k}{\mathsf{B}}(\tau') = \Ph{n}{k}{\mathsf{B}}(\tau)$. This is
  the only way two trees in $\LRT{n}{k}{B}$ may have the same image. Thus
  there exists $\tau_{\mathrm{st}}\in\Ph{n}{k}{B}^{-1}(F)$ such that, for all
  $u\in V(\tau_{\mathrm{st}}, 0)$ and for all children $u'$ of $u$ not in
  $V(\tau_{\mathrm{st}}, 0)$, the minimal label $\alpha'$ of
  $\mathrm{DLL}(\tau_{\mathrm{st}}, u')$ is not barred; by
  Definition~\ref{defn:standard-form}, the tree $\tau_{\mathrm{st}}$ is in
  standard form, and the lemma follows.
\end{proof}

Recall the definition~\eqref{eqn:pi.X.tau} of $\pitwo{X}{\tau}(Y)$ for
$\tau\in \LRTtwo{\mathsf{X}}{n}$ in
Section~\ref{subsec:atom.root.trees}.

\begin{lem}\label{lem:general-Poincare}
    Let $\mathsf{X}\in \{\mathsf{A},\mathsf{B},\mathsf{D}\}$ and $1\leq k \leq
    n$. If $F\in\Delta_{k-1}(\proplat(\mathsf{X}_n))$ then, for all
    $\tau\in\Ph{n}{k}{\mathsf{X}}^{-1}(F)$,
    $$\pi_F(Y) = \pitwo{X}{\tau}(Y).$$
\end{lem}

\begin{proof}
  First we show that, for $\tau,\tau'\in \Ph{n}{k}{X}^{-1}(F)$,
    \begin{align}\label{eqn:same-fiber}
        \pitwo{X}{\tau}(Y) &= \pitwo{X}{\tau'}(Y).
    \end{align} 
    It suffices to prove~\eqref{eqn:same-fiber} for $\mathsf{X}=\mathsf{B}$.
    From the proof of Lemma~\ref{lem:Phi-cardinality} we
    obtain~$\beta_{n,k}(\tau) = \beta_{n,k}(\tau')$. Since
    $\pitwo{X}{\tau}(Y)$ depends only on the underlying (unlabeled) rooted
    tree and the leaf label $0$, equation~\eqref{eqn:same-fiber} holds.

    Each generation in $\tau$ determines a set partition in $\Pitwo{X}{n}$; see
    Section~\ref{subsubsec:labels} for details. For $\ell\in [k-1]_0$, let
    $P,Q\in\Pitwo{X}{n}$ be the set partitions associated with generations
    $\ell$ and $\ell+1$ respectively. Via $\rhorep{X}{n}$ defined
    in~\eqref{eqn:type-A-iso} or in~\eqref{eqn:type-B-iso}, let
    $y=\rhorep{X}{n}(P)$ and $x=\rhorep{X}{n}(Q)$. The arrangement
    $(\mathsf{X}_n)_y$ is a product of classical Coxeter arrangements; the
    number of (possibly trivial) factors is equal the number of vertices in
    generation $\ell$ or, equivalently, $|P|$. Thus, $(\mathsf{X}_n)_y^x$ is a
    product of $|P|$ possibly trivial arrangements by
    Lemmas~\ref{lem:AB-Coxeter-res} and~\ref{lem:D-Coxeter-res}. A vertex $u$ in
    generation $\ell$ corresponds to a trivial factor of $(\mathsf{X}_n)_y^x$ if
    and only if~$\kind(\tau, u)=1$.
    
    Let $\Gamma$ be the induced subgraph of $\tau$ comprising vertices
    from generations $\ell$ and $\ell+1$, so $\Gamma$ is a forest of
    rooted trees of length one. Up to a (re-)labeling of the vertices from
    generation $\ell+1$ (i.e.~leaves of $\Gamma$) similar to the
    relabeling of polynomials in the proofs of
    Lemmas~\ref{lem:AB-Coxeter-res} and~\ref{lem:D-Coxeter-res}, it
    follows that $\Gamma$ is a disjoint union of total partitions and
    rooted trees with exactly two vertices, the latter corresponding
    to the trivial factors. Since Poincar\'e polynomials are
    multiplicative over factors~\cite[Lemma~2.50]{OrlikTerao/92},
    $\pi_{(\mathsf{X}_n)_y^x}(Y)$ is the product of the Poincar\'e
    polynomials of the disjoint total partitions of $\Gamma$. As a
    consequence of Theorem~\ref{thm:total-partition}, the Poincar\'e
    polynomial of $\tau'\in\TPnew{X}{n}$ is~$\pitwo{X}{\tau'}(Y)$,
    which proves the lemma.
\end{proof}

Recall the counting functions $\kind$ and $\kb$ defined in
\eqref{def:count.c} and \eqref{def:count.u}, respectively.

\begin{lem}\label{lem:fiber-sizes} 
    Let $\mathsf{X}\in \{\mathsf{A},\mathsf{B},\mathsf{D}\}$ and
    $1\leq k \leq n$; if $\mathsf{X}=\mathsf{D}$, then we also assume
    that~$n\geq 2$. If $F\in\Delta_{k-1}(\proplat(\mathsf{X}_n))$ and
    $\tau\in \Ph{n}{k}{\mathsf{X}}^{-1}(F)$, then
    \begin{align*} 
        \pi_F(1) &= \left|\Ph{n}{k}{\mathsf{X}}^{-1}(F)\right| \left|\Ome{n}{k}{\mathsf{X}}^{-1}(\tau)\right| C_{\mathsf{X}}(\tau),
    \end{align*}
    where $C_{\mathsf{A}}(\tau)=1$ and 
    \begin{align*} 
        C_{\mathsf{X}}(\tau) &= \begin{cases}
   \phantom{\frac{2\kind(\tau, v_0^+(\tau)) - \kb(\tau, v_0^+(\tau)) - 1}{2(\kind(\tau, v_0^+(\tau))-1)}}         \prod_{u\in V(\tau,0)} \kind(\tau, u)^{-1} & \textup{ if }\mathsf{X} = \mathsf{B}, \\
            \frac{2\kind(\tau, v_0^+(\tau)) - \kb(\tau, v_0^+(\tau)) - 1}{2(\kind(\tau, v_0^+(\tau))-1)}\prod_{u\in V(\tau,0)} \kind(\tau, u)^{-1} & \textup{ if }\mathsf{X} = \mathsf{D}. 
        \end{cases}
    \end{align*} 
\end{lem}

\begin{proof}
  By Lemma~\ref{lem:general-Poincare} it suffices to compute
  $\pitwo{X}{\tau}(1)$. We obtain these numbers by setting $Y=1$
  in~\eqref{eqn:pi.X.tau}, noting that
  $$\pb[1]{n}!=(n+1)!,\quad\textup{ whereas } \quad\pb[1]{n}!! = 2^n
  n!=:(2n)!!,$$ whence
    \begin{align*} 
        \pitwo{A}{\tau}(1) &= \prod_{u\in P(\tau)} \kind(\tau, u)!, \\
        \pitwo{B}{\tau}(1) &= \left(\prod_{u\in V(\tau,0)} (2\kind(\tau, u)-2)!!\right)\left( \prod_{u\in P(\tau)\setminus V(\tau,0)} \kind(\tau, u)!\right), \\
        \pitwo{D}{\tau}(1) &= \dfrac{\left(2\kind(\tau,v_0^+(\tau))-\kb(\tau,v_0^+(\tau))-1\right)}{2\left(\kind(\tau,v_0^+(\tau))-1\right)} \pitwo{B}{\tau}(1).
    \end{align*}
    For all types $\mathsf{X}$, we have
    \begin{align*} 
        \left|\Ome{n}{k}{\mathsf{X}}^{-1}(\tau)\right| &= \prod_{u\in V(\tau)} \kind(\tau, u)!,
    \end{align*} 
    and by Lemma~\ref{lem:Phi-cardinality},
    \begin{align*}
        \left|\Ph{n}{k}{\mathsf{X}}^{-1}(F)\right| &= \prod_{u\in V(\tau,0)} 2^{\kind(\tau,u)-1}. \qedhere
    \end{align*} 
\end{proof}

Before embarking on a proof of Theorem~\ref{thm:f-vector}, we need two more
lemmas to round off the computations. Both are variations of the same counting
arguments concerning the sums of probabilities modeled by decision trees.

It will be useful to discard most of the labels and work only with the set of
``$0$-labeled plane trees,'' viz.\ trees with one leaf labeled by $0$ and all
other leaves unlabeled. All of our counting functions are unaffected
by this simplification.

\begin{defn}\label{def:LPTzero}
  Let $\mathsf{X}\in\{\mathsf{B},\mathsf{D}\}$ and $1\leq k \leq n$.
  \begin{enumerate}
    \item Let $\LPTzero{n}{k}{X}$ be the set obtained from
      $\LPT{n}{k}{X}$ by removing all labels except for the label $0$. 
    \item For $\tau\in \LPTzero{n}{k}{X}$, let $\mathrm{Prob}(\tau) := \prod_{u\in V(\tau, 0)} \kind(\tau, u)^{-1}$.
\end{enumerate}
\end{defn}

\begin{lem}\label{lem:scotch-mist}
  Let $1 \leq k\leq n$. Then
    \begin{align*} 
        \sum_{\tau\in \LPT{n}{k}{B}} {\prod_{u\in V(\tau,0)}\kind(\tau, u)^{-1}} &= \dfrac{\left|\LPT{n}{k}{B}\right|}{n+1}.
    \end{align*}
\end{lem}

\begin{proof}
  It suffices to prove
    \begin{align} \label{eqn:simple-sm} 
        \sum_{\tau\in \LPTzero{n}{k}{B}}
        \mathrm{Prob}(\tau) &= \dfrac{\left|\LPTzero{n}{k}{B}\right|}{n+1} =
        \left|\PT{n+1}{k}\right|.
    \end{align}
    Let $\tau\in \PT{n+1}{k}$, and fix any ordering of its leaves. For $i\in
    [n+1]$, let $\tau_i \in \LPTzero{n}{k}{B}$ be the plane tree obtained from
    labeling the $i$th leaf of $\tau$ by~$0$.  Viewing $\tau$ as a decision
    tree, where every parent vertex assigns uniform probability to its children,
    we see that $\sum_{i=1}^{n+1} \mathrm{Prob}(\tau_i) = 1$. This
    implies~\eqref{eqn:simple-sm}, which proves the lemma.
\end{proof} 

\subsubsection{Grafting plane trees}
Lemma~\ref{lem:grafting-lem} below is a type-$\mathsf{D}$ analog of
Lemma~\ref{lem:scotch-mist}. In order to prove it, we define a relation
$\mathcal{G}_{n,k}$ (cf.\ Definition~\ref{def:graft}) on the set
$$\Trel = \LPTzero{n}{k}{B}\times \LPTzero{n}{k}{D}.$$ This subset may
be thought of as the set of pairs $(\sigma,\tau)\in\Trel$ with the property
that $\tau$ is obtained from $\sigma$ via an operation we call
\emph{grafting}, which we now explain.

Assume that $n\geq 2$, and let $\sigma\in \LPTzero{n}{k}{B}$. We call the
subtree induced by a vertex $u\in V(\sigma)$ and all of its
descendants the \emph{$u$-branch} of $\sigma$. A \emph{branch} of
$\sigma$ is a $u$-branch for some $u\in V(\sigma)$. We write
$v_0^+:=v_0^+(\sigma)$ for the first ancestor of $v_0$ with at least two
children. Suppose that $v_0^+$ is in generation~$g$ and write
$v_0^{++}$ for the first ancestor of $v_0^+$ with at least two
children; see Figure~\ref{fig:grafting} for an example.

Suppose that $\sigma\in \LPTzero{n}{k}{B}\setminus\LPTzero{n}{k}{D}$. In this
case it follows from the definitions that $v_0^+$ has exactly two children, both
of which are unbranched; see Definitions~\ref{def:X-label}~(\ref{def:D-label})
and~\ref{def:PT.LPT.LRT}. Let $u$ be the unique child of $v_0^+$ not contained
in $V(\sigma, 0)$, and let $B$ be the $u$-branch of~$\sigma$. Note that $B$ may
be ``on the right'' or ``on the left'' of $v_0^+$. We now remove the edge
connecting $v_0^+$ and $B$ and connect (or ``graft'') $B$ as a new branch to
exactly one of $\kind(\sigma,v_0^{++})-1$ judiciously chosen ``contemporaries''
of $v_0^+$ (viz.\ vertices of $\sigma$ of the same generation~$g$).  More
precisely, let $C$ be one of the $\kind(\sigma,v_0^{++})-1$ branches of $\sigma$
whose root is a child of $v_0^{++}$ and does not contain~$v_0^+$. We connect $B$
onto the rightmost (resp.\ leftmost) branch contained in $C$ at generation $g$
if $B$ is on the right (resp.\ left) of $v_0$. We observe that each of the
resulting $\kind(\sigma,v_0^{++})-1$ trees are elements of $\LPTzero{n}{k}{D}$
and call them the trees obtained from $\sigma$ via \emph{grafting}.
Figure~\ref{fig:grafting} illustrates an example with $B$ on the right of $v_0$.

If $\sigma\in\LPTzero{n}{k}{D}$, then, by definition, $\tau=\sigma$ is
the only tree obtained from $\sigma$ via grafting.

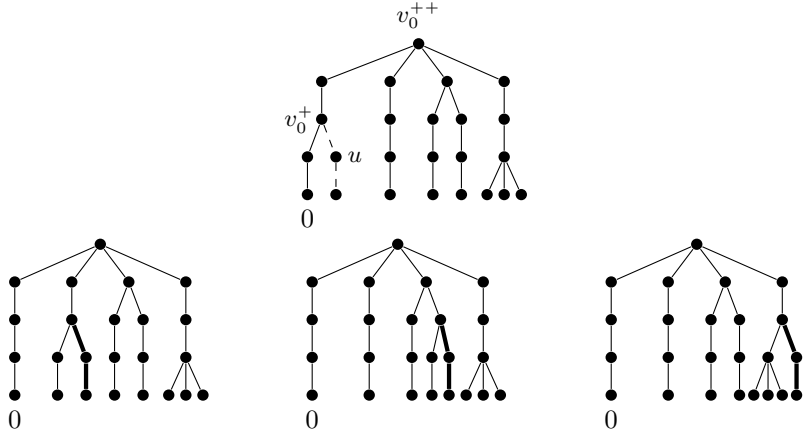
\begin{figure}[h]
    \centering 
    \begin{subfigure}{0.3\textwidth} 
        \centering 
        \begin{tikzpicture}
            \pgfmathsetmacro{\x}{0.75}
            \pgfmathsetmacro{\y}{0.5}
            \node(1) at (0,4*\y) [circle, fill=black, inner sep=1.5pt, label={above:{\small $v_0^{++}$}}] {};
            \node(2) at (-1.7*\x,3*\y) [circle, fill=black, inner sep=1.5pt] {};
            \node(3) at (-0.5*\x,3*\y) [circle, fill=black, inner sep=1.5pt] {};
            \node(4) at (0.5*\x,3*\y) [circle, fill=black, inner sep=1.5pt] {};
            \node(5) at (1.5*\x,3*\y) [circle, fill=black, inner sep=1.5pt] {};
            \node(6) at (-1.7*\x,2*\y) [circle, fill=black, inner sep=1.5pt, label={[label distance=-1mm]west:{\small $v_0^{+}$}}] {};
            \node(7) at (-0.5*\x,2*\y) [circle, fill=black, inner sep=1.5pt] {};
            \node(8) at (0.25*\x,2*\y) [circle, fill=black, inner sep=1.5pt] {};
            \node(9) at (0.75*\x,2*\y) [circle, fill=black, inner sep=1.5pt] {};
            \node(10) at (1.5*\x,2*\y) [circle, fill=black, inner sep=1.5pt] {};
            \node(11) at (-1.95*\x,\y) [circle, fill=black, inner sep=1.5pt] {};
            \node(12) at (-1.45*\x,\y) [circle, fill=black, inner sep=1.5pt, label={[label distance=-0.5mm]east:{\small $u$}}] {};
            \node(13) at (-0.5*\x,\y) [circle, fill=black, inner sep=1.5pt] {};
            \node(14) at (0.25*\x,\y) [circle, fill=black, inner sep=1.5pt] {};
            \node(15) at (0.75*\x,\y) [circle, fill=black, inner sep=1.5pt] {};
            \node(16) at (1.5*\x,\y) [circle, fill=black, inner sep=1.5pt] {};
            \node(17) at (-1.95*\x,0) [circle, fill=black, inner sep=1.5pt, label={below:$0$}] {};
            \node(18) at (-1.45*\x,0) [circle, fill=black, inner sep=1.5pt] {};
            \node(19) at (-0.5*\x,0) [circle, fill=black, inner sep=1.5pt] {};
            \node(20) at (0.25*\x,0) [circle, fill=black, inner sep=1.5pt] {};
            \node(21) at (0.75*\x,0) [circle, fill=black, inner sep=1.5pt] {};
            \node(22) at (1.2*\x,0) [circle, fill=black, inner sep=1.5pt] {};
            \node(23) at (1.5*\x,0) [circle, fill=black, inner sep=1.5pt] {};
            \node(24) at (1.8*\x,0) [circle, fill=black, inner sep=1.5pt] {};
            \draw (1) -- (2);
            \draw (1) -- (3);
            \draw (1) -- (4);
            \draw (1) -- (5);
            \draw (2) -- (6);
            \draw (3) -- (7);
            \draw (4) -- (8);
            \draw (4) -- (9);
            \draw (5) -- (10);
            \draw (6) -- (11);
            \draw[dashed] (6) -- (12);
            \draw (7) -- (13);
            \draw (8) -- (14);
            \draw (9) -- (15);
            \draw (10) -- (16);
            \draw (11) -- (17);
            \draw[dashed] (12) -- (18);
            \draw (13) -- (19);
            \draw (14) -- (20);
            \draw (15) -- (21);
            \draw (16) -- (22);
            \draw (16) -- (23);
            \draw (16) -- (24);
        \end{tikzpicture}
    \end{subfigure}

    \begin{subfigure}{0.3\textwidth} 
        \centering 
        \begin{tikzpicture}
            \pgfmathsetmacro{\x}{0.75}
            \pgfmathsetmacro{\y}{0.5}
            \node(1) at (0,4*\y) [circle, fill=black, inner sep=1.5pt] {};
            \node(2) at (-1.5*\x,3*\y) [circle, fill=black, inner sep=1.5pt] {};
            \node(3) at (-0.5*\x,3*\y) [circle, fill=black, inner sep=1.5pt] {};
            \node(4) at (0.5*\x,3*\y) [circle, fill=black, inner sep=1.5pt] {};
            \node(5) at (1.5*\x,3*\y) [circle, fill=black, inner sep=1.5pt] {};
            \node(6) at (-1.5*\x,2*\y) [circle, fill=black, inner sep=1.5pt] {};
            \node(7) at (-0.5*\x,2*\y) [circle, fill=black, inner sep=1.5pt] {};
            \node(8) at (0.25*\x,2*\y) [circle, fill=black, inner sep=1.5pt] {};
            \node(9) at (0.75*\x,2*\y) [circle, fill=black, inner sep=1.5pt] {};
            \node(10) at (1.5*\x,2*\y) [circle, fill=black, inner sep=1.5pt] {};
            \node(11) at (-1.5*\x,\y) [circle, fill=black, inner sep=1.5pt] {};
            \node(12) at (-0.25*\x,\y) [circle, fill=black, inner sep=1.5pt] {};
            \node(13) at (-0.75*\x,\y) [circle, fill=black, inner sep=1.5pt] {};
            \node(14) at (0.25*\x,\y) [circle, fill=black, inner sep=1.5pt] {};
            \node(15) at (0.75*\x,\y) [circle, fill=black, inner sep=1.5pt] {};
            \node(16) at (1.5*\x,\y) [circle, fill=black, inner sep=1.5pt] {};
            \node(17) at (-1.5*\x,0) [circle, fill=black, inner sep=1.5pt, label={below:$0$}] {};
            \node(18) at (-0.25*\x,0) [circle, fill=black, inner sep=1.5pt] {};
            \node(19) at (-0.75*\x,0) [circle, fill=black, inner sep=1.5pt] {};
            \node(20) at (0.25*\x,0) [circle, fill=black, inner sep=1.5pt] {};
            \node(21) at (0.75*\x,0) [circle, fill=black, inner sep=1.5pt] {};
            \node(22) at (1.2*\x,0) [circle, fill=black, inner sep=1.5pt] {};
            \node(23) at (1.5*\x,0) [circle, fill=black, inner sep=1.5pt] {};
            \node(24) at (1.8*\x,0) [circle, fill=black, inner sep=1.5pt] {};
            \draw (1) -- (2);
            \draw (1) -- (3);
            \draw (1) -- (4);
            \draw (1) -- (5);
            \draw (2) -- (6);
            \draw (3) -- (7);
            \draw (4) -- (8);
            \draw (4) -- (9);
            \draw (5) -- (10);
            \draw (6) -- (11);
            \draw[ultra thick] (7) -- (12);
            \draw (7) -- (13);
            \draw (8) -- (14);
            \draw (9) -- (15);
            \draw (10) -- (16);
            \draw (11) -- (17);
            \draw[ultra thick] (12) -- (18);
            \draw (13) -- (19);
            \draw (14) -- (20);
            \draw (15) -- (21);
            \draw (16) -- (22);
            \draw (16) -- (23);
            \draw (16) -- (24);
        \end{tikzpicture}
    \end{subfigure}~%
    \begin{subfigure}{0.3\textwidth} 
        \centering 
        \begin{tikzpicture}
            \pgfmathsetmacro{\x}{0.75}
            \pgfmathsetmacro{\y}{0.5}
            \node(1) at (0,4*\y) [circle, fill=black, inner sep=1.5pt] {};
            \node(2) at (-1.5*\x,3*\y) [circle, fill=black, inner sep=1.5pt] {};
            \node(3) at (-0.5*\x,3*\y) [circle, fill=black, inner sep=1.5pt] {};
            \node(4) at (0.5*\x,3*\y) [circle, fill=black, inner sep=1.5pt] {};
            \node(5) at (1.5*\x,3*\y) [circle, fill=black, inner sep=1.5pt] {};
            \node(6) at (-1.5*\x,2*\y) [circle, fill=black, inner sep=1.5pt] {};
            \node(7) at (-0.5*\x,2*\y) [circle, fill=black, inner sep=1.5pt] {};
            \node(8) at (0.25*\x,2*\y) [circle, fill=black, inner sep=1.5pt] {};
            \node(9) at (0.75*\x,2*\y) [circle, fill=black, inner sep=1.5pt] {};
            \node(10) at (1.5*\x,2*\y) [circle, fill=black, inner sep=1.5pt] {};
            \node(11) at (-1.5*\x,\y) [circle, fill=black, inner sep=1.5pt] {};
            \node(12) at (0.9*\x,\y) [circle, fill=black, inner sep=1.5pt] {};
            \node(13) at (-0.5*\x,\y) [circle, fill=black, inner sep=1.5pt] {};
            \node(14) at (0.25*\x,\y) [circle, fill=black, inner sep=1.5pt] {};
            \node(15) at (0.6*\x,\y) [circle, fill=black, inner sep=1.5pt] {};
            \node(16) at (1.5*\x,\y) [circle, fill=black, inner sep=1.5pt] {};
            \node(17) at (-1.5*\x,0) [circle, fill=black, inner sep=1.5pt, label={below:$0$}] {};
            \node(18) at (0.9*\x,0) [circle, fill=black, inner sep=1.5pt] {};
            \node(19) at (-0.5*\x,0) [circle, fill=black, inner sep=1.5pt] {};
            \node(20) at (0.25*\x,0) [circle, fill=black, inner sep=1.5pt] {};
            \node(21) at (0.6*\x,0) [circle, fill=black, inner sep=1.5pt] {};
            \node(22) at (1.2*\x,0) [circle, fill=black, inner sep=1.5pt] {};
            \node(23) at (1.5*\x,0) [circle, fill=black, inner sep=1.5pt] {};
            \node(24) at (1.8*\x,0) [circle, fill=black, inner sep=1.5pt] {};
            \draw (1) -- (2);
            \draw (1) -- (3);
            \draw (1) -- (4);
            \draw (1) -- (5);
            \draw (2) -- (6);
            \draw (3) -- (7);
            \draw (4) -- (8);
            \draw (4) -- (9);
            \draw (5) -- (10);
            \draw (6) -- (11);
            \draw[ultra thick] (9) -- (12);
            \draw (7) -- (13);
            \draw (8) -- (14);
            \draw (9) -- (15);
            \draw (10) -- (16);
            \draw (11) -- (17);
            \draw[ultra thick] (12) -- (18);
            \draw (13) -- (19);
            \draw (14) -- (20);
            \draw (15) -- (21);
            \draw (16) -- (22);
            \draw (16) -- (23);
            \draw (16) -- (24);
        \end{tikzpicture}
    \end{subfigure}~%
    \begin{subfigure}{0.3\textwidth} 
        \centering 
        \begin{tikzpicture}
            \pgfmathsetmacro{\x}{0.75}
            \pgfmathsetmacro{\y}{0.5}
            \node(1) at (0,4*\y) [circle, fill=black, inner sep=1.5pt] {};
            \node(2) at (-1.5*\x,3*\y) [circle, fill=black, inner sep=1.5pt] {};
            \node(3) at (-0.5*\x,3*\y) [circle, fill=black, inner sep=1.5pt] {};
            \node(4) at (0.5*\x,3*\y) [circle, fill=black, inner sep=1.5pt] {};
            \node(5) at (1.5*\x,3*\y) [circle, fill=black, inner sep=1.5pt] {};
            \node(6) at (-1.5*\x,2*\y) [circle, fill=black, inner sep=1.5pt] {};
            \node(7) at (-0.5*\x,2*\y) [circle, fill=black, inner sep=1.5pt] {};
            \node(8) at (0.25*\x,2*\y) [circle, fill=black, inner sep=1.5pt] {};
            \node(9) at (0.75*\x,2*\y) [circle, fill=black, inner sep=1.5pt] {};
            \node(10) at (1.5*\x,2*\y) [circle, fill=black, inner sep=1.5pt] {};
            \node(11) at (-1.5*\x,\y) [circle, fill=black, inner sep=1.5pt] {};
            \node(12) at (1.75*\x,\y) [circle, fill=black, inner sep=1.5pt] {};
            \node(13) at (-0.5*\x,\y) [circle, fill=black, inner sep=1.5pt] {};
            \node(14) at (0.25*\x,\y) [circle, fill=black, inner sep=1.5pt] {};
            \node(15) at (0.75*\x,\y) [circle, fill=black, inner sep=1.5pt] {};
            \node(16) at (1.25*\x,\y) [circle, fill=black, inner sep=1.5pt] {};
            \node(17) at (-1.5*\x,0) [circle, fill=black, inner sep=1.5pt, label={below:$0$}] {};
            \node(18) at (1.75*\x,0) [circle, fill=black, inner sep=1.5pt] {};
            \node(19) at (-0.5*\x,0) [circle, fill=black, inner sep=1.5pt] {};
            \node(20) at (0.25*\x,0) [circle, fill=black, inner sep=1.5pt] {};
            \node(21) at (0.75*\x,0) [circle, fill=black, inner sep=1.5pt] {};
            \node(22) at (1.0*\x,0) [circle, fill=black, inner sep=1.5pt] {};
            \node(23) at (1.25*\x,0) [circle, fill=black, inner sep=1.5pt] {};
            \node(24) at (1.5*\x,0) [circle, fill=black, inner sep=1.5pt] {};
            \draw (1) -- (2);
            \draw (1) -- (3);
            \draw (1) -- (4);
            \draw (1) -- (5);
            \draw (2) -- (6);
            \draw (3) -- (7);
            \draw (4) -- (8);
            \draw (4) -- (9);
            \draw (5) -- (10);
            \draw (6) -- (11);
            \draw[ultra thick] (10) -- (12);
            \draw (7) -- (13);
            \draw (8) -- (14);
            \draw (9) -- (15);
            \draw (10) -- (16);
            \draw (11) -- (17);
            \draw[ultra thick] (12) -- (18);
            \draw (13) -- (19);
            \draw (14) -- (20);
            \draw (15) -- (21);
            \draw (16) -- (22);
            \draw (16) -- (23);
            \draw (16) -- (24);
        \end{tikzpicture}
    \end{subfigure}
    \caption{Via grafting, the top tree $\sigma \in \LPTzero{7}{4}{B}
      \setminus \LPTzero{7}{4}{B}$ gives rise to the
      $\kind(\sigma,v_0^{++})-1=3$ bottom trees $\tau$ with
      $(\sigma,\tau)\in\Graft{7}{4}$. One removes the edge connecting $v_0^+$
      and $u$ and grafts the $u$-branch onto each of the three other
      possible branches. We thickened the grafted branch to
      distinguish it from the others.}
    \label{fig:grafting}
\end{figure}

\begin{defn}\label{def:graft}
  For $1\leq k \leq n$ and $2 \leq n$, we set
  $$\Graft{n}{k} = \left\{(\sigma,\tau)\in \Trel \mid \tau \textup{ may be obtained from $\sigma$
    by grafting} \right\}$$ and write $\sigma \sim \tau$ if
  $(\sigma,\tau)\in\Graft{n}{k}$.
\end{defn}

\begin{lem}\label{lem:grafting-lem}
    Let  $1\leq k\leq n$ and $2\leq n$, then 
    \begin{equation*} 
        \sum_{\tau\in\LPT{n}{k}{D}} \dfrac{2\kind(\tau, v_0^+(\tau)) - \kb(\tau, v_0^+(\tau)) - 1}{2(\kind(\tau, v_0^+(\tau)) - 1)}\prod_{u\in V(\tau, 0)} \kind(\tau,u)^{-1}  
        = 2^{n-1}n!k! S(n,k).
    \end{equation*}
\end{lem}

\begin{proof}
  It suffices to show that
    \begin{align}\label{ets} 
        \sum_{\tau\in\LPTzero{n}{k}{D}} \left(1 + \dfrac{\kind(\tau, v_0^+(\tau)) - \kb(\tau, v_0^+(\tau))}{\kind(\tau, v_0^+(\tau)) - 1}\right)\mathrm{Prob}(\tau) &= k! S(n,k).
    \end{align}

    Define $f: \Trel\rightarrow \mathbb{Q}$ by setting
    \begin{align*} 
      f(\sigma, \tau) &= \begin{cases}
        \mathrm{Prob}(\sigma) & \textup{ if }\sigma = \tau, \\
        \frac{\mathrm{Prob}(\sigma)}{\kind(\sigma, v_0^{++}(\sigma)) - 1} & \textup{ if }\sigma \sim \tau, \;\sigma \neq \tau, \\
        0 & \textup{ if } \sigma \not\sim \tau.
        \end{cases}
    \end{align*}
    We already observed that, for each
    $\sigma\in\LPTzero{n}{k}{B}\setminus\LPTzero{n}{k}{D}$, there are
    $\kind(\sigma, v_0^{++}(\sigma)) - 1$ trees
    $\tau\in\LPTzero{n}{k}{D}$ with
    $(\sigma,\tau)\in\Graft{n}{k}$. Thus Lemma~\ref{lem:plane-trees}
    and~\eqref{eqn:simple-sm} imply that
    \begin{equation} \label{eqn:LHS}
      \begin{split}
        \sum_{(\sigma,\tau)\in \Trel} f(\sigma,\tau) &= \sum_{\substack{(\sigma,\tau)\in \Trel \\ \sigma\in\LPTzero{n}{k}{D}}} f(\sigma,\tau) + \sum_{\substack{(\sigma,\tau)\in \Trel \\ \sigma\not\in\LPTzero{n}{k}{D}}} f(\sigma,\tau) \\
        &= \sum_{\sigma\in \LPTzero{n}{k}{B}} \mathrm{Prob}(\sigma) = k!S(n,k). 
      \end{split}
    \end{equation}

    For $\tau\in \LPTzero{n}{k}{D}$, set
    $$ 
    \mathsf{g}(\tau) := \kind(\tau, v_0^+(\tau)) - \kb(\tau,
    v_0^+(\tau)) \in\N_0.
    $$ 
    There are $2\mathsf{g}(\tau)$ different trees $\sigma \in
    \LPTzero{n}{k}{B}\setminus \LPTzero{n}{k}{D}$ such that $(\sigma,\tau)\in
    \Graft{n}{k}$. To see this, we reverse the grafting operation: there are
    $\mathsf{g}(\tau)$ branches that can be cut and grafted onto the branch
    containing $v_0$, on either the left or right side. Therefore, if
    $\mathsf{g}(\tau)=0$, then $\sigma\sim\tau$ implies $\sigma=\tau$. Hence,
    \begin{align} \label{eqn:g-zero}
        \sum_{\substack{(\sigma,\tau)\in \Trel \\ \mathsf{g}(\tau)=0}} f(\sigma,\tau) &= \sum_{\substack{\tau\in \LPTzero{n}{k}{D} \\ \mathsf{g}(\tau)=0}} \mathrm{Prob}(\tau).
    \end{align} 
    If $\mathsf{g}(\tau)>0$, then $\sigma\sim \tau$ and $\sigma\neq
    \tau$ imply both $v_0^{++}(\sigma) = v_0^+(\tau)$ and
    $2\mathrm{Prob}(\sigma) = \mathrm{Prob}(\tau)$. Therefore,
    \begin{equation} \label{eqn:g-positive}
      \begin{split}
      \sum_{\substack{(\sigma,\tau)\in T \\ \mathsf{g}(\tau)>0}} f(\sigma,\tau) 
      &= \sum_{\substack{(\sigma,\tau)\in T \\ \mathsf{g}(\tau)>0 \\ \sigma=\tau}} f(\sigma,\tau) + \sum_{\substack{(\sigma,\tau)\in T \\ \mathsf{g}(\tau)>0 \\ \sigma\neq\tau}} f(\sigma,\tau) \\
      &= \sum_{\substack{\tau\in \LPTzero{n}{k}{D} \\ \mathsf{g}(\tau)>0}} \mathrm{Prob}(\tau) + \sum_{\substack{\tau\in \LPTzero{n}{k}{D}}} \dfrac{\mathsf{g}(\tau)\mathrm{Prob}(\tau)}{\kind(\tau, v_0^+(\tau)) - 1} .
      \end{split}
    \end{equation} 
    
    We deduce~\eqref{ets} (and thus the lemma) by
    combining~\eqref{eqn:LHS},~\eqref{eqn:g-zero},
    and~\eqref{eqn:g-positive}:
    \begin{align*} 
        k!S(n,k) = \sum_{(\sigma,\tau)\in T} f(\sigma,\tau) 
        &= \sum_{\tau\in\LPTzero{n}{k}{D}} \left(1 + \dfrac{\kind(\tau, v_0^+(\tau)) - \kb(\tau, v_0^+(\tau))}{\kind(\tau, v_0^+(\tau)) - 1}\right)\mathrm{Prob}(\tau). \qedhere
    \end{align*} 
\end{proof}

\subsubsection{Proof of Theorem~\ref{thm:f-vector}}\label{subsubsec:stirling.proof}
    By Lemma~\ref{lem:plane-trees}, 
    \begin{align*} 
        |\LPT{n}{k}{X}| &= \begin{cases} \phantom{2^n}(n+1)!
          k! S(n,k) &\textup{ if } \mathsf{X}= \mathsf{A},
          \\ 2^n (n+1)! k! S(n,k) &\textup{ if } \mathsf{X} = \mathsf{B}.
        \end{cases}
    \end{align*}
    For each $F\in\Delta_{k-1}(\proplat(\mathsf{X}_n))$, choose $\tau_F\in
    \Ph{n}{k}{X}^{-1}(F)$. By Lemma~\ref{lem:fiber-sizes},
    \begin{align*} 
        \sum_{F\in\Delta_{k-1}(\proplat(\mathsf{X}_n))} \pi_F(1) &= \sum_{F\in\Delta_{k-1}(\proplat(\mathsf{X}_n))} \left|\Ph{n}{k}{X}^{-1}(F)\right| \left|\Ome{n}{k}{X}^{-1}(\tau_F)\right| C_\mathsf{X}(\tau_F) \\
        &= \sum_{\tau\in\LPT{n}{k}{X}} C_\mathsf{X}(\Ome{n}{k}{X}(\tau)).
    \end{align*} 
    If $\mathsf{X}=\mathsf{A}$, then
    \begin{align*} 
        \sum_{\tau\in\LPT{n}{k}{A}} C_\mathsf{A}(\Ome{n}{k}{A}(\tau)) &= |\LPT{n}{k}{A}| = \pi_{\mathsf{A}_n}(1) k! S(n,k). 
    \end{align*}
    If $\mathsf{X}=\mathsf{B}$, then, by Lemma~\ref{lem:scotch-mist},
    \begin{align*} 
        \sum_{\tau\in\LPT{n}{k}{B}} C_\mathsf{B}(\Ome{n}{k}{B}(\tau)) &= \dfrac{|\LPT{n}{k}{B}|}{n+1} = \pi_{\mathsf{B}_n}(1) k! S(n,k). 
    \end{align*}
    Lastly, if $\mathsf{X}=\mathsf{D}$, then by
    Lemma~\ref{lem:grafting-lem},
    \begin{align*} 
        \sum_{\tau\in \LPT{n}{k}{D}} C_{\mathsf{D}}(\Ome{n}{k}{D}(\tau)) &= \pi_{\mathsf{D}_n}(1)k! S(n,k).
    \end{align*}
This completes the proof of Theorem~\ref{thm:f-vector}.\hfill$\square$

\subsection{Proof of Theorem~\ref{thm:coarse.Y=1}}\label{subsec:proof.coarse.Y=1}
Theorem~\ref{thm:f-vector} establishes Theorem~\ref{thm:coarse.Y=1} in
the case of classical Coxeter arrangements.
Formulae for the exceptional irreducible Coxeter arrangements not
equal to $\mathsf{E}_8$ are given in
Appendix~\ref{subsubsec:coarse.rank.leq.7}. The proof of
Theorem~\ref{thm:coarse.Y=1} for these arrangements follows by
inspection of these formulae.

The case of general Coxeter arrangements with no irreducible factors
isomorphic to~$\mathsf{E}_8$ follows now from
Proposition~\ref{prop:Hadamard}. Indeed using, for instance, the
\emph{Carlitz identity} (cf. \cite[Cor.~1.1]{Petersen/15})
\begin{align*} 
    \dfrac{E_n(T)}{(1-T)^{n+1}} &= \sum_{k\geq 0}(k+1)^nT^k
\end{align*}
for $n\geq 0$, one sees that for $n,m\geq 0$,
\begin{align*}
    \dfrac{E_n(T)}{(1-T)^{n+1}} \star_T \dfrac{E_m(T)}{(1-T)^{m+1}} &= \dfrac{E_{n+m}(T)}{(1-T)^{n+m+1}}.
\end{align*}
This completes the proof of Theorem~\ref{thm:coarse.Y=1}.\hfill$\square$

\begin{acknowledgements}
  We are grateful to our DFG-project partners Anne Fr\"uhbis-Kr\"uger and
  Bernd Schober for valuable conversations that helped to get us started. We
  also thank Nero Budur, Michael Cuntz, Tobias Rossmann, Mima Stanojkovski,
  Christian Stump and Wim Veys for helpful correspondence.
\end{acknowledgements}

\bibliography{HyperplaneArrangements}
\bibliographystyle{abbrv}

\begin{appendices}

\section{Further examples of coarse flag Hilbert--Poincar\'e series}
\label{sec:app.exa.coarse}

We collect explicit formulae for coarse flag Hilbert--Poincar\'e series of
some hyperplane arrangements, including the irreducible Coxeter arrangements
of rank at most seven in Section~\ref{subsec:irr.cox}. Most were computed
using our package \textsf{HypIgu}~\cite{hypigu} for
SageMath~\cite{sagemath}. Recall the notation~\eqref{def:numer.coarse} for the
numerator of the coarse flag Hilbert--Poincar\'e series of a hyperplane
arrangement~$\A$. Throughout, let
$\mathsf{X}\in\{\mathsf{A},\mathsf{B},\mathsf{D}\}$.

\subsection{Irreducible Coxeter arrangements of rank at most
  seven}\label{subsec:irr.cox}

\subsubsection{Type $\mathsf{A}$}\label{subsubsec:type-A}
\begin{align*} 
    \mathcal{N}_{\mathsf{A}_1}(Y, T) &= 1 + Y, \\
    \mathcal{N}_{\mathsf{A}_2}(Y, T) &= 1 + 3Y + 2Y^2 + (2 + 3Y + Y^2)T, \\
    \mathcal{N}_{\mathsf{A}_3}(Y, T) &= 1 + 6Y + 11Y^2 + 6Y^3 + (11 + 37Y + 37Y^2 + 11Y^3)T \\
    &\quad + (6 + 11Y + 6Y^2 + Y^3)T^2, \\
    \mathcal{N}_{\mathsf{A}_4}(Y, T) &= 1 + 10Y + 35Y^2 + 50Y^3 + 24Y^4 \\
    &\quad + (47 + 260Y + 505Y^2 + 400Y^3 + 108Y^4)T \\ 
    &\quad + (108 + 400Y + 505Y^2 + 260Y^3 + 47Y^4)T^2 \\
    &\quad + (24 + 50Y + 35Y^2 + 10Y^3 + Y^4)T^3, \\
    \mathcal{N}_{\mathsf{A}_5}(Y, T) &= 1 + 15Y + 85Y^2 + 225Y^3 + 274Y^4 + 120Y^5 \\
    &\quad + (197 + 1546Y + 4670Y^2 + 6700Y^3 + 4493Y^4 + 1114Y^5)T \\
    &\quad + (1268 + 7172Y + 15320Y^2 + 15320Y^3 + 7172Y^4 + 1268Y^5)T^2 \\
    &\quad + (1114 + 4493Y + 6700Y^2 + 4670Y^3 + 1546Y^4 + 197Y^5)T^3 \\
    &\quad + (120 + 274Y + 225Y^2 + 85Y^3 + 15Y^4 + Y^5)T^4, \\
    \mathcal{N}_{\mathsf{A}_6}(Y, T) &= 1 + 21Y + 175Y^2 + 735Y^3 + 1624Y^4 + 1764Y^5 + 720Y^6 \\
    &\quad + (870 + 8918Y + 37163Y^2 + 80045Y^3 + 93065Y^4 + 54677Y^5 \\
    &\quad + 12542Y^6)T + (13184 + 100786Y + 309687Y^2 + 486220Y^3 \\
    &\quad + 408786Y^4 + 174034Y^5 + 29383Y^6)T^2 + (29383 + 174034Y \\
    &\quad + 408786Y^2 + 486220Y^3 + 309687Y^4 + 100786Y^5 + 13184Y^6)T^3 \\
    &\quad + (12542 + 54677Y + 93065Y^2 + 80045Y^3 + 37163Y^4 + 8918Y^5 \\
    &\quad + 870Y^6)T^4 + (720 + 1764Y + 1624Y^2 + 735Y^3 + 175Y^4 \\
    &\quad + 21Y^5 + Y^6)T^5, \\
    \mathcal{N}_{\mathsf{A}_7}(Y, T) &= 1 + 28Y + 322Y^2 + 1960Y^3 + 6769Y^4 + 13132Y^5 + 13068Y^6 \\
    &\quad + 5040Y^7 + (4132 + 52669Y + 282471Y^2 + 823704Y^3 \\
    &\quad + 1403598Y^4 + 1387281Y^5 + 728999Y^6 + 155546Y^7)T \\
    &\quad + (134802 + 1301092Y + 5254088Y^2 + 11457173Y^3 \\
    &\quad + 14497658Y^4 + 10592498Y^5 + 4124012Y^6 + 659797Y^7)T^2 \\
    &\quad + (628282 + 4893332Y + 15809808Y^2 + 27375138Y^3 \\
    &\quad + 27375138Y^4 + 15809808Y^5 + 4893332Y^6 + 628282Y^7)T^3 \\
    &\quad + (659797 + 4124012Y + 10592498Y^2 + 14497658Y^3 \\
    &\quad + 11457173Y^4 + 5254088Y^5 + 1301092Y^6 + 134802Y^7)T^4 \\ 
    &\quad + (155546 + 728999Y + 1387281Y^2 + 1403598Y^3 + 823704Y^4 \\
    &\quad + 282471Y^5 + 52669Y^6 + 4132Y^7)T^5 + (5040 + 13068Y \\
    &\quad + 13132Y^2 + 6769Y^3 + 1960Y^4 + 322Y^5 + 28Y^6 + Y^7)T^6 .
\end{align*} 

\subsubsection{Type $\mathsf{B}$}
\begin{align*} 
    \mathcal{N}_{\mathsf{B}_2}(Y, T) &= 1 + 4Y + 3Y^2 + (3 + 4Y + Y^2)T, \\
    \mathcal{N}_{\mathsf{B}_3}(Y, T) &= 1 + 9Y + 23Y^2 + 15Y^3 + (20 + 76Y + 76Y^2 + 20Y^3)T  \\
    &\quad + (15 + 23Y + 9Y^2 + Y^3)T^2, \\
    \mathcal{N}_{\mathsf{B}_4}(Y, T) &= 1 + 16Y + 86Y^2 + 176Y^3 + 105Y^4 \\
    &\quad + (111 + 736Y + 1642Y^2 + 1376Y^3 + 359Y^4)T \\
    &\quad + (359 + 1376Y + 1642Y^2 + 736Y^3 + 111Y^4)T^2 \\
    &\quad + (105 + 176Y + 86Y^2 + 16Y^3 + Y^4)T^3 \\
    \mathcal{N}_{\mathsf{B}_5}(Y, T) &= 1 + 25Y + 230Y^2 + 950Y^3 + 1689Y^4 + 945Y^5 \\
    &\quad + (642 + 6146Y + 22220Y^2 + 36940Y^3 + 27058Y^4 + 6834Y^5)T \\ 
    &\quad + (5978 + 37082Y + 83660Y^2 + 83660Y^3 + 37082Y^4 + 5978Y^5)T^2 \\
    &\quad + (6834 + 27058Y + 36940Y^2 + 22220Y^3 + 6146Y^4 + 642Y^5)T^3 \\ 
    &\quad + (945 + 1689Y + 950Y^2 + 230Y^3 + 25Y^4 + Y^5)T^4, \\
    \mathcal{N}_{\mathsf{B}_6}(Y, T) &= 1 + 36Y + 505Y^2 + 3480Y^3 + 12139Y^4 + 19524Y^5 + 10395Y^6 \\
    &\quad + (4081 + 51460Y + 260329Y^2 + 669400Y^3 + 905659Y^4 \\
    &\quad + 592420Y^5 + 143211Y^6)T + (92476 + 793400Y + 2682964Y^2 \\
    &\quad + 4511120Y^3 + 3914404Y^4 + 1653560Y^5 + 268236Y^6)T^2 \\
    &\quad + (268236 + 1653560Y + 3914404Y^2 + 4511120Y^3 + 2682964Y^4 \\
    &\quad + 793400Y^5 + 92476Y^6)T^3 + (143211 + 592420Y + 905659Y^2 \\
    &\quad + 669400Y^3 + 260329Y^4 + 51460Y^5 + 4081Y^6)T^4 + (10395 \\
    &\quad + 19524Y + 12139Y^2 + 3480Y^3 + 505Y^4 + 36Y^5 + Y^6)T^5, \\
    \mathcal{N}_{\mathsf{B}_7}(Y, T) &= 1 + 49Y + 973Y^2 + 10045Y^3 + 57379Y^4 + 177331Y^5 \\
    &\quad + 264207Y^6 + 135135Y^7 + (28632 + 452376Y + 2970744Y^2 \\
    &\quad + 10468920Y^3 + 21232008Y^4 + 24456264Y^5 + 14475816Y^6 \\
    &\quad + 3329640Y^7)T + (1456493 + 15959421Y + 72251753Y^2 \\
    &\quad + 173850425Y^3 + 237761447Y^4 + 182794199Y^5 + 72699267Y^6 \\
    &\quad + 11564915Y^7)T^2 + (8886784 + 73960064Y + 250782336Y^2 \\
    &\quad + 445675776Y^3 + 445675776Y^4 + 250782336Y^5 + 73960064Y^6 \\
    &\quad + 8886784Y^7)T^3 + (11564915 + 72699267Y + 182794199Y^2 \\
    &\quad + 237761447Y^3 + 173850425Y^4 + 72251753Y^5 + 15959421Y^6 \\
    &\quad + 1456493Y^7)T^4 + (3329640 + 14475816Y + 24456264Y^2 \\
    &\quad + 21232008Y^3 + 10468920Y^4 + 2970744Y^5 + 452376Y^6 \\
    &\quad + 28632Y^7)T^5 + (135135 + 264207Y + 177331Y^2 + 57379Y^3 \\
    &\quad + 10045Y^4 + 973Y^5 + 49Y^6 + Y^7)T^6.
\end{align*}

\subsubsection{Type $\mathsf{D}$}
\begin{align*}
    \mathcal{N}_{\mathsf{D}_4}(Y, T) &= 1 + 12Y + 50Y^2 + 84Y^3 + 45Y^4 \\
    &\quad + (67 + 396Y + 814Y^2 + 660Y^3 + 175Y^4)T \\
    &\quad +(175 + 660Y + 814Y^2 + 396Y^3 + 67Y^4)T^2 \\
    &\quad + (45 + 84Y + 50Y^2 + 12Y^3 + Y^4)T^3, \\
    \mathcal{N}_{\mathsf{D}_5}(Y, T) &= 1 + 20Y + 150Y^2 + 520Y^3 + 809Y^4 + 420Y^5 \\
    &\quad + (397 + 3471Y + 11630Y^2 + 18250Y^3 + 12933Y^4 + 3239Y^5)T \\ 
    &\quad + (3143 + 18767Y + 41450Y^2 + 41450Y^3 + 18767Y^4 + 3143Y^5)T^2 \\ 
    &\quad + (3239 + 12933Y + 18250Y^2 + 11630Y^3 + 3471Y^4 + 397Y^5)T^3 \\ 
    &\quad + (420 + 809Y + 520Y^2 + 150Y^3 + 20Y^4 + Y^5)T^4, \\
    \mathcal{N}_{\mathsf{D}_6}(Y, T) &= 1 + 30Y + 355Y^2 + 2100Y^3 + 6439Y^4 + 9390Y^5 + 4725Y^6 \\
    &\quad + (2539 + 29746Y + 141097Y^2 + 343660Y^3 + 445501Y^4 \\
    &\quad + 283234Y^5 + 67503Y^6)T + (50272 + 415232Y + 1363120Y^2 \\
    &\quad + 2246240Y^3 + 1931488Y^4 + 817568Y^5 + 134160Y^6)T^2 \\
    &\quad + (134160 + 817568Y + 1931488Y^2 + 2246240Y^3 + 1363120Y^4 \\
    &\quad + 415232Y^5 + 50272Y^6)T^3 + (67503 + 283234Y + 445501Y^2 \\
    &\quad + 343660Y^3 + 141097Y^4 + 29746Y^5 + 2539Y^6)T^4 + (4725 \\
    &\quad + 9390Y + 6439Y^2 + 2100Y^3 + 355Y^4 + 30Y^5 + Y^6)T^5, \\
    \mathcal{N}_{\mathsf{D}_7}(Y, T) &= 1 + 42Y + 721Y^2 + 6510Y^3 + 33019Y^4 + 92358Y^5 + 127539Y^6 \\
    &\quad + 62370Y^7 + (17859 + 264979Y + 1645791Y^2 + 5526031Y^3 \\ 
    &\quad + 10759161Y^4 + 11991721Y^5 + 6930789Y^6 + 1570869Y^7)T \\ 
    &\quad + (804618 + 8519625Y + 37458806Y^2 + 88014129Y^3 \\
    &\quad + 118241942Y^4 + 89884515Y^5 + 35579114Y^6 + 5666211Y^7)T^2 \\
    &\quad + (4578872 + 37483512Y + 125642888Y^2 + 221947208Y^3 \\
    &\quad + 221947208Y^4 + 125642888Y^5 + 37483512Y^6 + 4578872Y^7)T^3 \\
    &\quad + (5666211 + 35579114Y + 89884515Y^2 + 118241942Y^3 \\
    &\quad + 88014129Y^4 + 37458806Y^5 + 8519625Y^6 + 804618Y^7)T^4 \\
    &\quad + (1570869 + 6930789Y + 11991721Y^2 + 10759161Y^3 \\
    &\quad + 5526031Y^4 + 1645791Y^5 + 264979Y^6 + 17859Y^7)T^5 \\
    &\quad + (62370 + 127539Y + 92358Y^2 + 33019Y^3 + 6510Y^4 \\
    &\quad + 721Y^5 + 42Y^6 + Y^7)T^6.
\end{align*}

\subsubsection{Exceptional types}
\label{subsubsec:coarse.rank.leq.7}
\begin{align*}
    \mathcal{N}_{\mathsf{E}_6}(Y, T) &= 1 + 36Y + 510Y^2 + 3600Y^3 + 13089Y^4 + 22284Y^5 + 12320Y^6 \\
    &\quad + (4591 + 57420Y + 289824Y^2 + 748080Y^3 + 1020819Y^4 \\
    &\quad + 671940Y^5 + 162206Y^6)T + (103681 + 888840Y + 3011919Y^2 \\
    &\quad + 5080320Y^3 + 4411839Y^4 + 1858680Y^5 + 300401Y^6)T^2 \\
    &\quad + (300401 + 1858680Y + 4411839Y^2 + 5080320Y^3 + 3011919Y^4 \\
    &\quad + 888840Y^5 + 103681Y^6)T^3 + (162206 + 671940Y + 1020819Y^2 \\
    &\quad + 748080Y^3 + 289824Y^4 + 57420Y^5 + 4591Y^6)T^4 + (12320 \\
    &\quad + 22284Y + 13089Y^2 + 3600Y^3 + 510Y^4 + 36Y^5 + Y^6)T^5, \\
    \mathcal{N}_{\mathsf{E}_7}(Y, T) &= 1 + 63Y + 1617Y^2 + 21735Y^3 + 162939Y^4 + 663957Y^5 \\
    &\quad + 1286963Y^6 + 765765Y^7 + (90400 + 1553980Y + 11064984Y^2 \\
    &\quad + 42142884Y^3 + 92109360Y^4 + 113759940Y^5 + 70917656Y^6 \\
    &\quad + 16725596Y^7)T + (5577043 + 64210477Y + 304475955Y^2 \\
    &\quad + 763724661Y^3 + 1080226497Y^4 + 847444143Y^5 + 338480825Y^6 \\
    &\quad + 53381039Y^7)T^2 + (37767356 + 323700436Y + 1123040604Y^2 \\
    &\quad + 2022363924Y^3 + 2022363924Y^4 + 1123040604Y^5 \\
    &\quad + 323700436Y^6 + 37767356Y^7)T^3 + (53381039 + 338480825Y \\
    &\quad + 847444143Y^2 + 1080226497Y^3 + 763724661Y^4 + 304475955Y^5 \\
    &\quad + 64210477Y^6 + 5577043Y^7)T^4 + (16725596 + 70917656Y \\
    &\quad + 113759940Y^2 + 92109360Y^3 + 42142884Y^4 + 11064984Y^5 \\
    &\quad + 1553980Y^6 + 90400Y^7)T^5 + (765765 + 1286963Y + 663957Y^2 \\
    &\quad + 162939Y^3 + 21735Y^4 + 1617Y^5 + 63Y^6 + Y^7)T^6, \\
    \mathcal{N}_{\mathsf{F}_4}(Y, T) &= 1 + 24Y + 190Y^2 + 552Y^3 + 385Y^4 \\
    &\quad + (263 + 1992Y + 4994Y^2 + 4344Y^3 + 1079Y^4)T \\
    &\quad + (1079 + 4344Y + 4994Y^2 + 1992Y^3 + 263Y^4)T^2 \\
    &\quad + (385 + 552Y + 190Y^2 + 24Y^3 + Y^4)T^3, \\
    \mathcal{N}_{\mathsf{G}_2}(Y, T) &= 1 + 6Y + 5Y^2 + (5 + 6Y + Y^2)T, \\
    \mathcal{N}_{\mathsf{H}_2}(Y, T) &= 1 + 5Y + 4Y^2 + (4 + 5Y + Y^2)T, \\
    \mathcal{N}_{\mathsf{H}_3}(Y, T) &= 1 + 15Y + 59Y^2 + 45Y^3 \\
    &\quad + (44 + 196Y + 196Y^2 + 44Y^3)T \\
    &\quad + (45 + 59Y + 15Y^2 + Y^3)T^2, \\
    \mathcal{N}_{\mathsf{H}_4}(Y, T) &= 1 + 60Y + 1138Y^2 + 7140Y^3 + 6061Y^4 \\
    &\quad + (2099 + 20700Y + 63662Y^2 + 58500Y^3 + 13439Y^4)T \\
    &\quad + (13439  + 58500Y + 63662Y^2 + 20700Y^3 + 2099Y^4)T^2 \\
    &\quad + (6061  + 7140Y + 1138Y^2 + 60Y^3 + Y^4)T^3, \\
    \mathcal{N}_{\mathsf{I}_2(m)}(Y, T) &= 1 + mY + (m-1)Y^2 + (m-1 + mY + Y^2)T.
\end{align*}
(For type $\mathsf{I}_2(m)$, see also Proposition~\ref{pro:Im}.)

\subsection{Restrictions of type-$\mathsf{D}$ arrangements}
\label{sec:app.coarse.res-D}

Recall Definition~\ref{defn:D-restriction} of the restrictions of type-$\mathsf{D}$ arrangements: for $n\geq 1$ and $m\in [n]_0$,
\begin{align*} 
    \mathcal{A}_{n,m} := \mathsf{D}_n\cup \{X_k ~|~ k\in [n-m]\}. 
\end{align*} 
We record the non-Coxeter restrictions up to rank four.
\begin{align*} 
    \mathcal{N}_{\mathcal{A}_{3, 2}}(Y, T) &= 1 + 7Y + 15Y^2 + 9Y^3 + (14 + 50Y + 50Y^2 + 14Y^3)T \\
    &\quad + (9 + 15Y + 7Y^2 + Y^3)T^2, \\
    \mathcal{N}_{\mathcal{A}_{3, 1}}(Y, T) &= 1 + 8Y + 19Y^2 + 12Y^3 + (17 + 63Y + 63Y^2 + 17Y^3)T \\
    &\quad + (12 + 19Y + 8Y^2 + Y^3)T^2, \\
    \mathcal{N}_{\mathcal{A}_{4, 3}}(Y, T) &= 1 + 13Y + 59Y^2 + 107Y^3 + 60Y^4 \\
    &\quad + (78 + 481Y + 1021Y^2 + 839Y^3 + 221Y^4)T \\
    &\quad + (221 + 839Y + 1021Y^2 + 481Y^3 + 78Y^4)T^2 \\
    &\quad + (60 + 107Y + 59Y^2 + 13Y^3 + Y^4)T^3, \\
    \mathcal{N}_{\mathcal{A}_{4, 2}}(Y, T) &= 1 + 14Y + 68Y^2 + 130Y^3 + 75Y^4 \\
    &\quad + (89 + 566Y + 1228Y^2 + 1018Y^3 + 267Y^4)T \\
    &\quad + (267 + 1018Y + 1228Y^2 + 566Y^3 + 89Y^4)T^2 \\
    &\quad + (75 + 130Y + 68Y^2 + 14Y^3 + Y^4)T^3, \\ 
    \mathcal{N}_{\mathcal{A}_{4, 1}}(Y, T) &= 1 + 15Y + 77Y^2 + 153Y^3 + 90Y^4 \\
    &\quad + (100 + 651Y + 1435Y^2 + 1197Y^3 + 313Y^4)T \\
    &\quad + (313 + 1197Y + 1435Y^2 + 651Y^3 + 100Y^4)T^2 \\
    &\quad + (90 + 153Y + 77Y^2 + 15Y^3 + Y^4)T^3.
\end{align*} 

\subsection{Shi arrangements}\label{subsec:shi.coarse}
The \emph{Shi arrangement} of type $\mathsf{X}_n$ is the (noncentral)
hyperplane arrangement
\begin{align*} 
    \mathcal{S}\mathsf{X}_n &= \mathsf{X}_n \cup \{L - 1 ~|~ L\in \mathsf{X}_n \}.
\end{align*} 

\subsubsection{Type $\mathsf{A}$}
\begin{align*} 
    \mathcal{N}_{\mathcal{S}\mathsf{A}_1}(Y, T) &= 1 + 2Y + T, \\
    \mathcal{N}_{\mathcal{S}\mathsf{A}_2}(Y, T) &= 1 + 6Y + 9Y^2 + (10 + 24Y + 6Y^2)T + 4T^2 \\
    \mathcal{N}_{\mathcal{S}\mathsf{A}_3}(Y, T) &= 1 + 12Y + 48Y^2 + 64Y^3 + (69 + 352Y + 524Y^2 + 160Y^3)T \\
    &\quad + (151 + 380Y + 172Y^2 + 24Y^3)T^2 + 27T^3 \\
    \mathcal{N}_{\mathcal{S}\mathsf{A}_4}(Y, T) &= 1 + 20Y + 150Y^2 + 500Y^3 + 625Y^4 \\
    &\quad + (496 + 3905Y + 11045Y^2 + 12615Y^3 + 3955Y^4)T \\
    &\quad + (3526 + 17220Y + 27140Y^2 + 14420Y^3 + 2510Y^4)T^2 \\
    &\quad + (2931 + 7695Y + 4925Y^2 + 1305Y^3 + 120Y^4)T^3  + 256T^4.
\end{align*} 
(For $\mathcal{S}\mathsf{A}_2$, see also Proposition~\ref{pro:shi.A3}.)

\subsubsection{Type $\mathsf{B}$}
\begin{align*} 
    \mathcal{N}_{\mathcal{S}\mathsf{B}_2}(Y, T) &= 1 + 8Y + 16Y^2 + (16 + 44Y + 10Y^2)T + 9T^2 \\
    \mathcal{N}_{\mathcal{S}\mathsf{B}_3}(Y, T) &= 1 + 18Y + 108Y^2 + 216Y^3 + (165 + 982Y + 1694Y^2 + 502Y^3)T \\
    &\quad + (499 + 1370Y + 568Y^2 + 72Y^3)T^2 + 125T^3 \\
    \mathcal{N}_{\mathcal{S}\mathsf{B}_4}(Y, T) &= 1 + 32Y + 384Y^2 + 2048Y^3 + 4096Y^4 \\
    &\quad + (1912 + 17532Y + 57528Y^2 + 75596Y^3 + 24084Y^4)T \\ 
    &\quad + (18806 + 100520Y + 169440Y^2 + 87848Y^3 + 14528Y^4)T^2 \\
    &\quad + (20260 + 55436Y + 32928Y^2 + 8028Y^3 + 672Y^4)T^3 + 2401T^4 .
\end{align*} 

\subsubsection{Type $\mathsf{D}$}
\begin{align*} 
    \mathcal{N}_{\mathcal{S}\mathsf{D}_4}(Y, T) &= 1 + 24Y + 216Y^2 + 864Y^3 + 1296Y^4 \\
    &\quad + (788 + 6612Y + 19960Y^2 + 24236Y^3 + 7600Y^4)T \\
    &\quad + (6350 + 32320Y + 52456Y^2 + 27352Y^3 + 4616Y^4)T^2 \\
    &\quad + (5964 + 15956Y + 9736Y^2 + 2460Y^3 + 216Y^4)T^3 + 625T^4. 
\end{align*}

\subsection{Catalan arrangements}

The \emph{Catalan arrangement} of type $\mathsf{X}_n$ is the (noncentral)
hyperplane arrangement
\begin{align*} 
    \mathcal{C}\mathsf{X}_n &= \{L + \epsilon ~|~ L\in \mathsf{X}_n,\; \epsilon\in\{-1,0,1\} \}.
\end{align*}

\subsubsection{Type $\mathsf{A}$}
\begin{align*} 
    \mathcal{N}_{\mathcal{C}\mathsf{A}_1}(Y, T) &= 1 + 3Y + 2T, \\
    \mathcal{N}_{\mathcal{C}\mathsf{A}_2}(Y, T) &= 1 + 9Y + 20Y^2 + (20 + 57Y + 13Y^2)T + 12T^2, \\
    \mathcal{N}_{\mathcal{C}\mathsf{A}_3}(Y, T) &= 1 + 18Y + 107Y^2 + 210Y^3 \\ 
    &\quad + (169 + 999Y + 1703Y^2 + 513Y^3)T \\
    &\quad + (508 + 1377Y + 584Y^2 + 75Y^3)T^2 + 120T^3, \\
    \mathcal{N}_{\mathcal{C}\mathsf{A}_4}(Y, T) &= 1 + 30Y + 335Y^2 + 1650Y^3 + 3024Y^4 \\
    &\quad + (1597 + 14300Y + 45705Y^2 + 58420Y^3 + 18698Y^4)T \\
    &\quad + (15138 + 79380Y + 131865Y^2 + 69180Y^3 + 11637Y^4)T^2 \\
    &\quad + (15484 + 41890Y + 25495Y^2 + 6350Y^3 + 541Y^4)T^3 + 1680T^4.
\end{align*} 

\subsubsection{Type $\mathsf{B}$}
\begin{align*} 
    \mathcal{N}_{\mathcal{C}\mathsf{B}_2}(Y, T) &= 1 + 12Y + 35Y^2 + (31 + 100Y + 21Y^2)T + 24T^2, \\ 
    \mathcal{N}_{\mathcal{C}\mathsf{B}_3}(Y, T) &= 1 + 27Y + 239Y^2 + 693Y^3 \\
    &\quad + (408 + 2756Y + 5352Y^2 + 1564Y^3)T \\
    &\quad + (1583 + 4633Y + 1825Y^2 + 215Y^3)T^2 + 480T^3, \\
    \mathcal{N}_{\mathcal{C}\mathsf{B}_4}(Y, T) &= 1 + 48Y + 854Y^2 + 6672Y^3 + 19305Y^4 \\
    &\quad + (6399 + 65600Y + 237770Y^2 + 342784Y^3 + 110455Y^4)T \\
    &\quad + (78975 + 447344Y + 787946Y^2 + 403792Y^3 + 64855Y^4)T^2 \\
    &\quad + (98657 + 276896Y + 158262Y^2 + 36640Y^3 + 2857Y^4)T^3 \\
    &\quad + 13440T^4.
\end{align*} 

\subsubsection{Type $\mathsf{D}$}
\begin{align*} 
    \mathcal{N}_{\mathcal{C}\mathsf{D}_4}(Y, T) &= 1 + 36Y + 482Y^2 + 2844Y^3 + 6237Y^4 \\
    &\quad + (2611 + 24636Y + 83158Y^2 + 112260Y^3 + 35767Y^4)T \\
    &\quad + (27247 + 148140Y + 253006Y^2 + 130452Y^3 + 21379Y^4)T^2 \\
    &\quad + (30669 + 84660Y + 49562Y^2 + 11916Y^3 + 985Y^4)T^3 + 3840T^4.
\end{align*} 

\subsection{Generic central arrangements}\label{subsec:gen.cen.arr}
Recall the generic central arrangements $\mcU_{n,m}$ for $n\leq m$ introduced
in Section~\ref{subsec:gen.central.arr}. We exemplify the general formula
given in Proposition~\ref{prop:gen-cen.coarse} with those for $m=n+1$
and~$n\in\{3,4,5,6\}$. Note that $\mcU_{n,n+1}$ may be obtained by adding to
the Boolean arrangement $\mathsf{A}_1^n$ (cf.\ Section~\ref{subsec:Boolean})
one additional hyperplane in general position; the one defined by the sum of
coordinates will do. We also record formulae for $(n,m)\in\{(4,7),(4,8)\}$;
see Section~\ref{subsec:generic-cen.coarse}.
\begin{align*}
    \mcN_{\mcU_{3,4}}(Y, T) &= 1 + 4Y + 6Y^2 + 3Y^3 + (8 + 26Y + 26Y^2 + 8Y^3)T \\
    &\quad + (3 + 6Y + 4Y^2 + Y^3)T^2, \\
    \mcN_{\mcU_{4,5}}(Y, T) &= 1 + 5Y + 10Y^2 + 10Y^3 + 4Y^4 \\
    &\quad + (22 + 100Y + 170Y^2 + 125Y^3 + 33Y^4)T \\ 
    &\quad + (33 + 125Y + 170Y^2 + 100Y^3 + 22Y^4)T^2 \\
    &\quad + (4 + 10Y + 10Y^2 + 5Y^3 + Y^4)T^3, \\
    \mathcal{N}_{\mathcal{U}_{4,7}}(Y,T) &= 1 + 7Y + 21Y^2 + 35Y^3 + 20Y^4 \\
    &\quad + (60 + 315Y + 609Y^2 + 483Y^3 + 129Y^4)T \\
    &\quad + (129 + 483Y + 609Y^2 + 315Y^3 + 60Y^4)T^2 \\
    &\quad + (20 + 35Y + 21Y^2 + 7Y^3 + Y^4)T^3, \\
    \mathcal{N}_{\mathcal{U}_{4,8}}(Y,T) &= 1 + 8Y + 28Y^2 + 56Y^3 + 35Y^4 \\
    &\quad + (89 + 488Y + 980Y^2 + 792Y^3 + 211Y^4)T \\
    &\quad + (211 + 792Y + 980Y^2 + 488Y^3 + 89Y^4)T^2 \\
    &\quad + (35 + 56Y + 28Y^2 + 8Y^3 + Y^4)T^3, \\
    \mcN_{\mcU_{5,6}}(Y, T) &= 1 + 6Y + 15Y^2 + 20Y^3 + 15Y^4 + 5Y^5 \\
    &\quad + (52 + 297Y + 685Y^2 + 795Y^3 + 459Y^4 + 104Y^5)T \\ 
    &\quad + (198 + 1023Y + 2085Y^2 + 2085Y^3 + 1023Y^4 + 198Y^5)T^2 \\
    &\quad + (104 + 459Y + 795Y^2 + 685Y^3 + 297Y^4 + 52Y^5)T^3 \\
    &\quad + (5 + 15Y + 20Y^2 + 15Y^3 + 6Y^4 + Y^5)T^4, \\
    \mcN_{\mcU_{6,7}}(Y, T) &= 1 + 7Y + 21Y^2 + 35Y^3 + 35Y^4 + 21Y^5 + 6Y^6 + (114 + 777Y \\
    &\quad + 2233Y^2 + 3465Y^3 + 3059Y^4 + 1449Y^5 + 285Y^6)T + (906 \\
    &\quad + 5796Y + 15400Y^2 + 21700Y^3 + 17052Y^4 + 7070Y^5 + 1208Y^6)T^2 \\
    &\quad + (1208 + 7070Y + 17052Y^2 + 21700Y^3 + 15400Y^4 + 5796Y^5 \\
    &\quad + 906Y^6)T^3 + (285 + 1449Y + 3059Y^2 + 3465Y^3 + 2233Y^4 \\
    &\quad + 777Y^5 + 114Y^6)T^4 + (6 + 21Y + 35Y^2 + 35Y^3 + 21Y^4 \\
    &\quad + 7Y^5 + Y^6)T^5.
\end{align*}

\subsection{Resonance arrangements}

For $n\in\N$, the \emph{resonance arrangement} is
\begin{align*} 
  \mathcal{R}_n &= \left\{\sum_{i\in I}X_i ~\middle|~ \emptyset\neq I\subseteq [n] \right\}. 
\end{align*} 
For $n\leq 2$, $\mathcal{R}_n\cong \mathsf{A}_n$, given in
Section~\ref{subsubsec:type-A}, and $\mathcal{R}_3\cong \mathcal{A}_{3,2}$,
given in Section~\ref{sec:app.coarse.res-D}.
\begin{align*} 
  \mcN_{\mathcal{R}_4}(Y, T) &= 1 + 15Y + 80Y^2 + 170Y^3 + 104Y^4 \\
  &\quad + (112 + 730Y + 1630Y^2 + 1365Y^3 + 353Y^4)T \\
  &\quad + (353 + 1365Y + 1630Y^2 + 730Y^3 + 112Y^4)T^2 \\
  &\quad + (104 + 170Y + 80Y^2 + 15Y^3 + Y^4)T^3, \\
  \mcN_{\mathcal{R}_5}(Y, T) &= 1 + 31Y + 375Y^2 + 2130Y^3 + 5270Y^4 + 3485Y^5 + (1782 + 17817Y \\
  &\quad + 68375Y^2 + 121745Y^3 + 92659Y^4 + 23254Y^5)T + (18818 \\
  &\quad + 120923Y + 280775Y^2 + 280775Y^3 + 120923Y^4 + 18818Y^5)T^2 \\
  &\quad + (23254 + 92659Y + 121745Y^2 + 68375Y^3 + 17817Y^4 + 1782Y^5)T^3 \\
  &\quad + (3485 + 5270Y + 2130Y^2 + 375Y^3 + 31Y^4 + Y^5)T^4, \\
  \mcN_{\mathcal{R}_6}(Y, T) &= 1 + 63Y + 1652Y^2 + 22435Y^3 + 159460Y^4 + 510524Y^5 + 371909Y^6 \\
  &\quad + (77254 + 1088220Y + 6136361Y^2 + 17666495Y^3 + 26863403Y^4 \\
  &\quad + 19032139Y^5 + 4709836Y^6)T + (2362293 + 21610148Y \\
  &\quad + 77625345Y^2 + 137120970Y^3 + 121428629Y^4 + 50644846Y^5 \\
  &\quad + 7959697Y^6)T^2 + (7959697 + 50644846Y + 121428629Y^2 \\
  &\quad + 137120970Y^3 + 77625345Y^4 + 21610148Y^5 + 2362293Y^6)T^3 \\
  &\quad + (4709836 + 19032139Y + 26863403Y^2 + 17666495Y^3 + 6136361Y^4 \\
  &\quad + 1088220Y^5 + 77254Y^6)T^4 + (371909 + 510524Y + 159460Y^2 \\
  &\quad + 22435Y^3 + 1652Y^4 + 63Y^5 + Y^6)T^5.
\end{align*}

\subsection{$2$-sum arrangements}

For $n\geq 3$, the $2$-\emph{sum arrangement} is the central hyperplane
arrangement
\begin{align*}
    \mathcal{S}_{n} &= \mathsf{A}_n \cup \{X_i + X_j - X_k - X_\ell ~|~ \text{distinct } i,j,k,\ell\in[n+1]\}.
\end{align*} 
\begin{align*}
    \mathcal{N}_{\mathcal{S}_{3}}(Y, T) &= 1 + 9Y + 23Y^2 + 15Y^3
    \\ &\quad + (20 + 76Y + 76Y^2 + 20Y^3)T \\ &\quad + (15 + 23Y +
    9Y^2 + Y^3)T^2, \\ 
    \mathcal{N}_{\mathcal{S}_{4}}(Y, T) &= 1 + 25Y
    + 215Y^2 + 695Y^3 + 504Y^4 \\ &\quad + (342 + 2605Y + 6625Y^2 +
    5795Y^3 + 1433Y^4)T \\ &\quad + (1433 + 5795Y + 6625Y^2 + 2605Y^3
    + 342Y^4)T^2 \\ &\quad + (504 + 695Y + 215Y^2 + 25Y^3 + Y^4)T^3, \\
    \mcN_{\mathcal{S}_5}(Y, T) &= 1 + 60Y + 1360Y^2 + 14010Y^3 + 59119Y^4 + 46410Y^5 + (12332 \\
    &\quad + 147046Y + 656060Y^2 + 1328440Y^3 + 1075448Y^4 + 268354Y^5)T \\
    &\quad + (179963 + 1253627Y + 3070730Y^2 + 3070730Y^3 + 1253627Y^4 \\
    &\quad + 179963Y^5)T^2 + (268354 + 1075448Y + 1328440Y^2 + 656060Y^3 \\
    &\quad + 147046Y^4 + 12332Y^5)T^3 + (46410 + 59119Y + 14010Y^2 \\
    &\quad + 1360Y^3 + 60Y^4 + Y^5)T^4.
\end{align*}

\end{appendices}

\end{document}